\documentclass{amsart}

\usepackage{amsmath,amssymb,amsthm,enumerate} 
\usepackage{amsfonts} 
\usepackage[top=30mm,bottom=30mm,left=25mm,right=25mm]{geometry}
\usepackage[dvipdfmx]{graphicx}
\usepackage{color}

\theoremstyle{plain}
\newtheorem{pretheo}{Theorem}[section]
\newtheorem{preassu}[pretheo]{Assumption}
\newtheorem{precoro}[pretheo]{Corollary}
\newtheorem{predefi}[pretheo]{Definition}
\newtheorem{preexam}[pretheo]{Example}
\newtheorem{prelemm}[pretheo]{Lemma}
\newtheorem{preprop}[pretheo]{Proposition}
\newtheorem{prerema}[pretheo]{Remark}

\newenvironment{theo}{\begin{pretheo}}{\end{pretheo}}

\newenvironment{defi}{\begin{predefi}}{\end{predefi}}

\newenvironment{lemm}{\begin{prelemm}}{\end{prelemm}}
\newenvironment{prop}{\begin{preprop}}{\end{preprop}}
\newenvironment{rema}{\begin{prerema}\rm}{\end{prerema}}

\DeclareMathOperator{\di}{div}

\DeclareMathOperator{\Di}{Div}

\newcommand{\intd}{\,d}
\newcommand{\supp}{{\rm supp}\,}
\newcommand{\jump}[1]{\ensuremath{[\![#1]\!]}}

\newcommand{\loc}{{\rm loc}}
\newcommand{\pa}{\partial}

\newcommand{\wh}[1]{\widehat{#1}}
\newcommand{\wt}[1]{\widetilde{#1}}

\newcommand{\Hol}{{\rm Hol}\,}

\newcommand{\BBM}{\mathbb{M}}

\newcommand{\Ba}{\mathbf{a}}
\newcommand{\Bb}{\mathbf{b}}

\newcommand{\Bf}{\mathbf{f}}
\newcommand{\Bg}{\mathbf{g}}
\newcommand{\Bh}{\mathbf{h}}

\newcommand{\Bk}{\mathbf{k}}

\newcommand{\Bm}{\mathbf{m}}
\newcommand{\Bn}{\mathbf{n}}

\newcommand{\Bu}{\mathbf{u}}
\newcommand{\Bv}{\mathbf{v}}
\newcommand{\Bw}{\mathbf{w}}

\newcommand{\BA}{\mathbf{A}}
\newcommand{\BB}{\mathbf{B}}
\newcommand{\BC}{\mathbf{C}}
\newcommand{\BD}{\mathbf{D}}
\newcommand{\BE}{\mathbf{E}}
\newcommand{\BF}{\mathbf{F}}
\newcommand{\BG}{\mathbf{G}}
\newcommand{\BH}{\mathbf{H}}
\newcommand{\BI}{\mathbf{I}}
\newcommand{\BJ}{\mathbf{J}}
\newcommand{\BK}{\mathbf{K}}
\newcommand{\BL}{\mathbf{L}}
\newcommand{\BM}{\mathbf{M}}
\newcommand{\BN}{\mathbf{N}}

\newcommand{\BP}{\mathbf{P}}

\newcommand{\BR}{\mathbf{R}}
\newcommand{\BS}{\mathbf{S}}
\newcommand{\BT}{\mathbf{T}}
\newcommand{\BU}{\mathbf{U}}
\newcommand{\BV}{\mathbf{V}}
\newcommand{\BW}{\mathbf{W}}

\newcommand{\CA}{\mathcal{A}}
\newcommand{\CB}{\mathcal{B}}
\newcommand{\CC}{\mathcal{C}}
\newcommand{\CD}{\mathcal{D}}
\newcommand{\CE}{\mathcal{E}}
\newcommand{\CF}{\mathcal{F}}
\newcommand{\CG}{\mathcal{G}}
\newcommand{\CH}{\mathcal{H}}
\newcommand{\CI}{\mathcal{I}}
\newcommand{\CJ}{\mathcal{J}}
\newcommand{\CK}{\mathcal{K}}
\newcommand{\CL}{\mathcal{L}}
\newcommand{\CM}{\mathcal{M}}
\newcommand{\CN}{\mathcal{N}}

\newcommand{\CP}{\mathcal{P}}

\newcommand{\CR}{\mathcal{R}}
\newcommand{\CS}{\mathcal{S}}
\newcommand{\CT}{\mathcal{T}}
\newcommand{\CU}{\mathcal{U}}
\newcommand{\CV}{\mathcal{V}}
\newcommand{\CW}{\mathcal{W}}
\newcommand{\CX}{\mathcal{X}}
\newcommand{\CY}{\mathcal{Y}}
\newcommand{\CZ}{\mathcal{Z}}

\newcommand{\bdry}{{\BR_0^N}}
\newcommand{\dws}{{\dot{\BR}^N}}
\newcommand{\lhs}{{\BR_-^N}}
\newcommand{\uhs}{{\BR_+^N}}
\newcommand{\tws}{{\BR^{N-1}}}
\newcommand{\ws}{{\BR^N}}

\newcommand{\al}{\alpha}

\newcommand{\ga}{\gamma}
\newcommand{\de}{\delta}
\newcommand{\ep}{\varepsilon}
\newcommand{\te}{\theta}
\newcommand{\ka}{\kappa}
\newcommand{\la}{\lambda}
\newcommand{\si}{\sigma}
\newcommand{\ph}{\varphi}
\newcommand{\om}{\omega}

\newcommand{\Ga}{\Gamma}
\newcommand{\De}{\Delta}

\newcommand{\La}{\Lambda}
\newcommand{\Si}{\Sigma}
\newcommand{\Om}{\Omega}

\numberwithin{equation}{section}

\begin{document}
\title[Two-phase Stokes resolvent equations]
{On the $\CR$-boundedness of solution operator families for two-phase Stokes resolvent equations}


\author[Sri Maryani]{Sri Maryani}
\address{Faculty of Mathematics and Natural Science, Department of Mathematics, Jenderal Soedirman University, Indonesia}
\email{sri.maryani@unsoed.ac.id}

\author[Hirokazu Saito]{Hirokazu Saito}
\address{Department of Pure and Applied Mathematics, Graduate School of Fundamental Science and Engineering,
Waseda University, 3-4-1 Ookubo, Shinjuku-ku, Tokyo, 169-8555, Japan}
\email{hsaito@aoni.waseda.jp}






\begin{abstract}
The aim of this paper is to show the existence of $\CR$-bounded solution operator families for two-phase Stokes resolvent equations in $\dot\Om=\Om_+\cup\Om_-$,
where $\Om_\pm$ are uniform $W_r^{2-1/r}$ domains of $N$-dimensional Euclidean space $\ws$ ($N\geq 2$, $N<r<\infty$).
More precisely, given a uniform $W_r^{2-1/r}$ domain $\Om$ with two boundaries $\Ga_\pm$ satisfying $\Ga_+\cap\Ga_-=\emptyset$,
we suppose that some hypersurface $\Ga$ divides $\Om$ into two sub-domains, that is,
there exist domains $\Om_\pm\subset\Om$ such that
$\Om_+\cap\Om_-=\emptyset$ and $\Om\setminus\Ga=\Om_+\cup\Om_-$,
where $\Ga\cap\Ga_+=\emptyset$, $\Ga\cap\Ga_-=\emptyset$, and the boundaries of $\Om_\pm$ consist of two parts $\Ga$ and $\Ga_\pm$, respectively.
The domains $\Om_\pm$ are filled with viscous, incompressible, and immiscible fluids with density $\rho_\pm$ and viscosity $\mu_\pm$, respectively.
Here $\rho_\pm$ are positive constants, while $\mu_\pm=\mu_\pm(x)$ are functions of $x\in\ws$.
On the boundaries $\Ga$, $\Ga_+$, and $\Ga_-$, we consider an interface condition, a free boundary condition, and the Dirichlet boundary condition, respectively.
We also show, by using the $\CR$-bounded solution operator families, some maximal $L_p\text{-}L_q$ regularity as well as
generation of analytic semigroup for a time-dependent problem associated with the two-phase Stokes resolvent equations.
This kind of problems arises in the mathematical study of the motion of two viscous, incompressible, and immiscible fluids with free surfaces.
The essential assumption of this paper is the unique solvability of
a weak elliptic transmission problem
for $\Bf\in L_q(\Om)^N$,
that is, it is assumed that the unique existence of solutions $\te\in\CW_q^1(\Om)$ to the variational problem:
$(\rho^{-1}\nabla\te,\nabla\ph)_{\dot\Om}=(\Bf,\nabla\ph)_{\Om}$
for any $\ph\in\CW_{q'}^1(\Om)$ with $1<q<\infty$ and $q'=q/(q-1)$,
where $\rho$ is defined by $\rho=\rho_+$ ($x\in\Om_+$), $\rho=\rho_-$ ($x\in\Om_-$)  and
$\CW_q^1(\Om)$ is a suitable Banach space endowed with norm $\|\cdot\|_{\CW_q^1(\Om)}:=\|\nabla\cdot\|_{L_q(\Om)}$.
Our assumption covers e.g. the following domains as $\Om$:
$\ws$, $\BR_\pm^N$, perturbed $\BR_\pm^N$, layers, perturbed layers, and bounded domains, 
where $\BR_+^N$ and $\BR_-^N$ are the open upper and lower half spaces, respectively.
\end{abstract}

\maketitle

\renewcommand{\thefootnote}{\fnsymbol{footnote}}


\section{Introduction}\label{sect:intro}
\subsection{Problem}
Let $\Om$ be a domain of $\ws$, $N \geq 2$, with two boundaries $\Ga_\pm$ satisfying $\Ga_+\cap\Ga_-=\emptyset$.
Suppose that some hypersurface $\Ga$ divides $\Om$ into two sub-domains,
that is, there exist domains $\Om_\pm\subset\Om$ such that $\Om_+\cap\Om_-=\emptyset$ and $\Om\setminus\Ga=\Om_+\cup\Om_-$,
where $\Ga\cap\Ga_+=\emptyset$, $\Ga\cap\Ga_-=\emptyset$, and the boundaries of $\Om_\pm$ consist of two parts $\Ga$ and $\Ga_\pm$, respectively.
Set $\dot\Om=\Om_+\cup\Om_-$ and $\Si_{\ep,\la_0} = \{\,\la\in\BC \mid |\arg\la|\leq\pi-\ep,\,|\la|\geq\la_0\,\}$
for $0<\ep<\pi/2$ and $\la_0>0$.
In this paper, we show the existence of $\CR$-bounded solution operator families for the following 
two-phase Stokes resolvent equations with resolvent parameter $\la$ varying in $\Si_{\ep,\la_0}$:
\begin{equation}\label{TSRP}
	\left\{\begin{aligned}
		\la\Bu -\rho^{-1}\Di\BT(\Bu,\te) &= \Bf, &\di\Bu&=g && \text{in $\dot\Om$,} \\
		\jump{\BT(\Bu,\te)\Bn} &= \jump{\Bh}, &\jump{\Bu}&=0 && \text{on $\Ga$,} \\
		\BT(\Bu,\te)\Bn_+ &= \Bk && && \text{on $\Ga_+$,} \\
		\Bu &= 0 && && \text{on $\Ga_-$.} 
	\end{aligned}\right.
\end{equation}

Here the unknowns $\Bu = (u_1(x),\dots,u_N(x))^T\footnote[2]{$\BM^T$ denotes the transposed $\BM$.}$
and $\te=\te(x)$ are an $N$-vector function and a scalar function, respectively,
while the right members $\Bf=(f_1(x),\dots,f_N(x))^T$, $g=g(x)$, $\Bh=(h_1(x),\dots,h_N(x))^T$, and $\Bk=(k_1(x),\dots,k_N(x))^T$
are given functions.
Let $\rho_\pm$ be positive constants and 
$\mu_\pm=\mu_\pm(x)$ scalar functions defined on $\ws$,
and let $\chi_D$ be the indicator function of $D\subset\ws$.
Then $\rho=\rho_+\chi_{\Om_+}+\rho_-\chi_{\Om_-}$, $\mu=\mu_+\chi_{\Om_+}+\mu_-\chi_{\Om_-}$, and
$\BT(\Bu,\te)=\mu\BD(\Bu)-\te\BI$, where
$\BI$ is the $N\times N$ identity matrix and
$\BD(\Bu)$ is the doubled deformation tensor,
that is, the $(i,j)$-entry $D_{ij}(\Bu)$ of $\BD(\Bu)$ is given by $D_{ij}(\Bu)=\pa_i u_j + \pa_j u_i$
for $i,j=1,\dots,N$ and $\pa_i = \pa/\pa x_i$.
In addition, $\Bn$ denotes on $\Ga$ a unit normal vector, pointing from $\Om_+$ to $\Om_-$,
while $\Bn_+$ the unit outward normal vector on $\Ga_+$.
For any function $f$ defined on $\dot{\Om}$,
$\jump{f}$ denotes a jump of $f$ across the interface $\Ga$ as follows:
\begin{equation*}
	\jump{f}
		= \jump{f}(x)
		= \lim_{y\to x,\,y\in\Om_+}f(y) - \lim_{y\to x,\,y\in\Om_-}f(y) \quad \text{($x\in\Ga$)}.
\end{equation*}

Here and subsequently, we use the following symbols for differentiations:

Let $f=f(x)$, $\Bg=(g_1(x),\dots,g_N(x))^T$,
and $\BM=(M_{ij}(x))$ be a scalar, an $N$-vector, and an $N\times N$-matrix function defined on a domain of $\ws$, respectively,
and then
\begin{align*}
	&\nabla f = (\pa_1 f(x),\dots, \pa_N f(x))^T, \quad
	\Delta f = \sum_{j=1}^N\pa_j^2 f(x), \quad \De\Bg = (\De g_1(x),\dots,\De g_N(x))^T, \displaybreak[0] \\
	&\di\Bg = \sum_{j=1}^N\pa_j g_j(x), \quad
	\nabla^2 \Bg = \{\,\pa_i\pa_j g_k(x) \mid i,j,k=1,\dots,N\,\}, \displaybreak[0] \\ 
	&\nabla\Bg =
		\begin{pmatrix}
			\pa_1 g_1(x) & \dots & \pa_N g_1(x) \\
			\vdots & \ddots & \vdots \\
			\pa_1 g_N(x) & \dots & \pa_N g_N(x)
		\end{pmatrix}, \quad
	\Di\BM = \left(\sum_{j=1}^N \pa_j M_{1j}(x),\dots,\sum_{j=1}^N \pa_j M_{Nj}(x)\right)^T.
\end{align*}

Problem \eqref{TSRP} arises from a linearized system of some two-phase problem
of the Navier-Stokes equations for viscous, incompressible, and immiscible fluids without taking surface tension into account.
There are a lot of studies of two-phase problems for the Navier-Stokes equations.
To see the history of study briefly,
we restrict ourselves to the case where the two fluids are both viscous, incompressible, and immiscible in the following.
Such a situation was treated in several function spaces as follows:

{\it $L_2$-in-time and $L_2$-in-space setting.}
Denisova \cite{Denisova90, Denisova94} treated the motion of a drop $\Om_{+t}$,
which is the region occupied by the drop at time $t>0$,
in another liquid $\Om_{-t}=\BR^3\setminus\overline{\Om_{+t}}$.
More precisely, \cite{Denisova90} showed some estimates of solutions for linearized problems and
\cite{Denisova94}  an unique existence theorem local in time for the two-phase problem
describing the aforementioned situation with or without surface tension.
In addition, Denisova \cite{Denisova14} proved the unique existence of global-in-time solutions for
small initial data and its exponential stability in the case where $\Om_{-t}$ is bounded and surface tension does not work.
Concerning non-homogeneous incompressible fluids,
Tanaka \cite{Tanaka93} showed the unique existence of global-in-time solutions for small initial data
when $\Om_{- t}$ is bounded, but surface tension is taken into account.

{\it H\"older function spaces.}
A series of papers Denisova-Solonnikov \cite{DS91,DS95} and Denisova \cite{Denisova93} treated
the same motion as in \cite{Denisova90,Denisova94} mentioned above.
Especially, \cite{DS91} and \cite{Denisova93} established estimates of solutions for some linearized problems,
and \cite{DS95} proved an unique existence theorem local in time for the two-phase problem with surface tension.
On the other hand, the unique existence of global-in-time solutions for small initial data was proved by Denisova \cite{Denisova07} without surface tension
and by Denisova-Solonnikov \cite{DS11} with surface tension in the case where $\Om_{-t}$ is bounded.
Furthermore, there are other topics Denisova \cite{Denisova05} and  Denisova-Ne{\v c}asov\'a \cite{DN08},
which consider thermocapillary convection and Oberbeck-Boussinesq approximation, respectively.

{\it $L_p$-in-time and $L_p$-in-space setting.}
Pr\"uss and Simonett \cite{PS10,PS10b,PS11} treated a situation that
two fluids occupy $\Om_{\pm t}=\{(x',x_N) \mid x'\in\BR^{N-1},\, \pm (x_N-h(x',t))>0\}$, respectively, where
$h(x',t)$ is an unknown scalar function describing
the interface $\Ga_t=\{(x',x_N) \mid x'\in\tws,\,x_N=h(x',t)\}$ of the fluids.
\cite{PS10b} and \cite{PS11} proved the local solvability of the two-phase problem with surface tension
and with both surface tension and gravity, respectively, for small initial data.
On the other hand, \cite{PS10} pointed out that
the Rayleigh-Taylor instability occurs if gravity works and
if the fluid occupying $\Om_{+t}$ is heavier than the other one .
Furthermore, Hieber and Saito \cite{HS1} extended the results of the Newtonian case of \cite{PS10b,PS11} to a generalized Newtonian one.
K\"ohne, Pr\"uss, and Wilke \cite{KPW13} showed the local solvability and the global solvability in the case where
$\Om_{\pm t}$ are bounded and surface tension is taken into account.

{\it $L_p$-in-time and $L_q$-in-space setting.}
Shibata-Shimizu \cite{SS11b} showed a maximal $L_p\text{-}L_q$ regularity theorem for
a linearized system of the two-phase problem considered in \cite{PS10, PS11} mentioned above.

This paper is a continuation of Shibata-Shimizu \cite{SS11b}.
Our aim is in the present paper to prove the existence of $\CR$-bounded solution operator families of \eqref{TSRP}
for $\dot\Om=\Om_+\cup\Om_-$ with uniform $W_r^{2-1/r}$ domains $\Om_\pm$ ($N<r<\infty$), which is introduced in Definition \ref{defi:uni} below.
In addition, the $\CR$-bounded solution operator families enable us to show generation of analytic semigroup and
some maximal $L_p\text{-}L_q$ regularity theorem for a time-dependent problem associated with \eqref{TSRP},
which are provided in Subsection \ref{sub:anal} and Subsection \ref{sub:max}, respectively.
We want to emphasize that the maximal $L_p\text{-}L_q$ regularity theorem extends \cite{SS11b} to uniform $W_r^{2-1/r}$ domains and to variable viscosities.

The strategy of this paper follows Shibata \cite{Shibata14}.
We extend his method for one-phase problem to one for two-phase problem.
For example, a two-phase version of the weak Dirichlet-Neumann problem
(it is called a weak elliptic transmission problem in the present paper)
introduced in Definition \ref{defi:weakDN} below,
which plays an important role in this paper, and especially in derivation of
reduced Stokes resolvent equations (cf. Subsection \ref{sub:reduced} below) and in  Lemma \ref{lemm:K} below.
One of  the main advantage of the reduced equations is that we can eliminate the divergence equation: $\di\Bu = g$ in $\dot\Om$,
which is difficult to treat in localized problems, from the problem \eqref{TSRP}.
On the other hand, Lemma \ref{lemm:K} enable us to control localized pressure term.
There however is a remark on Shibata's paper \cite{Shibata14}.
It seems to be difficult to obtain \cite[Theorem 3.1]{Shibata14} from \cite[Theorem 3.4]{Shibata14}
and to obtain \cite[Theorem 3.8]{Shibata14} from \cite[Theorem 3.10]{Shibata14},
because the $\CR$-boundedness of $\la\mathfrak{g}_D(\la)$, $\la\mathfrak{g}_N(\la)$ was not proved
in his paper (cf. \cite[Proof of Theorem 3.1, Proof of Theorem 3.4]{Shibata14}).
We essentially need the $\CR$-boundedness of such operators since the right members $\wt f$ for (3,7),  (3.20)
of \cite{Shibata14} contain $\la V_F(g)$, $\la V_D(g)$, respectively. 
Natural spaces for ranges of the operators $\la\mathfrak{g}_D(\la)$, $\la\mathfrak{g}_N(\la)$ are given by negative Sobolev spaces,
which is main difficulty to modify his proof.
To overcome this difficulty,  we introduce in this paper Proposition \ref{prop:w-w},
which allows us to avoid such negative spaces. 
Following the strategy of Proposition \ref{prop:w-w}, we can also complete his results.


\subsection{Notation and main results}\label{sub:1-2}
We first state notation used throughout this paper.

Let $\BN$ be the set of all natural numbers and $\BN_0=\BN\cup\{0\}$.
For any multi-index $\al=(\al_1,\dots,\al_N)\in\BN_0^N$,
we set $D^\al f =\pa_1^{\al_1}\dots\pa_N^{\al_N}f$.
Let $G$ be an open set of $\BR^N$. Then $L_q(G)$ and $W_q^m(G)$ with $m\in\BN$ denote
the usual $\BK$-valued Lebesgue space and Sobolev space on $G$ for $\BK=\BR$ or $\BK=\BC$,
while $\|\cdot\|_{L_q(G)}$ and $\|\cdot\|_{W_q^m(G)}$ their norms, respectively.
Here we set $W_q^0(G)=L_q(G)$.
In addition, $W_q^s(G)$ with $s\in(0,\infty)\setminus\BN$ denotes the $\BK$-valued Sobolev-Slobodezki space endowed with norm $\|\cdot\|_{W_q^s(G)}$,
and also $C_0^\infty(G)$ the function space of all $C^\infty$ functions $f:G\to\BK$ such that $\supp f$ is compact and $\supp f\subset G$.

For two Banach spaces $X$ and $Y$, $\CL(X,Y)$ is the set of all bounded linear operators from $X$ to $Y$,
and $\CL(X)$ the abbreviation of $\CL(X,X)$.
Let $U$ be a domain of $\BC$, and then $\Hol(U,\CL(X,Y))$ stands for the set of all $\CL(X,Y)$-valued holomorphic functions defined on $U$.

For $d\in\BN$ with $d\geq 2$, $X^d$ denotes the $d$-product space of $X$.
Let $\|\cdot\|_X$ be a norm of $X$, while $\|\cdot\|_X$ also denotes the norm of the product space $X^d$ for short,
that is, $\|\Bf\|_X = \sum_{j=1}^d\|f_j\|_X$ for $\Bf = (f_1,\dots,f_d)\in X^d$.

Let $\Ba=(a_1,\dots,a_N)^T$ and $\Bb=(b_1,\dots,b_N)^T$, and then we write $\Ba\cdot\Bb=<\Ba,\Bb>=\sum_{j=1}^N a_j b_j$
and $\Ba\otimes\Bb=(a_ib_j)$ that is an $N\times N$ matrix with the $(i,j)$-entry $a_ib_j$.
On the other hand, for any vector functions $\Bu$, $\Bv$ on $G$,
we set $(\Bu,\Bv)_G =\int_G \Bu\cdot\Bv\intd x$ and 
$(\Bu,\Bv)_{\pa G} = \int_{\pa G}\Bu\cdot\Bv\intd \si$,
where $\pa G$ is the boundary of $G$ and $d\si$ the surface element of $\pa G$.

Given $1<q<\infty$, we set $q'=q/(q-1)$.
Let $L_{q,\loc}(\overline{G})$ be the vector space of all measurable functions $f$: $G\to \BK$ such that
 $f\in L_q(G\cap B)$ for any ball $B$ of $\ws$.
We define a homogeneous space $\wh{W}_q^1(G)$ by
$\wh{W}_q^1(G) = \{f\in L_{q,\loc}(\overline{G}) \mid \nabla f \in L_q(G)^N\}$ with norm $\|\cdot\|_{\wh{W}_q^1(G)}:=\|\nabla\cdot\|_{L_q(G)}$,
where we have to identify two elements differing by a constant.
In addition, let $\wh{W}_{q,0}^1(G)$ and $W_{q,0}^1(G)$ be Banach spaces defined by
$X_{q,0}^1(G) = \{f\in X_q^1(G) \mid \text{$f=0$ on $\pa G$}\}$ $(X\in\{\wh{W},W\})$
with norms $\|\cdot\|_{\wh{W}_{q,0}^1(G)}:=\|\nabla\cdot\|_{L_q(G)}$ and $\|\cdot\|_{W_{q,0}^1(G)}:=\|\cdot\|_{W_q^1(G)}$, respectively.

Throughout this paper, the letter $C$ denotes generic constants and $C_{a,b,c,\dots}$ means that
the constant depends on the quantities $a$, $b$, $c$, \dots.
The values of constants $C$ and $C_{a,b,c,\dots}$ may change from line to line.


Secondly, we show some definitions.
Uniform $W_r^{2-1/r}$ domains are defined as follows:

\begin{defi}\label{defi:uni}
Let $1<r<\infty$ and $D$ be a domain of $\ws$ with boundary $\pa D$.
We say that $D$ is a uniform $W_r^{2-1/r}$ domain,
if there exist positive constants $\al$, $\beta$, and $K$ such that
for any $x_0=(x_{01},\dots,x_{0N})\in\pa D$
there are a coordinate number $j$ and a $W_r^{2-1/r}$ function $h(x')$
$(x'=(x_1,\dots,\wh{x}_j,\dots,x_N))$ defined on $B_\al'(x_0')$,
with $x_0'=(x_{01},\dots,\wh{x}_{0j},\dots,x_{0N})$ and
$\|h\|_{W_r^{2-1/r}(B_\al'(x_0'))}\leq K$, such that
\begin{align*}
	D\cap B_\beta(x_0) &= \{x\in\ws \mid x_j>h(x'),\, x'\in B_\al'(x_0')\} \cap B_\beta(x_0), \displaybreak[0] \\
	\pa D \cap B_\beta(x_0) &= \{x \in \ws \mid x_j=h(x'),\, x'\in B_\al'(x_0')\} \cap B_\beta(x_0).
\end{align*}
Here $(x_1,\dots,\wh{x}_j,\dots,x_N)=(x_1,\dots,x_{j-1},x_{j+1},\dots,x_N)$,
$B_\al'(x_0')=\{x'\in\tws \mid |x'-x_0'|<\al\}$, and
$B_\beta(x_0) = \{x\in\ws \mid |x-x_0|<\beta\}$.
\end{defi}

We next introduce the definition of the $\CR$-boundedness of operator families.

\begin{defi}\label{defi:R}
Let $X$ and $Y$ be two Banach spaces.
A family of operators $\CT\subset\CL(X,Y)$ is called $\CR$-bounded on $\CL(X,Y)$,
if there exist constants $C>0$ and $p\in[1,\infty)$ such that the following assertion holds:
For each natural number $n$, $\{T_j\}_{j=1}^n\subset \CT$, $\{f_j\}_{j=1}^n\subset X$ and for all sequences $\{r_j(u)\}_{j=1}^n$
of independent, symmetric, $\{-1,1\}$-valued random variables on $[0,1]$,
there holds the inequality:
\begin{equation*}
	\Big(\int_0^1\Big\|\sum_{j=1}^n r_j(u)T_j f_j\Big\|_Y^p\intd u\Big)^{1/p}
		\leq C\Big(\int_0^1\Big\|\sum_{j=1}^n r_j(u)f_j\Big\|_X^p\intd u\Big)^{1/p}.
\end{equation*}
The smallest such $C$ is called $\CR$-bound of $\CT$ on $\CL(X,Y)$,
which is denoted by $\CR_{\CL(X,Y)}$.
\end{defi}

\begin{rema}
The constant $C$ in Definition \ref{defi:R} depends on $p$.
On the other hand, it is well-known that $\CT$ is $\CR$-bounded for any $p\in[1,\infty)$, provided that $\CT$ is $\CR$-bounded for some $p\in[1,\infty)$.
This fact follows from Kahane's inequality (cf. \cite[Theorem 2.4]{KW04}).
\end{rema}

Furthermore, we introduce a weak elliptic transmission problem.
In the present paper, $\Ga_+=\emptyset$ or $\Ga_-=\emptyset$ are admissible,
but note that $\Ga\neq\emptyset$.
Let $W_{q,\Ga_+}^1(\Om)$ and $\wh{W}_{q,\Ga_+}^1(\Om)$ be Banach spaces defined by
\begin{equation*}
	X_{q,\Ga_+}^1(\Om) =
		\left\{\begin{aligned}
			&\{f\in X_q^1(\Om) \mid f=0 \text{ on $\Ga_+$}\} && \text{if $\Ga_+\neq \emptyset$,} \\
			&X_q^1(\Om) && \text{if $\Ga_+=\emptyset$}
		\end{aligned}\right.
\end{equation*}
for $X\in\{W,\wh{W}\}$, and their norms are given by
	$\|\cdot\|_{W_{q,\Ga_+}^1(\Om)} = \|\cdot\|_{W_q^1(\Om)}$ and $\|\cdot\|_{\wh{W}_{q,\Ga_+}^1(\Om)} = \|\nabla \cdot\|_{L_q(\Om)}$, respectively.
The unique solvability of the weak elliptic transmission problem is defined in the following.
\begin{defi}\label{defi:weakDN}
	Let $1<q<\infty$ and $q'=q/(q-1)$.
	Let 	$\CW_q^1(\Om)$ be a closed subspace of $\wh{W}_{q,\Ga_+}^1(\Om)$,
	and suppose that $W_{q,\Ga_+}^1(\Om)$ is dense in $\CW_q^1(\Om)$.
	Set $\rho=\rho_+\chi_{\Om_+}+\rho_-\chi_{\Om_-}$ for positive constants $\rho_\pm$.
	Then we say that the weak elliptic transmission problem is uniquely solvable on $\CW_q^1(\Om)$ for $\rho_\pm$
	if the following assertion holds:
	For any $\Bf \in L_q(\Om)^N$, there is a unique $\te\in \CW_q^1(\Om)$ satisfying the variational equation:
	\begin{equation}\label{151214}
		(\rho^{-1}\nabla\te,\nabla\ph)_{\dot\Om} = (\Bf,\nabla\ph)_{\Om} \quad
		\text{for all $\ph\in\CW_{q'}^1(\Om)$,}
	\end{equation}
	which possesses the estimate: $\|\nabla\te\|_{L_q(\Om)}\leq C\|\Bf\|_{L_q(\Om)}$ with a positive constant $C$
	independent of $\te$, $\ph$, and $\Bf$.	
\end{defi}

\begin{rema}\label{rema:weakDN}
\begin{enumerate}[(1)]
	\item\label{rema:weakDN4}
		Let $1<q<\infty$, $q'=q/(q-1)$, and let the weak elliptic transmission problem be uniquely solvable on $\CW_q^1(\Om)$ for $\rho_+=\rho_-=1$.
		We define $J_q(\Om)$ and $G_q(\Om)$ by
		\begin{align*}
			&J_q(\Om) = \{\Bf\in L_q(\Om)^N \mid (\Bf,\nabla\ph)_{\Om} =0 \quad \text{for all $\ph\in \CW_{q'}^1(\Om)$}\}, \\
			&G_q(\Om) = \{\Bf\in L_q(\Om)^N \mid \Bf = \nabla\te \quad \text{for some $\te\in\CW_q^1(\Om) $} \}.
		\end{align*}
		Then, by the standard proof, the so-called {\it Helmholtz decomposition}: $L_q(\Om)^N = J_q(\Om)\oplus G_q(\Om)$ holds.
	\item\label{rema:weakDN3}
		In applications, we choose $\CW_q^1(\Om)$ in such a way that
		the weak elliptic transmission problem is uniquely solvable for $\rho_\pm$.
		Typical examples are as follows:
		$\CW_q^1(\ws)=\wh{W}_q^1(\ws)$;
		$\CW_q^1(\uhs)=\wh{W}_q^1(\uhs)$ with $\Ga_+=\emptyset$ and $\Ga_-=\BR_0^N=\{(x',x_N) \mid x'\in\tws, x_N=0\}$,
		and $\CW_q^1(\Om)$ is analogously defined by $\wh{W}_q^1(\Om)$ when $\Om$ is a perturbed $\uhs$;
		$\CW_q^1(\lhs)=\wh{W}_{q,\Ga_+}^1(\lhs)$ with $\Ga_+=\BR_0^N$ and $\Ga_-=\emptyset$,
		and $\CW_q^1(\Om)$ is analogously defined by $\wh{W}_{q,\Ga_+}^1(\Om)$ when $\Om$ is a perturbed $\lhs$;
		$\CW_q^1(\Om)=\wh{W}_{q,\Ga_+}^1(\Om)$ when $\Om$ is a bounded domain, a layer, or a perturbed layer. 
		We refer e.g. to \cite[Appendix A.1]{KPW13} for the treatment of weak elliptic transmission problems.
	\item\label{rema:weakDN2}
		We set $W_q^1(\dot\Om)+\CW_q^1(\Om)
		=\{\te_1+\te_2 \mid \te_1\in W_q^1(\dot\Om),\,\te_2\in\CW_q^1(\Om)\}$.
		Suppose that the weak elliptic transmission problem is uniquely solvable on $\CW_q^1(\Om)$ for $\rho_\pm$.
		Then, for any $\al\in L_q(\dot\Om)^N$, $\beta\in W_q^{1-1/q}(\Ga)$, and $\ga\in W_q^{1-1/q}(\Ga_+)$,
		there exists a unique $\te\in W_q^1(\dot\Om)+\CW_q^1(\Om)$ satisfying the weak problem:
		\begin{equation}\label{160211_5}
			(\rho^{-1}\nabla\te,\nabla\ph)_{\dot\Om} = (\al,\nabla\ph)_{\dot\Om} \quad
			\text{for all $\ph\in\CW_q^1(\Om)$,} \quad
			\jump{\te} = \beta \quad \text{on $\Ga$,} \quad \te = \ga \quad \text{on $\Ga_+$,}
		\end{equation}
		which possesses the estimate:
		\begin{equation*}
			\|\nabla\te\|_{L_q(\dot\Om)}
				\leq C \left(\|\al\|_{L_q(\dot\Om)}+\|\beta\|_{W_q^{1-1/q}(\Ga)}+\|\ga\|_{W_q^{1-1/q}(\Ga_+)}\right)
		\end{equation*}
		with some positive constant $C$ independent of $\al$, $\beta$, $\ga$, $\te$, and $\ph$.
		Thus, it is possible to define a linear operator
		$\CK:L_q(\dot\Om)^N\times W_q^{1-1/q}(\Ga)\times W_q^{1-1/q}(\Ga_+)\to W_q^1(\dot\Om)+\CW_q^1(\Om)$
		by $\CK(\al,\beta,\ga)=\te$ satisfying \eqref{160211_5}.
		If $\Ga_+=\emptyset$, then we denote $\CK(\al,\beta,\ga)$ by $\CK(\al,\beta,\emptyset)$ when $\Ga_-\neq\emptyset$
		and by $\CK(\al,\beta)$ when $\Ga_-=\emptyset$. 
\end{enumerate}
\end{rema}

We now state our main results.
To this end, we introduce a data space for the divergence equation:
$\di\Bu = g$ in $\dot\Om$ with boundary conditions: $\jump{\Bu}\cdot\Bn =0$ on $\Ga$ and $\Bu\cdot\Bn_-=0$ on $\Ga_-$,
where $\Bn_-$ is the unit outward normal vector on $\Ga_-$.
Let $\CW_q^{-1}(\Om)$ be the dual space of $\CW_{q'}^1(\Om)$ for $1<q<\infty$ and $q'=q/(q-1)$,
and let $\|\cdot\|_{\CW_q^{-1}(\Om)}$ and $<\cdot,\cdot>_\Om$ be its norm and the duality pairing between $\CW_q^{-1}(\Om)$ and $\CW_{q'}^1(\Om)$,
respectively.
Then we set
\begin{equation*}
	L_{q}(\Om)\cap \CW_q^{-1}(\Om)
		=\left\{g\in L_q(\Om) \mid \exists M>0 \text{ s.t. } |(g,\ph)_\Om|\leq M\|\nabla\ph\|_{L_{q'}(\Om)}
		\text{ for any $\ph\in W_{q',\Ga_+}^1(\Om)$}\right\}.
\end{equation*}
Let $g\in L_q(\Om)\cap \CW_{q}^{-1}(\Om)$, and thus $g$ can be extended uniquely to an element of $\CW_{q}^{-1}(\Om)$.
Such an extended $g$ is again denoted by $g$ for short.
We can see $g$ as a functional on $\{\nabla\te\mid\te\in\CW_{q'}^1(\Om)\}\subset L_{q'}(\Om)^N$,
which, combined with Hahn-Banach's theorem, furnishes that
there is a $\BG\in L_q(\Om)^N$ such that
$\|g\|_{\CW_q^{-1}(\Om)}=\|\BG\|_{L_q(\Om)}$ and $<g,\ph>_\Om = -(\BG,\nabla\ph)_\Om$ for all $\ph\in\CW_{q'}^1(\Om)$.
In what follows, $\BG$ is restricted to the functional on $\{\nabla\te\mid\te\in \CW_{q'}^1(\Om)\}$.
Let $\wt{L}_q(\Om)=L_q(\Om)^N/J_q(\Om)$, and let $[\BG]=\{\BG+\Bf \mid \Bf\in J_q(\Om)\}\in \wt{L}_q(\Om)$.
Then $g\mapsto[\BG]$ is well-defined,
so that we denote $[\BG]$ by $\CG(g)$.
Especially, we have, for $g\in L_q(\Om)\cap \CW_q^{-1}(\Om)$ and for any representative $\mathfrak{g}\in L_q(\Om)^N$ of $\CG(g)$,
\begin{equation}\label{DI}
	(g,\ph)_{\Om} = -(\mathfrak{g},\nabla\ph)_\Om \quad \text{for all $\ph\in W_{q',\Ga_+}^1(\Om)$}.
\end{equation}

Here we set $\BW_q^{-1}(\Om)=L_q(\Om)\cap\CW_q^{-1}(\Om)$. 
Then $W_q^1(\dot\Om)\cap\BW_q^{-1}(\Om)$ is a Banach space endowed with norm
$\|\cdot\|_{W_q^1(\dot\Om)\cap\BW_q^{-1}(\Om)}:=\|\cdot\|_{W_q^1(\dot\Om)} + \|\cdot\|_{\CW_q^{-1}(\Om)}$,
and the function space is characterized as the data space for the divergence equation above.
The following theorem presents the main result of this paper.



\begin{theo}\label{theo:main}
Let $1<q<\infty$, $0<\ep<\pi/2$, $N<r<\infty$, and $\max(q,q')\leq r$ with $q'=q/(q-1)$.
Let $\rho_\pm$ be positive constants.
Suppose that the following three conditions holds:
\begin{enumerate}[{\rm (a)}]
	\item
		$\Om_\pm$ are uniform $W_r^{2-1/r}$ domains;
	\item
		The weak elliptic transmission problem is uniquely solvable on $\CW_q^1(\Om)$ and $\CW_{q'}^1(\Om)$ for $\rho_\pm$;
	\item
		$\mu_\pm$ are real valued uniformly continuous functions defined on $\ws$ and
		there exist positive constants $\mu_{\pm1}$, $\mu_{\pm 2}$ such that
		\begin{equation*}
			\mu_{+1} \leq \mu_+(x) \leq \mu_{+2}, \quad
			\mu_{-1} \leq \mu_-(x) \leq \mu_{-2} \quad 
			\text{for any $x\in\ws$.}
		\end{equation*}
		In addition, $\mu_\pm\in W_{r,\loc}^1(\ws)$ and $\|\nabla\mu_\pm\|_{L_r(B)}\leq K_{r,\tau}$
		with some positive constant $K_{r,\tau}$ for any ball $B\subset\ws$ with radius $\tau$.
\end{enumerate}
\begin{enumerate}[$(1)$]
	\item {\bf Existence.}
		Set 
		\begin{align*}
			X_q =& \{(\Bf,g,\Bh,\Bk) \mid \Bf\in L_q(\dot\Om)^N,\,g\in W_q^1(\dot\Om)\cap \BW_q^{-1}(\Om),\,\Bh\in W_q^1(\dot\Om)^N,\,\Bk\in W_q^1(\Om_+)^N\}, \\
			\CX_q =& \{(F_1,\dots,F_{11})\mid 
				F_1,F_2,F_4,F_7\in L_q(\dot\Om)^N,\, F_3\in L_q(\dot\Om),\, F_5\in W_q^1(\dot\Om), \\ 
				&\quad F_6\in L_q(\dot\Om)^{N^2},\,F_8\in W_q^1(\dot\Om)^N,\,F_9\in L_q(\Om_+)^{N^2},\,F_{10}\in L_q(\Om_+)^N,\,F_{11}\in W_q^1(\Om_+)^N\}.
		\end{align*}
		Then there exists a constant $\la_0\geq 1$ and operator families:
		\begin{equation*}
			\BA(\la) \in \Hol(\Si_{\ep,\la_0}, \CL(\CX_q,W_q^2(\dot\Om)^N)), \quad
			\BP(\la) \in \Hol(\Si_{\ep,\la_0}, \CL(\CX_q,W_q^1(\dot\Om)+\CW_q^1(\Om)))
		\end{equation*}
		such that, for any $\la\in\Si_{\ep,\la_0}$ and for any $(\Bf,g,\Bh,\Bk)\in X_q$ and $\mathfrak{g}\in\CG(g)$,
		$\Bu = \BA(\la)F_\la(\Bf,g,\mathfrak{g},\Bh,\Bk)$ and $\te=\BP(\la)F_\la(\Bf,g,\mathfrak{g},\Bh,\Bk)$
		are solutions to \eqref{TSRP}, and furthermore,
		\begin{align}
			\CR_{\CL(\CX_q,L_q(\dot\Om)^{\wt N})}\Big(\Big\{\,\Big(\la\frac{d}{d\la}\Big)^l
				\left(R_\la\BA(\la)\right)\mid \la\in\Si_{\ep,\la_0}\,\Big\}\Big) &\leq \ga_0 \quad (l=0,1), \label{160124_1} \\
			\CR_{\CL(\CX_q,L_q(\dot\Om)^N)}\Big(\Big\{\,\Big(\la\frac{d}{d\la}\Big)^l
				\nabla\BP(\la)\mid \la\in\Si_{\ep,\la_0}\,\Big\}\Big) &\leq \ga_0 \quad (l=0,1) \label{160124_2}
		\end{align}
		for some positive constant $\ga_0$. Here we have set
		$\wt N = N^3+N^2 + N$, $R_\la\Bu = (\nabla^2\Bu,\la^{1/2}\nabla\Bu,\la\Bu)$, and
		\begin{equation*}
			F_\la(\Bf,g,\mathfrak{g},\Bh,\Bk)
				= (\Bf, \nabla g,\la^{1/2}g, \la\mathfrak{g},g,\nabla\Bh,\la^{1/2}\Bh,\Bh,\nabla\Bk,\la^{1/2}\Bk,\Bk).  
		\end{equation*} 
	\item {\bf Uniqueness.}
		There exists a $\la_0\geq 1$ such that
		if $\Bu\in W_q^2(\dot\Om)^N\cap J_q(\Om)$ and $\te\in W_q^1(\dot\Om)+\CW_q^1(\Om)$
		satisfies the homogeneous equations:
		\begin{align*}
			&\la\Bu-\rho^{-1}\Di\BT(\Bu,\te) =0 \quad \text{in $\dot\Om$,}\quad
			\jump{\BT(\Bu,\te)\Bn}=0, \quad \jump{\Bu}=0 \quad \text{on $\Ga$,} \\
			&\BT(\Bu,\te)\Bn_+ =0 \quad \text{on $\Ga_+$,} \quad
			\Bu=0 \quad \text{on $\Ga_-$}
		\end{align*}
		with $\la\in\Si_{\ep,\la_0}$, then $\Bu=0$ in $\dot\Om$.
\end{enumerate}
\end{theo}

\begin{rema}\label{rema:main}
The symbols $F_1$, $F_2$, $F_3$, $F_4$, $F_5$, $F_6$, $F_7$, $F_8$, $F_9$, $F_{10}$, and $F_{11}$
are variables corresponding to $\Bf$, $\nabla g$, $\la^{1/2}g$, $\la\mathfrak{g}$, $g$,
$\nabla\Bh$, $\la^{1/2}\Bh$, $\Bh$, $\nabla\Bk$, $\la^{1/2}\Bk$, and $\Bk$, respectively.  
The norm of space $\CX_q$ is given by
\begin{equation*}
	\|(F_1,\dots,F_{11})\|_{\CX_q} = \|(F_1,F_2,F_3,F_4,F_6,F_7)\|_{L_q(\dot\Om)}+\|(F_5,F_8)\|_{W_q^1(\dot\Om)}+
	\|(F_9,F_{10})\|_{L_q(\Om_+)} + \|F_{11}\|_{W_q^1(\Om_+)}.
\end{equation*}
\end{rema}

This paper is organized as follows:
The next section first tells us some equivalence of \eqref{TSRP} and
two-phase reduced Stokes resolvent equations, which are obtained by elimination of pressure term from \eqref{TSRP},
in Subsection \ref{sub:reduced} and Subsection \ref{sub:stokes}.
Secondly, we state our main result for the two-phase reduced Stokes resolvent equations in Subsection \ref{sub:main},
which, combined with what pointed out in Subsection \ref{sub:stokes}, allows us to conclude that Theorem \ref{theo:main} holds.
Thirdly, we state generation of analytic semigroup and some maximal $L_p\text{-}L_q$ regularity theorem
for two-phase problems of time-dependent Stokes equations in Subsection \ref{sub:anal} and Subsection \ref{sub:max}, respectively,
with help of Theorem \ref{theo:main} and the main result stated in Subsection \ref{sub:main}. 
Section \ref{sec:ws} proves
our main result for the two-phase reduced Stokes resolvent equations
in the case where $\dot\Om=\dws=\BR_+^N\cup\BR_-^N$,
$\BR_\pm^N=\{(x',x_N) \mid x'\in\tws,\,\pm x_N>0\}$, with constant viscosity coefficients.
Section \ref{sec:bent} proves our main result for the two-phase reduced Stokes resolvent equations  
with variable viscosity coefficients
when $\dot\Om$ is a perturbed $\dws$  by using results obtained in Section \ref{sec:ws}.
Section \ref{sec:proof} shows the main result stated in Subsection \ref{sub:main}
by using results obtained in Section \ref{sec:bent} together with some localization technique.

\section{Generation of analytic semigroup and maximal regularity}\label{sec:st-rst}
In this section, after introducing the Stokes operator in \eqref{op:stokes} below,
we consider the following initial-boundary value problem:
\begin{equation}\label{IBVP:1}
	\left\{\begin{aligned}
		\pa_t\Bv -\rho^{-1}\Di\BT(\Bv,\pi) &= \Bf, &\di\Bv&=g && \text{in $\dot\Om\times(0,\infty)$,} \\
		\jump{\BT(\Bv,\pi)\Bn} &= \jump{\Bh} &\jump{\Bv}&=0 && \text{on $\Ga\times(0,\infty)$,} \\
		\BT(\Bv,\pi)\Bn_+ &= \Bk && && \text{on $\Ga_+\times(0,\infty)$,} \\
		\Bv &= 0 && && \text{on $\Ga_-\times(0,\infty)$,}\\
		\Bv|_{t=0} &= \Bv_0 && && \text{in $\dot\Om$,}
	\end{aligned}\right.
\end{equation}
which is called the two-phase Stokes equations in this paper.
We discuss the generation of analytic semigroup associated with \eqref{IBVP:1} and
some maximal $L_p\text{-}L_q$ regularity property for \eqref{IBVP:1}.
To consider the generation of analytic semigroup, we have to formulate \eqref{IBVP:1} in the semigroup setting, that is,
we have to eliminate the pressure term from \eqref{IBVP:1}. 
Throughout this section,  for some $1<q<\infty$ and positive constants $\rho_\pm$,
we assume that
the weak elliptic transmission problem is uniquely solvable on $\CW_q^1(\Om)$ for $\rho_\pm$.
The assumption plays an essential role to eliminate the pressure term from \eqref{IBVP:1}. 

\subsection{Two-phase reduced Stokes resolvent equations}\label{sub:reduced}
Let $1<q<\infty$, $q'=q/(q-1)$, and $\Bu\in W_q^2(\dot\Om)^N$.
Set $K(\Bu)=\CK(\al,\beta,\ga)\in W_q^1(\dot{\Om})+\CW_q^1(\Om)$,
defined in Remark \ref{rema:weakDN} \eqref{rema:weakDN2}, with
\begin{equation}\label{abc}
	\al = \rho^{-1}\Di(\mu\BD(\Bu))-\nabla\di\Bu, \quad
	\beta= <\jump{\mu\BD(\Bu)\Bn},\Bn>-\jump{\di\Bu}, \quad
	\ga=<\mu\BD(\Bu)\Bn_+,\Bn_+>-\di\Bu.
\end{equation}
%
%
Then $\Bu\mapsto \nabla K(\Bu)$ is a bounded linear operator from $W_q^2(\dot\Om)^N$ to $L_q(\dot\Om)^N$
with $\|\nabla K(\Bu)\|_{L_q(\dot\Om)}\leq C\|\Bu\|_{W_q^2(\dot\Om)}$ for some positive constant $C$ independent of $\Bu$.
We consider the equations as follows:
\begin{equation}\label{redu-eq:1}
	\left\{\begin{aligned}
		\la\Bu-\rho^{-1}\Di\BT(\Bu,K(\Bu)) 
			&= \Bf && \text{in $\dot{\Om}$,} \\
		\jump{\BT(\Bu,K(\Bu))\Bn}
			&= \jump{\Bh} && \text{on $\Ga$,} \\
		\jump{\Bu} &= 0 && \text{on $\Ga$,} \\
		\BT(\Bu,K(\Bu))\Bn_+ &= \Bk && \text{on $\Ga_+$,} \\
		\Bu &= 0 && \text{on $\Ga_-$,}
	\end{aligned}\right.
\end{equation}
which is called the two-phase reduced Stokes resolvent equations.
In this subsection, we construct a solution to \eqref{redu-eq:1} on the assumption that \eqref{TSRP} is solvable.
To this end, we treat the following auxiliary problem:
\begin{equation}\label{weak:1-1}
			(\la u,\ph)_{\dot\Om}+(\nabla u, \nabla\ph)_{\dot\Om} =(\Bf,\nabla\ph)_{\dot\Om} \quad \text{for all $\ph\in W_{q',\Ga_+}^1(\Om)$,} \quad
			\jump{u} =g \quad \text{on $\Ga$,} \quad  u = h \quad \text{on $\Ga_+$,} 
\end{equation}
which is the weak elliptic transmission problem with resolvent parameter $\la$.
Employing the same argument as in the proof of our main result in the present paper, we can show the following proposition.

\begin{prop}\label{prop:weak1}
Let $0<\ep<\pi/2$, $1<q<\infty$, $N<r<\infty$, and $\max(q,q')\leq r$ with $q'=q/(q-1)$.
Suppose that $\Om_\pm$ are uniform $W_r^{2-1/r}$ domains.
Then there is a positive number $\la_0\geq 1$ such that,
for any $\la\in\Si_{\ep,\la_0}$, $\Bf\in L_q(\dot\Om)^N$, $g\in W_q^{1-1/q}(\Ga)$, and $h\in W_q^{1-1/q}(\Ga_+)$,
\eqref{weak:1-1} admit a unique solution $u\in W_q^1(\dot\Om)\cap\BW_q^{-1}(\Om)$.
\end{prop}

We solve \eqref{redu-eq:1} by means of solutions to \eqref{TSRP}.
Given $\Bf\in L_q(\dot\Om)^N$, $\Bh\in W_q^1(\dot\Om)^N$, and $\Bk\in W_q^1(\Om_+)^N$,
we choose by Proposition \ref{prop:weak1} some $g$ in such a way that $g$ solves the weak problem:
\begin{align}
		&(\la g,\ph)_{\dot\Om} + (\nabla g,\nabla\ph)_{\dot\Om} = -(\Bf,\nabla\ph)_{\dot\Om} \quad \text{for all $\ph\in W_{q',\Ga_+}^1(\Om)$,}  \label{150601_1-1}\\
		&\jump{g} = <\jump{\Bh},\Bn> \quad \text{on $\Ga$,} \quad
		g=<\Bk,\Bn_+> \quad \text{on $\Ga_+$.} \label{150601_1-2}
\end{align}

Let $\Bu\in W_q^2(\dot\Om)^N$ and $\te\in W_q^1(\dot\Om) + \CW_q^1(\Om)$ be solutions to \eqref{TSRP} with $\Bf$, $g$, $\Bh$, and $\Bk$ as above.
Then, by the definition of $K(\Bu)$ and Gauss's divergence theorem together with $\jump{\Bu}=0$ on $\Ga$, $\Bu=0$ on $\Ga_-$, we see that
\begin{align*}
	(\Bf,\nabla\ph)_{\dot\Om} =& (\la\Bu-\nabla\di\Bu -\rho^{-1}\nabla K(\Bu)+\rho^{-1}\nabla\te,\nabla\ph)_{\dot\Om} \\
		=& -(\la g,\ph)_{\dot\Om} -(\nabla g,\nabla\ph)_{\dot\Om} + (\rho^{-1}\nabla(\te- K(\Bu)),\nabla\ph)_{\dot\Om}
\end{align*}
for any $\ph\in W_{q,\Ga_+}^1(\Om)$. This combined with \eqref{150601_1-1} and the denseness of $W_{q',\Ga_+}^1(\Om)$ in $\CW_{q'}^1(\Om)$ furnishes that
\begin{equation*}
	(\rho^{-1}\nabla(\te-K(\Bu)),\nabla\ph)_{\dot\Om} =0 \quad \text{for all $\ph\in \CW_{q'}^1(\Om)$.}
\end{equation*}
In addition, it holds that $\jump{K(\Bu)-\te}=0$ on $\Ga$ and $K(\Bu)-\te=0$ on $\Ga_+$,
since $g$ satisfies \eqref{150601_1-2} and
\begin{align*}
	<\jump{\Bh},\Bn> &= <\jump{\mu\BD(\Bu)\Bn},\Bn> - \jump{\te} =\jump{K(\Bu)-\te} + \jump{\di\Bu} 
		= \jump{K(\Bu)-\te} + \jump{g} \quad \text{on $\Ga$,} \\
	<\Bk,\Bn_+> &= <\mu\BD(\Bu)\Bn_+,\Bn_+> - \te = K(\Bu)-\te + \di\Bu  
		= K(\Bu)-\te + g \quad \text{on $\Ga_+$.} 
\end{align*}
Thus the unique solvability of the weak elliptic transmission problem implies $K(\Bu)=\te$,
which means that the solution $\Bu\in W_q^2(\dot\Om)^N$ of \eqref{TSRP} solves \eqref{redu-eq:1} for 
$\Bf\in L_q(\dot\Om)^N$, $\Bh\in W_q^1(\dot\Om)^N$, $\Bk\in W_q^1(\Om_+)^N$, and $g$ of \eqref{150601_1-1}-\eqref{150601_1-2}.

\subsection{Reduced Stokes implies Stokes}\label{sub:stokes}
In this subsection, we solve \eqref{TSRP} on the assumption that \eqref{redu-eq:1} is solvable.
Let $1<q<\infty$ and $q'=q/(q-1)$.
Given $\Bf\in L_q(\dot\Om)^N$, $\Bh\in W_q^1(\dot\Om)^N$, and $\Bk\in W_q^1(\Om_+)^N$,
let $\ka\in W_q^1(\dot\Om) + \CW_q^1(\Om)$ be a solution to the weak problem:
\begin{align*}
	&(\rho^{-1}\nabla\ka,\nabla\ph)_{\dot\Om} = (\Bf,\nabla\ph)_{\dot\Om} \quad \text{for all $\ph\in \CW_{q'}^1(\Om)$,} \\
	&\jump{\ka}=-<\jump{\Bh},\Bn> \quad \text{on $\Ga$,} \quad 
	\ka = -<\Bk,\Bn_+> \quad \text{on $\Ga_+$.}
\end{align*}
Then the system \eqref{TSRP} is reduced to 
\begin{equation*}
	\left\{\begin{aligned}
		\la\Bu -\rho^{-1}\Di\BT(\Bu,\te-\ka) &= \Bf -\rho^{-1}\nabla\ka, &\di\Bu&=g && \text{in $\dot\Om$,} \\
		\jump{\BT(\Bu,\te-\ka)\Bn} &= \jump{\Bh}-<\jump{\Bh},\Bn>\Bn, &\jump{\Bu}&=0 && \text{on $\Ga$,} \\
		\BT(\Bu,\te-\ka)\Bn_+ &= \Bk -<\Bk,\Bn_+>\Bn_+ && && \text{on $\Ga_+$,} \\
		\Bu &=0 && && \text{on $\Ga_-$.}
	\end{aligned}\right.
\end{equation*}
It thus suffices to consider \eqref{TSRP} under the condition that 
\begin{equation}\label{151117_1}
	(\Bf,\nabla\ph)_{\dot\Om} =0 \text{ for all $\ph\in \CW_{q'}^1(\Om)$,} \quad
	<\jump{\Bh},\Bn> = 0 \text{ on $\Ga$,} \quad
	<\Bk,\Bn_+> =0 \text{ on $\Ga_+$.}
\end{equation}

For $\BG=(G_1,G_2)\in L_q(\dot\Om)^N\times W_q^1(\dot\Om)$,
we set  $L(\BG)=L(G_1,G_2)=\CK(G_1-\nabla G_2,-\jump{G_2},-G_2)$
by $\CK$ of Remark \ref{rema:weakDN} \eqref{rema:weakDN2}.
Then $\BG\mapsto\nabla L(\BG)$ is a bounded linear operator from $L_q(\dot\Om)^N \times W_q^1(\dot\Om)$ to $L_q(\dot\Om)^N$.


Given $g\in W_q^1(\dot\Om)\cap\BW_q^{-1}(\Om)$, we choose a representative $\mathfrak{g}$ of $\CG(g)$.
For these $g$, $\mathfrak{g}$ and for $\Bf$, $\Bh$, $\Bk$ satisfying \eqref{151117_1},
let $\Bu\in W_q^2(\dot\Om)^N$ be a solution to the two-phase reduced Stokes resolvent equations as follows:
\begin{equation*}
	\left\{\begin{aligned}
		\la \Bu -\rho^{-1}\Di\BT(\Bu,K(\Bu)) &= \Bf + \rho^{-1}\nabla L(\la\mathfrak{g},g) && \text{in $\dot\Om$,} \\
		\jump{\BT(\Bu,K(\Bu))\Bn} &= \jump{\Bh} + \jump{g}\Bn && \text{on $\Ga$,} \\
		\jump{\Bu} &= 0 && \text{on $\Ga$,} \\
		\BT(\Bu,K(\Bu))\Bn_+ &= \Bk + g\Bn_+ && \text{on $\Ga_+$,} \\
		\Bu &= 0 && \text{on $\Ga_-$.}
	\end{aligned}\right.
\end{equation*}
Then, by \eqref{DI}, \eqref{151117_1} and by the definition of $K(\Bu)$, $L(\la\mathfrak{g},g)$, we have
\begin{align*}
	0 &= (\Bf,\nabla\ph)_{\dot\Om} = (\la\Bu-\rho^{-1}\Di(\mu\BD(\Bu))+\rho^{-1}\nabla K(\Bu)-\rho^{-1}\nabla L(\la\mathfrak{g},g),\nabla\ph)_{\dot\Om} \\
		&= (\la\Bu,\nabla\ph)_{\dot\Om} - (\nabla\di\Bu,\nabla\ph)_{\dot\Om} +(\la g,\ph)_{\dot\Om}+ (\nabla g,\nabla\ph)_{\dot\Om}
\end{align*}
for any $\ph\in W_{q,\Ga_+}^1(\Om)$,
which, combined with Gauss's divergence theorem together with $\jump{\Bu}=0$ on $\Ga$, $\Bu =0$ on $\Ga_-$, furnishes that
	$(\la(\di\Bu -g),\ph)_{\dot\Om} + (\nabla(\di\Bu-g),\nabla\ph)_{\dot\Om} =0$ for all $\ph\in W_{q',\Ga_+}^1(\Om)$.
In addition, we see, by \eqref{151117_1} and the definition of $K(\Bu)$, that
	$\jump{g} = <\jump{\mu\BD(\Bu)\Bn},\Bn> -\jump{K(\Bu)} = \jump{\di\Bu}$  on $\Ga$, 
	$g = <\mu\BD(\Bu)\Bn_+,\Bn_+> - K(\Bu) =\di\Bu$ on $\Ga_+$,
which implies that
	$\jump{\di\Bu -g} =0 $ on $\Ga$, $\di\Bu-g =0$  on $\Ga_+$.
Thus, by Proposition \ref{prop:weak1}, $\di\Bu=g$ in $\dot\Om$,
which means that $\Bu$ and $\te =K(\Bu)-L(\la\mathfrak{g},g)$ solves \eqref{TSRP}.
\subsection{$\CR$-bounded solution operator families of reduced Stokes}\label{sub:main}
According to what was pointed out in Subsection \ref{sub:reduced} and Subsection \ref{sub:stokes},
we consider the two-phase reduced Stokes resolvent equations \eqref{redu-eq:1} instead of \eqref{TSRP} from Section \ref{sec:ws} through Section \ref{sec:proof}.
More precisely, we prove the following theorem.
\begin{theo}\label{theo:redu1}
	Let $1<q<\infty$, $0<\ep<\pi/2$, $N<r<\infty$, and $\max(q,q')\leq r$ with $q'=q/(q-1)$.
	Let $\rho_\pm$ be positive constants.
	Suppose that $({\rm a})$, $({\rm b})$, and $({\rm c})$ stated in Theorem \ref{theo:main} hold.
	For any open set $G$ of $\ws$, let $X_{\CR,q}(G)$ and $\CX_{\CR,q}(G)$ be defined as
	\begin{align*}
		X_{\CR,q}(G) 
			=& \{(\Bf,\Bh,\Bk) \mid \Bf\in L_q(G)^N,\,\Bh \in W_q^1(G)^N,\,\Bk \in W_q^1(G\cap\Om_+)^N\}, \\
	\CX_{\CR,q}(G)
			=& \{(H_1,\dots,H_7) \mid H_1,H_3\in L_q(G)^N,\,H_2\in L_q(G)^{N^2}, \,H_4\in W_q^1(G)^N, \\
			&\quad H_5\in L_q(G\cap\Om_+)^{N^2},\,H_6\in L_q(G\cap\Om_+)^N,\,H_7\in W_q^1(G\cap\Om_+)^N\}.
	\end{align*}
	Then there exist a positive number $\la_0\geq 1$ and an operator family 
		$\BB(\la) \in \Hol (\Si_{\ep,\la_0},\CL(\CX_{\CR,q}(\dot\Om),W_q^2(\dot\Om)^N))$
	such that, for any $\la\in\Si_{\ep,\la_0}$ and $(\Bf,\Bh,\Bk)\in X_{\CR,q}(\dot\Om)$,
	$\Bu = \BB(\la)F_{\CR,\la}(\Bf,\Bh,\Bk)$ is a unique solution to \eqref{redu-eq:1}, and furthermore,
	\begin{equation}\label{160211_1}
		\CR_{\CL(\CX_{R,q}(\dot\Om),L_q(\dot\Om)^{\wt N})}
		\Big(\Big\{\Big(\la\frac{d}{d\la}\Big)^l \left(R_\la\BB(\la)\right)\mid \la\in\Si_{\ep,\la_0}\Big\}\Big)
		\leq \ga_0 \quad(l=0,1)
	\end{equation}
	for some positive constant $\ga_0$. Here we have set $\wt N = N^3+N^2+N$,
	$R_\la\Bu = (\nabla^2\Bu,\la^{1/2}\nabla\Bu,\la\Bu)$, and
	\begin{equation*}
		F_{\CR,\la}(\Bf,\Bh,\Bk) = (\Bf,\nabla\Bh,\la^{1/2}\Bh,\Bh,\nabla\Bk,\la^{1/2}\Bk,\Bk). 
	\end{equation*}
\end{theo}

\begin{rema}\label{rema:div}
\begin{enumerate}[(1)]
	\item\label{rema:div3}
		The symbols $H_1$, $H_2$, $H_3$, $H_4$, $H_5$, $H_6$, and $H_7$ are variables corresponding to
		$\Bf$, $\nabla\Bh$, $\la^{1/2}\Bh$, $\Bh$, $\nabla\Bk$, $\la^{1/2}\Bk$, and $\Bk$, respectively.
		The norm of space $\CX_{\CR,q}(\dot\Om)$ is given by
			$\|(H_1,\dots,H_7)\|_{\CX_{\CR,q}(\dot\Om)} = \|(H_1,H_2,H_3)\|_{L_q(\dot\Om)} + \|H_4\|_{W_q^1(\dot\Om)} + \|(H_5,H_6)\|_{L_q(\Om_+)} + \|H_7\|_{W_q^1(\Om_+)}$.
	\item\label{rema:div2}
		If $\Bu$ satisfies \eqref{redu-eq:1} with $\Bf \in J_q(\Om)$, $<\jump{\Bh},\Bn>=0$ on $\Ga$, $<\Bk,\Bn_+>=0$ on $\Ga_+$, and $\la\in\Si_{\ep,\la_0}$,
		then $\di\Bu=0$ in $\dot\Om$. 
		This fact can be obtained in the same manner as in Subsection \ref{sub:stokes} with $g=0$.
		It then holds that $\Bu$ belongs to $J_q(\Om)$ 
		by Gauss's divergence theorem together with $\jump{\Bu}=0$ on $\Ga$, $\Bu=0$ on $\Ga_-$.
		Here and subsequently, we can see $J_q(\Om)$ as a closed subspace of $L_q(\dot\Om)^N$, that is,
		$J_q(\Om)$ are regarded as Banach spaces endowed with $\|\cdot\|_{L_q(\dot\Om)}$.
\end{enumerate}
\end{rema}

At this point, we introduce several propositions used throughout this paper.
The following two propositions are fundamental properties of the $\CR$-boundedness
(cf. \cite[Proposition 3.4]{DHP03}, \cite[Remark 3.2. (4)]{DHP03}).


\begin{prop}\label{prop:R}
\begin{enumerate}[$(1)$]
	\item \label{prop:R1}
		Let $X$ and $Y$ be Banach spaces, and let $\CT$ and $\CS$ be $\CR$-bounded families in $\CL(X,Y)$.
		Then $\CT+\CS=\{T+S\mid T\in\CT, \enskip S\in\CS\}$ is also $\CR$-bounded in $\CL(X,Y)$ and
			$\CR_{\CL(X,Y)}(\CT+\CS) \leq \CR_{\CL(X,Y)}(\CT) + \CR_{\CL(X,Y)}(\CS)$.
	\item \label{prop:R2}
		Let $X$, $Y$, and $Z$ be Banach spaces, and let $\CT$ and $\CS$ be $\CR$-bounded families in
		$\CL(X,Y)$ and $\CL(Y,Z)$, respectively.
		Then $\CS\CT=\{ST \mid S \in\CS, \enskip T\in\CT\}$ is also $\CR$-bounded in $\CL(X,Z)$ and
			$\CR_{\CL(X,Z)}(\CS\CT)\leq\CR_{\CL(X,Y)}(\CT)\CR_{\CL(Y,Z)}(\CS)$.
\end{enumerate}
\end{prop}

\begin{prop}\label{lemm:multi}
Let $1\leq q<\infty$.
Let $m(\la)$ be a bounded function defined on a subset $\La$ in the complex plane $\BC$,
and let $M_m(\la)$ be a multiplication operator with $m(\la)$ defined by $M_m(\la)f=m(\la)f$
for any $f\in L_q(G)$ with an open set $G$ of $\ws$. Then
	$\CR_{\CL(L_q(G))}(\{M_m(\la)\mid \la\in\La\}) \leq K_q^2\|m\|_{L_\infty(\La)}$,
where $K_q$ is a positive constant in Khintchine's inequality $($cf. also \cite[Theorem 2.4]{KW04}$)$.
\end{prop}

The next one is used to estimate terms arising from uniform $W_r^{2-1/r}$ domains,
for example unit normal vectors $\Bn$, $\Bn_+$.

\begin{prop}\label{lemm:Shi13}
	Let $1\leq q\leq r <\infty$ and $N<r<\infty$.
	Suppose that $\Om_\pm$ are uniform $W_r^{2-1/r}$ domains.
	Then there exists a positive constant $C_{N,q,r}$ such that,
	for any $\si>0$, $a\in L_r(\dot\Om)$, and $b\in W_q^1(\dot\Om)$,
	it holds the estimate:
	\begin{equation*}
		\|ab\|_{L_q(\dot\Om)} \leq \si\|\nabla b\|_{L_q(\dot\Om)}
		 + C_{N,q,r}\left(\si^{-\frac{N}{r-N}}\|a\|_{L_r(\dot\Om)}^{\frac{r}{r-N}} + \|a\|_{L_r(\dot\Om)}\right)\|b\|_{L_q(\dot\Om)}.
	\end{equation*}
\end{prop}

\begin{proof}
We first show the following inequality: For $q< s \leq \infty $ and $N(1/q-1/s)<1$,
\begin{equation}\label{160203_5}
	\|u\|_{L_s(\dot\Om)}\leq C_{N,q,r,s}\Big(\|\nabla u\|_{L_q(\dot\Om)}^{N\left(\frac{1}{q}-\frac{1}{s}\right)}
	\|u\|_{L_q(\dot\Om)}^{1-N\left(\frac{1}{q}-\frac{1}{s}\right)} + \|u\|_{L_q(\dot\Om)}\Big)
	\quad \text{for any $u\in W_q^1(\dot\Om)$}
\end{equation} 
with some positive constant $C_{N,q,r,s}$ independent of $u$.
To this end, let $E_\pm$ be  extension operators for $\Om_\pm$, introduced in \cite[Proposition 5.22]{AF03}\footnote[2]{
The book \cite[Proposition 5.22]{AF03} only considered bounded boundary,
but we can extend the result to uniform $W_r^{2-1/r}$ domains as mentioned in \cite[Remark 5.23 (1)]{AF03}.},
that is, 
	$\|E_\pm u_\pm\|_{W_q^l(\ws)} \leq C_{N,q,r}\|u_\pm\|_{W_q^l(\Om_\pm)}$ for $l=0,1$ and for any $u_\pm\in W_q^l(\Om_\pm)$,
respectively. These inequalities combined with Sobolev embedding inequality:
\begin{equation*}
	\|f\|_{L_s(\ws)} \leq C_{N,q,s}\|\nabla f\|_{L_q(\ws)}^{N\left(\frac{1}{q}-\frac{1}{s}\right)}\|f\|_{L_q(\ws)}^{1-N\left(\frac{1}{q}-\frac{1}{s}\right)}
\end{equation*}
with $q<s\leq \infty$ and $N(1/q-1/s)<1$ yield that
\begin{align*}
	&\|u_\pm\|_{L_s(\Om_\pm)} \leq \|E_\pm u_\pm\|_{L_s(\ws)}
	\leq C_{N,q,s}\|\nabla E_\pm u_\pm \|_{L_q(\ws)}^{N\left(\frac{1}{q}-\frac{1}{s}\right)}\|E_\pm u_\pm\|_{L_q(\ws)}^{1-N\left(\frac{1}{q}-\frac{1}{s}\right)} \\
	&\leq C_{N,q,r,s}\|u_\pm\|_{W_q^1(\Om_\pm)}^{N\left(\frac{1}{q}-\frac{1}{s}\right)}\|u_\pm\|_{L_q(\Om_\pm)}^{1-N\left(\frac{1}{q}-\frac{1}{s}\right)}, \nonumber 
\end{align*}
respectively. Let $u_\pm = u\chi_{\BR_\pm^N}$ for $u\in W_q^1(\dot\Om)$.
Then we have
\begin{align*}
	&\|u\|_{L_s(\dot\Om)}^q \leq 2^{q}(\|u_+\|_{L_s(\Om_+)}^{q}+\|u_-\|_{L_s(\Om_-)}^{q}) 
		\leq C_{N,q,r,s}(\|u_+\|_{W_q^1(\Om_+)}^{q/p_1}\|u_+\|_{L_q(\Om_+)}^{q/p_2}
		+\|u_-\|_{W_q^1(\Om_-)}^{q/p_1}\|u_-\|_{L_q(\Om_-)}^{q/p_2}),
\end{align*}
where we have set $1/p_1 = N(1/q-1/s)$ and $1/p_2 = 1- N(1/q-1/s)$.
We combine the last inequality with H\"older's inequality:
	$a_+ b_+ + a_-b_- \leq (a_+^{p_1} + a_-^{p_1})^{1/p_1}(b_+^{p_2} + b_-^{p_2})^{1/p_2}$
for $a_\pm = \|u_\pm\|_{W_q^1(\Om_\pm)}^{q/p_1}$ and $b_\pm = \|u_\pm\|_{L_q(\Om_\pm)}^{q/p_2}$, respectively, in order to obtain
\begin{align*}
	\|u\|_{L_s(\dot\Om)} \leq C_{N,q,r,s}\|u\|_{W_q^1(\dot\Om)}^{N\left(\frac{1}{q}-\frac{1}{s}\right)}
	\|u\|_{L_q(\dot\Om)}^{1-N\left(\frac{1}{q}-\frac{1}{s}\right)},
\end{align*}
which implies \eqref{160203_5}.
The required estimate of Proposition \ref{lemm:Shi13} follows from \eqref{160203_5}
in the same manner as in the proof of \cite[Lemma 2.4]{Shibata13}.
\end{proof}

We devote the last part of this subsection to the proof of Theorem \ref{theo:main}.

\begin{proof}[Proof of Theorem \ref{theo:main}]
We prove Theorem \ref{theo:main} under the assumption that Theorem \ref{theo:redu1} holds.

\noindent{\bf Step 1: Proof of \eqref{160124_1}, \eqref{160124_2}.}
It will be shown in Remark \ref{rema:normal} of Section \ref{sec:proof} below that
the unit normals $\Bn$, $\Bn_+$ can be regards as vector functions defined on $\ws$
and that, for any $f\in L_q(\dot\Om)$ and $g\in W_q^1(\dot\Om)$,
\begin{equation}\label{normal:est}
	\|f\nu\|_{L_q(\dot\Om)} \leq C\|f\|_{L_q(\dot\Om)}, \quad
	\|g\nabla\nu\|_{L_q(\dot\Om)} \leq C\|g\|_{W_q^1(\dot\Om)}, \quad 
	\|g\nu\|_{W_q^1(\dot\Om)} \leq C\|g\|_{W_q^1(\dot\Om)}
\end{equation}
with $\nu\in\{\Bn,\Bn_+\}$ and with some positive constant $C$.


Let $(\Bf,g,\Bh,\Bk)\in X_q$ and $\mathfrak{g}\in\CG(g)$.
Suppose that $(\Bf,\Bh,\Bk)$ satisfy \eqref{151117_1}\footnote[2]{As was discussed in Subsection \ref{sub:stokes},
it suffices to consider $(\Bf,\Bh,\Bk)$ satisfying \eqref{151117_1}.
In fact, we can extend it to any $(\Bf,\Bh,\Bk)\in X_{\CR,q}(\dot\Om)$, similarly to the proof of Step 1,
with the help of $\ka$ used in Subsection \ref{sub:stokes}.}.
Then, in view of Subsection \ref{sub:stokes} and Theorem \ref{theo:redu1}, we set
\begin{equation*}
	\Bu = \BB(\la)F_{\CR,\la}(\Bf+\rho^{-1}\nabla L (\la\mathfrak{g},g),\Bh+g\Bn,\Bk+g\Bn_+), \quad
	\te = K(\Bu) - L(\la\mathfrak{g},g)
\end{equation*}
to see that $(\Bu,\te)$ solves the problem \eqref{TSRP}. Here,
\begin{align*}
	&F_{\CR,\la}(\Bf+\rho^{-1}\nabla L (\la\mathfrak{g},g),\Bh+g\Bn,\Bk+g\Bn_+) 
	=\Big(\Bf+\rho^{-1}\nabla L (\la\mathfrak{g},g), \nabla\Bh+\nabla g\otimes\Bn+g\nabla\Bn, \\
	&\la^{1/2}\Bh+\la^{1/2}g\Bn,
	\Bh+g\Bn,\nabla\Bk + \nabla g\otimes\Bn_+ + g\nabla\Bn_+,\la^{1/2}\Bk+\la^{1/2}g\Bn_+,\Bk+g\Bn_+\Big).
\end{align*}
Thus, recalling Remark \ref{rema:main}, we define $\BA(\la)\BF$, $\BP(\la)\BF$ with $\BF=(F_1,\dots,F_{11})$ as follows:
\begin{align*}
	\BA(\la)\BF =& \BB(\la)\Big(F_1 + \rho^{-1}\nabla L(F_4,F_5),F_6+ F_2\otimes\Bn + F_5\nabla\Bn,F_7 + F_3\Bn, \\
		&\quad F_8+F_5\Bn, F_9 + F_2\otimes\Bn_+ + F_5\nabla\Bn_+, F_{10} + F_3\Bn_+, F_{11} + F_{5}\Bn_+\Big), \\
	\BP(\la)\BF =& K(\BA(\la)\BF) - L(F_4,F_5),
\end{align*}
which furnishes that $(\Bu,\te)=(\BA(\la)F_\la(\Bf,g,\mathfrak{g},\Bh,\Bk),\BP(\la)F_\la(\Bf,g,\mathfrak{g},\Bh,\Bk))$.

From now on, we show the estimates \eqref{160124_1}, \eqref{160124_2}.
By Theorem \ref{theo:redu1}, Proposition \ref{prop:R}, and \eqref{normal:est}, we easily have \eqref{160124_1}. 
To prove \eqref{160124_2}, we check the definition of $\CR$-boundedness. 
Let $n\in\BN$, $\{\la_j\}_{j=1}^n\subset\Si_{\ep,\la_0}$, and $\{\BF_j\}_{j=1}^n\subset\CX_{q}$ with $\BF_j = (F_{1j},\dots,F_{11j})$.
Since $\{\la(d/d\la)\}^l\nabla K(\BA(\la)\BF) = \nabla K(\{\la (d/d\la)\}^l\BA(\la)\BF)$ $(l=0,1)$, 
we have, by Proposition \ref{lemm:multi} and \eqref{160124_1},
\begin{align*}
	&\int_0^1\Big\|\sum_{j=1}^n r_j(u)\Big[\Big(\la\frac{d}{d\la}\Big)^l\nabla\BP(\la)\Big]_{\la=\la_j}\BF_j\Big\|_{L_q(\dot\Om)}^q\intd u \displaybreak[0] \\
	&\leq C_{\ga_0}\Big(\int_0^1\Big\|\sum_{j=1}^n r_j(u)\Big[\Big(\la\frac{d}{d\la}\Big)^l\BA(\la)\Big]_{\la=\la_j}\BF_j\Big\|_{W_q^2(\dot\Om)}^q\intd u
		+ \int_0^1\Big\|\sum_{j=1}^n r_j(u)(F_{4j},F_{5j})\Big\|_{L_q(\dot\Om)^N\times W_q^1(\dot\Om)}^q\intd u\Big) \displaybreak[0] \\
	&\leq C_{\ga_0}\Big\{\big(\la_0^{-q}+\la_0^{-q/2}+1\big)\int_0^1\Big\|\sum_{j=1}^n r_j(u)\BF_j\Big\|_{\CX_q}^q\intd u
		+ \int_0^1\Big\|\sum_{j=1}^n r_j(u)\BF_j\Big\|_{\CX_q}^q\intd u\Big\} \displaybreak[0] \\
	&\leq C_{\ga_0,\la_0}\int_0^1\Big\|\sum_{j=1}^n r_j(u)\BF_j\Big\|_{\CX_q}^q\intd u,
\end{align*}
which furnishes \eqref{160124_2}.

\noindent{\bf Step 2: Uniqueness.}
Let $\Bu\in W_q^2(\dot\Om)^N\cap J_q(\Om)$ and $\te=\te_1+\te_2\in W_q^1(\dot\Om)+\CW_q^1(\Om)$
satisfy 
\begin{equation}\label{150803_4}
	\left\{\begin{aligned}
		\la\Bu-\rho^{-1}\Di\BT(\Bu,\te) &= 0 && \text{in $\dot\Om$,} \\
		\jump{\BT(\Bu,\te)\Bn}&=0 && \text{on $\Ga$,} \\
		\jump{\Bu}&=0  && \text{on $\Ga$,} \\
		\BT(\Bu,\te)\Bn_+ &= 0 && \text{on $\Ga_+$,} \\
		\Bu &=0 && \text{on $\Ga_-$.}
	\end{aligned}\right.
\end{equation}
We prove that $\Bu=0$ in $\dot\Om$, which leads to the uniqueness of  \eqref{TSRP}.
To this end, it suffices to show that
\begin{equation}\label{151217_1}
	(\rho\Bu,\psi)_{\dot\Om}=0 \quad \text{for any $\psi\in C_0^\infty(\dot\Om)^N$}
\end{equation}
in what follows. In fact, it holds that $\Bu=0$ in $\Om_\pm$ if we choose $\psi\in C_0^\infty(\Om_\pm)^N$ in \eqref{151217_1}, respectively.

The assumption (b), stated in Theorem \ref{theo:main}, allows us to choose a $\ka\in \CW_{q'}^1(\Om)$ satisfying
\begin{equation*}
	(\rho^{-1}\nabla\ka,\nabla\ph)_{\dot\Om}
		= (\psi,\nabla\ph)_{\dot\Om} \quad \text{for any $\ph\in \CW_{q}^1(\Om)$.}
\end{equation*}
In addition, since the two-phase reduced Stokes resolvent equations \eqref{redu-eq:1} is solvable for $q'=q/(q-1)$,
we have a solution $\Bv\in W_{q'}^2(\dot\Om)^N$ to the equations:
\begin{equation*}
	\left\{\begin{aligned}
		\la\Bv - \rho^{-1}\Di\BT(\Bv,K(\Bv)) &= \psi-\rho^{-1}\nabla\ka && \text{in $\dot\Om$,} \\
		\jump{\BT(\Bv,K(\Bv))\Bn} &= 0 && \text{on $\Ga$,} \\
		\jump{\Bv}&=0 && \text{on $\Ga$,} \\
		\BT(\Bv,K(\Bv))\Bn_+ &= 0 && \text{on $\Ga_+$,} \\
		\Bv&=0 && \text{on $\Ga_-$.}
	\end{aligned}\right.
\end{equation*}
Then $\psi-\rho^{-1}\nabla\ka\in J_{q'}(\Om)$ implies that $\Bv\in J_{q'}(\Om)$ as was discussed in Remark \ref{rema:div} \eqref{rema:div2}.
Setting $K(\Bv)=w_1+w_2\in W_{q'}^1(\dot\Om)+\CW_{q'}^1(\Om)$,
we have, by Gauss's divergence theorem, $(\Bu,\nabla \ka)_{\dot\Om}=0$, and $(\Bu,\nabla w_2)_{\dot\Om}=0$,
\begin{align}\label{151009_6}
	&(\rho\Bu,\psi)_{\dot\Om} 
		=(\rho\Bu,\la\Bv-\rho^{-1}\Di\BT(\Bv,w_1+w_2+\ka))_{\dot\Om} \\
		&= \la(\rho\Bu,\Bv)_{\dot\Om} -(\Bu,\Di(\mu\BD(\Bv)))_{\dot\Om} + (\Bu,\nabla w_1)_{\dot\Om} \nonumber \\
	& = \la(\rho\Bu,\Bv)_{\dot\Om} +(\BD(\Bu),\mu\BD(\Bv))_{\dot\Om} -(\Bu,\jump{\mu\BD(\Bv)\Bn})_\Ga - (\Bu,\mu\BD(\Bv)\Bn_+)_{\Ga_+} \nonumber \\
		& -(\di\Bu,w_1)_{\dot\Om} + (\Bu,\jump{w_1\Bn})_{\Ga} + (\Bu,w_1\Bn_+)_{\Ga_+}. \nonumber  
\end{align}
Noting that $\jump{w_2}=0$ on $\Ga$ and $w_2=0$ on $\Ga_+$, we see that
$\jump{\mu\BD(\Bv)\Bn-w_1\Bn}=\jump{\mu\BD(\Bv)\Bn-K(\Bv)\Bn}=0$ on $\Ga$ and 
$\mu\BD(\Bv)\Bn-w_1\Bn=\mu\BD(\Bv)-K(\Bv)\Bn=0$ on $\Ga_+$.
In addition, it holds that $\di\Bu=0$ in $\dot\Om$, since
\begin{equation*}
	0= - (\Bu,\nabla\ph)_{\dot\Om} = (\di\Bu,\ph)_{\dot\Om} \quad \text{for any $\ph\in C_0^\infty(\dot\Om)$,}
\end{equation*}
where we have used $\Bu\in J_q(\Om)$ and the relation $C_0^\infty(\dot\Om)\subset\CW_{q'}^1(\Om)$.
Hence, \eqref{151009_6} implies that
\begin{equation}\label{151009_5}
	(\rho\Bu,\psi)_{\dot\Om} = \la(\rho\Bu,\Bv)_{\dot\Om}+(\BD(\Bu),\mu\BD(\Bv))_{\dot\Om}.
\end{equation}

On the other hand, it holds by the first equation of \eqref{150803_4} that $\la\rho\Bu-\Di\BT(\Bu,\te)=0$ in $\dot\Om$, which,
combined with Gauss's divergence theorem, furnishes that
\begin{align*}
	 0 &= (\la\rho\Bu-\Di\BT(\Bu,\te),\Bv)_{\dot\Om} \\
		&=\la(\rho\Bu,\Bv)_{\dot\Om} + (\mu\BD(\Bu),\BD(\Bv))_{\dot\Om}-(\jump{\mu\BD(\Bu)\Bn},\Bv)_{\Ga}
			-(\mu\BD(\Bu)\Bn_+,\Bv)_{\Ga_+} \\
			&-(\te_1,\di\Bv)_{\dot\Om} + (\jump{\te_1\Bn},\Bv)_\Ga + (\te_1\Bn_+,\Bv)_{\Ga_+} 
\end{align*} 
since $(\nabla\te_2,\Bv)_{\dot\Om}=0$ by $\Bv\in J_{q'}(\Om)$.
We thus obtain $\la(\rho\Bu,\Bv)_{\dot\Om}+(\mu\BD(\Bu),\BD(\Bv))_{\dot\Om}=0$ in the same manner as we have obtained \eqref{151009_5} from \eqref{151009_6}.
The last identity combined with \eqref{151009_5} implies \eqref{151217_1},
which completes the proof of the uniqueness.
\end{proof}

\subsection{Generation of analytic semigroup}\label{sub:anal} 
In this and the next subsection, we discuss time-dependent problems.
We now consider the following initial-boundary value problem:
\begin{equation}\label{IBVP:2}
	\left\{\begin{aligned}
		\pa_t\Bu -\rho^{-1}\Di\BT(\Bu,K(\Bu)) &= 0 && \text{in $\dot\Om\times(0,\infty)$,} \\
		\jump{\BT(\Bu,K(\Bu))\Bn} &= 0 && \text{on $\Ga\times(0,\infty)$,} \\
		\jump{\Bu}&= 0 && \text{on $\Ga\times(0,\infty)$,} \\
		\BT(\Bu,K(\Bu))\Bn_+ &= 0 && \text{on $\Ga_+\times(0,\infty)$,} \\
		\Bu &= 0 && \text{on $\Ga_-\times(0,\infty)$,} \\
		\Bu |_{t=0} &= \Bu_0 && \text{in $\dot\Om$.}
	\end{aligned}\right.
\end{equation}

To discuss the generation of analytic semigroup associated with \eqref{IBVP:2},
we formulate \eqref{IBVP:2} in the semigroup setting.
For this purpose, we introduce the Stokes operator $\CA$ and its domain $\CD_q(\CA)$ as follows:
\begin{align}\label{op:stokes}
	\CD_q(\CA) =& \{\Bu \in W_q^2(\dot\Om)^N\cap J_q(\Om)\mid \jump{\CT_\Bn(\mu\BD(\Bu)\Bn)}=0 \quad \text{on $\Ga$,} \\ 
			& \quad \jump{\Bu}=0 \quad \text{on $\Ga$,} \quad
			\CT_{\Bn_+}(\mu\BD(\Bu)\Bn_+) = 0 \quad \text{on $\Ga_+$,} \quad
			\Bu=0 \quad \text{on $\Ga_-$}\}, \nonumber \\
	\CA\Bu =& \rho^{-1}\Di\BT(\Bu,K(\Bu)) \quad \text{for $\Bu\in \CD_q(\CA)$}, \nonumber
\end{align} 
where we have set
$\CT_\Bn \Bf = \Bf - <\Bf,\Bn>\Bn$ and $\CT_{\Bn_+}\Bf = \Bf - <\Bf,\Bn_+>\Bn_+$
that are the tangential parts of $N$-vector $\Bf$ with respect to $\Bn$ and $\Bn_+$, respectively.
Then it is possible to rewrite \eqref{IBVP:2} as follows:
\begin{equation*}\label{150828_1}
	\pa_t\Bu -\CA\Bu = 0 \quad (t>0), \quad \Bu|_{t=0}=\Bu_0.
\end{equation*}
By Theorem \ref{theo:redu1}, the resolvent set $\rho(\CA)$ of $\CA$ contains $\Si_{\ep,\la_0}$.
In addition, denoting the resolvent operator of $\CA$ by $(\la-\CA)^{-1}$ and noting Remark \ref{rema:div} \eqref{rema:div2},
we see that, for any $\la\in\Si_{\ep,\la_0}$ and $\Bf\in J_q(\Om)$, $(\la-\CA)^{-1}\Bf=\BB(\la)(\Bf,0,0,0,0)\in J_q(\Om)$.
Since the $\CR$-boundedness of $\BB(\la)$ implies the usual boundedness, it holds that
\begin{equation*}
	\|(\la-\CA)^{-1}\|_{\CL(J_q(\Om))} \leq \frac{M_{\ep,\la_0}}{|\la|} \quad (\la\in\Si_{\ep,\la_0})
\end{equation*}
with some positive constant $M_{\ep,\la_0}$.
By this resolvent estimate, we have the following theorem.
\begin{theo}\label{theo:anal}
	Let $1<q<\infty$, $N<r<\infty$, and $\max(q',q)\leq r$ with $q' = q/(q-1)$.
	Let $\rho_\pm$ be positive constants.
	Suppose that the conditions $({\rm a})$, $({\rm b})$, and $({\rm c})$ stated in Theorem \ref{theo:main} hold.
	Then the Stokes operator $\CA$ generates a $C_0$-semigroup $\{T(t)\}_{t\geq0}$ on $J_q(\Om)$,
	which is analytic.
\end{theo}

\subsection{Maximal $L_p\text{-}L_q$ regularity}\label{sub:max}
Since the system \eqref{IBVP:1} is linear,
we consider the following two problems:
\begin{align}
	&\left\{\begin{aligned}
		\pa_t\Bu -\rho^{-1}\Di\BT(\Bu,\te) &= 0, &\di\Bu&=0 && \text{in $\dot\Om\times(0,\infty)$,} \\
		\jump{\BT(\Bu,\te)\Bn}&=0, &\jump{\Bu}&=0 && \text{on $\Ga\times(0,\infty)$,} \\
		\BT(\Bu,\te)\Bn_+ &= 0 && && \text{on $\Ga_+\times(0,\infty)$,} \\
		\Bu &= 0 && && \text{on $\Ga_-\times(0,\infty)$,} \\
		\Bu|_{t=0} &= \Bu_0 && && \text{in $\dot\Om$,}
	\end{aligned}\right.\label{150803_1} \displaybreak[0] \\
	&\left\{\begin{aligned}
		\pa_t\Bu -\rho^{-1}\Di\BT(\Bu,\te)&=\Bf, &\di\Bu &=g && \text{in $\dot\Om\times(0,\infty)$,} \\
		\jump{\BT(\Bu,\te)\Bn}&=\jump{\Bh}, &\jump{\Bu}&=0 && \text{on $\Ga\times(0,\infty)$,} \\
		\BT(\Bu,\te)\Bn_+&=\Bk && && \text{on $\Ga_+\times(0,\infty)$,} \\
		\Bu &= 0 && && \text{on $\Ga_-\times(0,\infty)$,} \\
		\Bu|_{t=0} &= 0 && && \text{in $\dot\Om$.}
	\end{aligned}\right.\label{150803_2}
\end{align}

To state maximal regularity theorems for \eqref{150803_1} and \eqref{150803_2},
we introduce several function spaces. For a Banach space $X$,
we denote the usual Lebesgue space and Sobolev space of $X$-valued functions defined on time interval $I$
by $L_p(I,X)$ and $W_p^m(I,X)$ with $m\in\BN$,
and their associated norms by $\|\cdot\|_{L_p(I,X)}$ and $\|\cdot\|_{W_p^m(I,X)}$, respectively.
We set for $\ga>0$
\begin{align*}
	&L_{p,\ga}(I,X) =
		\{f:I\to X \mid e^{-\ga t}f\in L_p(I,X)\}, \quad 
	L_{p,0,\ga}(\BR,X) =
		\{f\in L_{p,\ga}(\BR,X) \mid f(t)=0 \enskip \text{for $t<0$}\}, \displaybreak[0] \\
	&W_{p,\ga}^m(I,X) =
		\{f\in L_{p,\ga}(I,X) \mid e^{-\ga t}\pa_t^j f(t)\in L_p(I,X) \enskip (j=1,\dots,m)\}, \displaybreak[0] \\
	&W_{p,0,\ga}^m(\BR,X) =
		W_{p,\ga}^m(\BR,X) \cap L_{p,0,\ga}(\BR,X). 
\end{align*}
Let $\CL$, $\CL^{-1}$, $\CF$, and $\CF^{-1}$ denote the Laplace transform,
the Laplace inverse transform, the Fourier transform,
and the Fourier inverse transform, which are denoted by
\begin{align*}
	\CL[f](\la) &=\int_{-\infty}^\infty e^{-\la t}f(t)\intd t,
	&\CL^{-1}[g](t) &= \frac{1}{2\pi}\int_{-\infty}^\infty e^{\la t}g(\la)\intd\tau \quad (\la=\ga+i\tau), \displaybreak[0] \\
	\CF[f](\tau) &= \int_{-\infty}^\infty e^{-i\tau t}f(t)\intd t,
	&\CF^{-1}[g](t) &=\frac{1}{2\pi}\int_{-\infty}^\infty e^{i\tau t}g(\tau)\intd \tau.
\end{align*}
Note that $\CL[f](\la)=\CF[e^{-\ga t}f(t)](\tau)$ and $\CL^{-1}[g](t)=e^{\ga t}\CF^{-1}[g(\ga+i\tau)](t)$.
For any real number $s\geq 0$, let $H_{p,\ga}^s(\BR,X)$ be the Bessel potential space of order $s$ defined by
\begin{equation*}
	H_{p,\ga}^s(\BR,X) = 
		\{f\in L_{p,\ga}(\BR,X) \mid e^{-\ga t}(\La_\ga^s f)(t)\in L_p(\BR,X)\}, \quad
		(\La_\ga^s f)(t)=\CL^{-1}[\la^s\CL[f]](t).
\end{equation*}
We also set 	$H_{p,0,\ga}^s(\BR,X)=\{f\in H_{p,\ga}^s(\BR,X) \mid f(t)=0 \enskip \text{for $t<0$}\}$.
For solutions of problems \eqref{150803_1} and \eqref{150803_2},
$W_{q,p,\ga}^{2,1}(\dot\Om\times(0,\infty))$ and $W_{q,p,0,\ga}^{2,1}(\dot\Om\times\BR)$ are defined by
\begin{align*}
	&W_{q,p,\ga}^{2,1}(\dot\Om\times(0,\infty)) =
		W_{p,\ga}^1((0,\infty),L_q(\dot\Om)^N) \cap L_{p,\ga}((0,\infty),W_q^2(\dot\Om)^N), \displaybreak[0] \\
	&W_{q,p,0,\ga}^{2,1}(\dot\Om\times\BR) =
		W_{p,0,\ga}^1(\BR,L_q(\dot\Om)^N) \cap L_{p,0,\ga}(\BR,W_q^2(\dot\Om)^N).
\end{align*}

First we discuss a maximal $L_p\text{-}L_q$ regularity theorem for \eqref{150803_1}.
Setting $\Bu(t)=T(t)\Bu_0$ and $\te(t)=K(\Bu(t))$,
we see that $\di\Bu(t)=0$ in $\dot\Om$ for $t>0$ by $\Bu(t)\in J_q(\Om)$,  
and thus $\Bu(t)$ and $\te(t)$ satisfy \eqref{150803_1}.
Since $\{T(t)\}_{t\geq0}$ is analytic, we have, for some $\la_0\geq1$ and for any $t>0$,
\begin{equation*}
	\begin{aligned}
		\|T(t)\Bu_0\|_{J_q(\Om)} &\leq
			C_{q,\la_0} e^{\la_0 t}\|\Bu_0\|_{J_q(\Om)} && \text{for $\Bu_0\in J_q(\Om)$,} \\
		\|\pa_t T(t)\Bu_0\|_{J_q(\Om)} &\leq
			C_{q,\la_0} t^{-1}e^{\la_0 t}\|\Bu_0\|_{J_q(\Om)} && \text{for $\Bu_0\in J_q(\Om)$,} \\
		\|\pa_t T(t)\Bu_0\|_{J_q(\Om)} & \leq
			C_{q,\la_0} e^{\la_0 t}\|\Bu_0\|_{\CD_q(\CA)} && \text{for $\Bu_0\in \CD_q(\CA)$}
	\end{aligned}
\end{equation*}
with some positive constant $C_{q,\la_0}$.
We then obtain in the same manner as in \cite[Theorem 3.9]{SS08}
\begin{equation*}
	\|e^{-2\la_0 t}(\pa_t\Bu,\Bu,\nabla\Bu,\nabla^2\Bu)\|_{L_p((0,\infty),L_q(\dot\Om))}  
	\leq C_{p,q,\la_0}\|\Bu_0\|_{\CD_{q,p}^{2(1-1/p)}(\dot\Om)}
\end{equation*}
for some positive constant $C_{p,q,\la_0}$ with $1<p,q<\infty$,
where we have set $\CD_{q,p}^{2(1-1/p)}(\dot\Om) = (J_q(\Om),\CD_q(\CA))_{1-1/p,p}$
with real interpolation functor $(\cdot,\cdot)_{\te,p}$ ($0<\te<1$, $1<p<\infty$).
Then, the following theorem holds.

\begin{theo}\label{theo:max1}
	Let $1<p,q<\infty$, $N<r<\infty$, and $\max(q,q')\leq r$ with $q'=q/(q-1)$.
	Let $\rho_\pm$ be positive constants.
	Suppose that the conditions $({\rm a})$, $({\rm b})$, and $({\rm c})$ stated in Theorem \ref{theo:main} hold.
	Then we have the following two assertions: 
	\begin{enumerate}[$(1)$]
		\item 
			There exists a positive constant $\ga_0\geq1$ such that, for any $\Bu_0\in \CD_{q,p}^{2(1-1/p)}(\dot\Om)$,
			the problem \eqref{150803_1} admits a unique solution
				$(\Bu,\te)\in W_{q,p,\ga_0}^{2,1}(\dot\Om\times(0,\infty))\times L_{p,\ga_0}((0,\infty),W_q^1(\dot\Om)+\CW_q^1(\Om))$,
			which satisfies 
			\begin{align*}
				&\|e^{-\ga_0 t}(\pa_t\Bu,\Bu,\nabla\Bu,\nabla^2\Bu)\|_{L_p((0,\infty),L_q(\dot\Om))}  
				 + \|e^{-\ga_0 t}\nabla\te\|_{L_p((0,\infty),L_q(\dot\Om))}
				\leq C_{p,q,\ga_0}\|\Bu_0\|_{\CD_{q,p}^{2(1-1/p)}(\dot\Om)}
			\end{align*}
			with some positive constant $C_{p,q,\ga_0}$.
		\item
			There exists a positive constant $\ga_0\geq1$ such that, for any
			\begin{align*}
				&\Bf \in L_{p,0,\ga_0}(\BR,L_q(\dot\Om)^N), \quad
				g\in H_{p,0,\ga_0}^{1/2}(\BR,L_q(\dot\Om)^N) \cap L_{p,0,\ga_0}(\BR,W_q^1(\dot\Om)\cap \BW_q^{-1}(\Om))), \displaybreak[0] \\ 
				&\Bh \in H_{p,0,\ga_0}^{1/2}(\BR,L_q(\dot\Om)^N) \cap L_{p,0,\ga_0}(\BR,W_q^1(\dot\Om)^N), \displaybreak[0] \\
				&\Bk \in H_{p,0,\ga_0}^{1/2}(\BR,L_q(\Om_+)^N) \cap W_{p,0,\ga_0}^1(\BR,L_q(\Om_+)^N)
			\end{align*}
			and for any representative $\mathfrak{g}\in W_{p,0,\ga_0}^1(\BR,L_q(\dot\Om)^N)$ of $\CG(g)$,
			the problem \eqref{150803_2} a unique solution
			\begin{equation*}
				(\Bu,\te) \in W_{q,p,0,\ga_0}^{2,1}(\dot\Om\times\BR)\times L_{p,0,\ga_0}(\BR,W_q^1(\dot\Om)+\CW_{q}^1(\Om)),
			\end{equation*}
			which possesses the estimate:
			\begin{align}\label{151009_2}
				&\|e^{-\ga_0 t}(\pa_t\Bu,\Bu,\La_{\ga_0}^{1/2}\nabla\Bu,\nabla^2\Bu)\|_{L_p(\BR,L_q(\dot\Om))}
				+ \|e^{-\ga_0 t}\nabla \te\|_{L_p(\BR,L_q(\dot\Om))} 
				\leq C_{p,q,\ga_0}\CN_{p,q,\ga_0}(\Bf,g,\mathfrak{g},\Bh,\Bk) 
			\end{align}
			for some positive constant $C_{p,q,\ga_0}$ with
			\begin{align*}
				\CN_{p,q,\ga_0}(\Bf,g,\mathfrak{g},\Bh,\Bk) 
				 = &
					\|e^{-\ga_0 t}(\Bf,\nabla g,\La_{\ga_0}^{1/2}g,\pa_t\mathfrak{g}, \nabla\Bh,
					\La_{\ga_0}^{1/2}\Bh)\|_{L_p(\BR,L_q(\dot\Om))} +\|e^{-\ga_ 0 t}(g,\Bh)\|_{L_p(\BR,W_q^1(\dot\Om))} \\
					& +\|e^{-\ga_0 t}(\nabla\Bk,\La_{\ga_0}^{1/2}\Bk)\|_{L_p(\BR,L_q(\Om_+))} + \|e^{-\ga_0 t}\Bk\|_{L_p(\BR,W_q^1(\Om_+))}.
			\end{align*}
			In addition, if $g=0$, $\Bh=0$, and $\Bk=0$, then
			\begin{equation}\label{151009_1}
				\ga\|e^{-\ga t}\Bu\|_{L_p(\BR,L_q(\dot\Om))} \leq C_{p,q,\ga_0}\|e^{-\ga_0 t}\Bf\|_{L_p(\BR,L_q(\dot\Om))}
				\quad \text{for any $\ga\geq\ga_0$.}
			\end{equation}
	\end{enumerate}
\end{theo}

\begin{proof}
We prove the assertion (2).
Smooth functions having compact supports with respect to time variable are dense in the spaces for
$\Bf$, $g$, $\mathfrak{g}$, $\Bh$, and $\Bk$, so that
we may assume that $\Bf$, $g$, $\mathfrak{g}$, $\Bh$, and $\Bk$ are smooth and supported compactly with respect to time variable.
Applying the Laplace transform with respect to time $t\in\BR$ to \eqref{150803_2}, we have
\begin{equation*}
	\left\{\begin{aligned}
		\la\Bv -\rho^{-1}\Di\BT(\Bv,\pi) &= \CL[\Bf](\la), &\di\Bv&=\CL[g](\la) && \text{in $\dot\Om$,} \\
		\jump{\BT(\Bv,\pi)\Bn}&=\CL[\Bh](\la) &\jump{\Bv}&=0 && \text{on $\Ga$,} \\
		\BT(\Bv,\pi)\Bn_+&=\CL[\Bk](\la) && && \text{on $\Ga_+$,} \\
		\Bv &= 0 && && \text{on $\Ga_-$.}
	\end{aligned}\right.
\end{equation*}
On the other hand, we observe that
	$(\CL[g](\la),\ph)_\Om = -(\CL[\mathfrak{g}](\la),\nabla\ph)_\Om$ for all $\ph \in W_{q',\Ga_+}^1(\Om)$,
because $(g(t),\ph)_\Om = - (\mathfrak{g}(t),\nabla\ph)_\Om$ for $t\in\BR$ by \eqref{DI}.
This implies that $\CL[g](\la)\in\BW_q^{-1}(\Om)$ and $\CL[\mathfrak{g}](\la)\in \CG(\CL[g](\la))$, so that we define, in view of Theorem \ref{theo:main}, $\Bu$ and $\te$ by
\begin{equation*}
	\Bu = \CL^{-1}\left[\BA(\la)F_\la\left(\CL[\Bf],\CL[g],\CL[\mathfrak{g}],\CL[\Bh],\CL[\Bk]\right)\right], \quad
	\te = \CL^{-1}\left[\BP(\la)F_\la\left(\CL[\Bf],\CL[g],\CL[\mathfrak{g}],\CL[\Bh],\CL[\Bk]\right)\right].
\end{equation*}
Since we assume that $\Bf$, $g$, $\mathfrak{g}$, $\Bh$, and $\Bk$ are supported compactly,
it holds that $\CL[\Bf]$, $\CL[g]$, $\CL[\mathfrak{g}]$, $\CL[\Bh]$, and $\CL[\Bk]$ are holomorphic functions with respect to $\la$.
Thus $\Bu$ and $\te$ are defined independently of $\ga\geq\ga_0$ for $\la=\ga+i\tau$,
where $\ga_0$ is a positive number greater than $\la_0$ stated in Theorem \ref{theo:main}.
Then, 
\begin{align}\label{151216_1}
	&e^{-\ga_0 t}\left(\pa_t\Bu,\La_{\ga_0}^{1/2}\nabla\Bu,\nabla^2\Bu\right)
		= \CF^{-1}\left[R_{\mu_0}\BA(\mu_0)\CF[e^{-\ga_0 t}\BF]\right], \displaybreak[0] \\
	&e^{-\ga_0 t}\Bu
		= \CF^{-1}\left[\mu_0^{-1}\left(\mu_0\BA(\mu_0)\right)\CF[e^{-\ga_0 t}\BF]\right], \quad
	e^{-\ga_0 t}\nabla\te
		= \CF^{-1}\left[\nabla\BP(\mu_0)\CF[e^{-\ga_0 t}\BF]\right] \nonumber
\end{align}
with $\mu_0=\ga_0+i\tau$ and 
	$\BF= (\Bf,\nabla g,\La_{\ga_0}^{1/2}g,\pa_t\mathfrak{g},g,\nabla\Bh,\La_{\ga_0}^{1/2}\Bh,\Bh,\nabla\Bk,\La_{\ga_0}^{1/2}\Bk,\Bk)$,
which, combined with Weis's operator valued Fourier multiplier theorem (cf. \cite[Theorem 3.4]{Weis01})
together with Theorem \ref{theo:main} and Proposition \ref{lemm:multi}, 
allows us to conclude that the estimate \eqref{151009_2} holds.



Analogously, we can obtain the estimate \eqref{151009_1} if $\mu_0$ is replaced by $\la=\ga+i\tau$ $(\ga\geq\ga_0)$ in the second formula of \eqref{151216_1}.
Finally, \eqref{151009_1} combined with the argumentation used in \cite[Section 7]{Saito15b} furnishes that $\Bu(t)=0$, $\te(t)=0$ for $t<0$ and the uniqueness holds.
This completes the proof of Theorem \ref{theo:max1}.
\end{proof}

\section{Two-phase reduced Stokes resolvent equations in $\dws$}\label{sec:ws}
In this section, we discuss $\CR$-bounded solution operator families to the two-phase reduced Stokes resolvent equations
with an interface condition in $\dws=\BR_+^N\cup\BR_-^N$, 
that is, we consider the following resolvent problem with resolvent parameter $\la$
varying in $\Si_\ep=\{\la\in\BC\setminus\{0\}\mid |\arg\la|<\pi-\ep\}$:
\begin{equation}\label{reduce:1}
	\left\{\begin{aligned}
		\la\Bu-\rho^{-1}\Di\BT(\Bu,K_I(\Bu)) &= \Bf && \text{in $\dws$,} \\
		\jump{\BT(\Bu,K_I(\Bu))\Bn_0} &= \jump{\Bh} && \text{on $\BR_0^N$,} \\
		\jump{\Bu} &= 0 && \text{on $\BR_0^N$,}
	\end{aligned}\right.
\end{equation}
where $\Bn_0=(0,\dots,0,-1)^T$ 
and $\BT(\Bu,K_I(\Bu))=\mu\BD(\Bu)-K_I(\Bu)\BI$.
Here $\rho=\rho_+\chi_{\uhs}+\rho_-\chi_{\lhs}$ for positive constants $\rho_\pm$,
and suppose that
\begin{enumerate}[(d)]
	\item
		viscosity coefficient $\mu$ is given by $\mu=\mu_+\chi_{\uhs}+\mu_-\chi_{\lhs}$
		for positive constants $\mu_\pm$ satisfying $\mu_{\pm1}\leq \mu_\pm\leq \mu_{\pm2}$, respectively,
		where $\mu_{\pm1}$ and $\mu_{\pm2}$ are the same constants as in Theorem \ref{theo:main}.
\end{enumerate}
Furthermore, for $1<q<\infty$ and $q'=q/(q-1)$,
let $K_I(\Bu)$ be defined by $K_I(\Bu)=\CK(\al,\beta)$ with
\begin{equation*}
	\al = \rho^{-1}\Di(\mu\BD(\Bu))-\nabla\di\Bu, \quad \beta = <\jump{\mu\BD(\Bu)\Bn_0},\Bn_0>-\jump{\di\Bu} \quad \text{for $\Bu\in W_q^2(\dws)^N$,}
\end{equation*}
where $\CK(\al,\beta)$ is given in Remark \ref{rema:weakDN} \eqref{rema:weakDN2} with $\dot\Om=\dws$,
that is, $K_I(\Bu)$ is the unique solution to
\begin{align}\label{160125_1}
	&(\rho^{-1}\nabla K_I(\Bu),\nabla\ph)_{\dws}
		= (\rho^{-1}\Di(\mu\BD(\Bu))-\nabla\di\Bu,\nabla\ph)_{\dws} \quad
	\text{for all $\ph\in \wh{W}_{q'}^1(\ws)$,} \\
	&\jump{K_I(\Bu)} = <\jump{\mu\BD(\Bu)\Bn_0},\Bn_0>-\jump{\di\Bu} \quad \text{on $\BR_0^N$.} \label{160125_1-2}
\end{align}
Especially, we know that $\|\nabla K_I(\Bu)\|_{L_q(\dws)}\leq \ga_0 \|\nabla\Bu\|_{W_q^1(\dws)}$.
Here and hereafter, $\ga_0$ denotes a generic constant depending solely on
$N$, $q$, $\rho_+$, $\rho_-$, $\mu_{+1}$, $\mu_{+2}$, $\mu_{-1}$, and $\mu_{-2}$.

We will prove the following theorem in this section.
\begin{theo}\label{theo:ws}
	Let $1<q<\infty$, $0<\ep<\pi/2$, and $\rho_\pm$ be positive constants.  
	Suppose that the condition $({\rm d})$ holds.
	For any open set $G$ of $\ws$, let $Y_{\CR,q}(G)$ and $\CY_{\CR,q}(G)$ be defined as 
	\begin{align*}
		Y_{\CR,q}(G) 
			&= \{(\Bf,\Bh) \mid \Bf\in L_q(G)^N,\,\Bh\in W_q^1(G)^N\}, \\
		\CY_{\CR,q}(G)
			&= \{(H_1,H_2,H_3)\mid H_1,H_3\in L_q(G)^N,\,H_2\in L_q(G)^{N^2}\}.\nonumber
	\end{align*}
	Then there exists an operator family 
		$\BS_I(\la) \in \Hol(\Si_\ep,\CL(\CY_{\CR,q}(\dws),W_q^2(\dws)^N))$
	such that, for any $\la\in\Si_\ep$ and $(\Bf,\Bh)\in Y_{\CR,q}(\dws)$,
	$\Bu = \BS_I(\la)G_{\CR,\la}(\Bf,\Bh)$ is a unique solution to the problem \eqref{reduce:1}, and furthermore,
	\begin{equation*}
		\CR_{\CL(\CY_{\CR,q}(\dws),L_q(\dws)^{\wt{N}})}
		\Big(\Big\{\Big(\la \frac{d}{d\la}\Big)^l\left(R_\la \BS_I(\la)\right)\mid\la\in\Si_\ep\Big\}\Big) \leq \ga_1
		\quad(l=0,1).
	\end{equation*} 
	Here and subsequently, we set $\wt N = N^3+N^2+N$, $R_\la\Bu = (\nabla^2\Bu,\la^{1/2}\nabla\Bu,\la\Bu)$,
	$G_{\CR,\la}(\Bf,\Bh)=(\Bf,\nabla\Bh,\la^{1/2}\Bh)$
	and $\ga_1$ denotes a constant depending solely on $N$, $q$, $\ep$, $\rho_+$, $\rho_-$, $\mu_{+1}$, $\mu_{+2}$, $\mu_{-1}$, and $\mu_{-2}$.
\end{theo}

In view of Subsection \ref{sub:reduced},
it is sufficient to consider the two-phase Stokes resolvent equations in $\dws$:
\begin{equation}\label{150608_1}
	\left\{\begin{aligned}
		\la\rho\Bu - \Di(\mu\BD(\Bu))+\nabla\te &= \rho\Bf && \text{in $\dws$,} \\
		\di\Bu &= g && \text{in $\dws$,} \\
		\jump{(\mu\BD(\Bu)-\te\BI)\Bn_0} &= \jump{\Bh} && \text{on $\BR_0^N$,} \\
		\jump{\Bu} &= 0 && \text{on $\BR_0^N$.}
	\end{aligned}\right.
\end{equation}
Here, the Fourier transform $\CF$ and its inverse formula $\CF^{-1}$ are defined by
\begin{equation}\label{Fourier}
	\CF[f](\xi) = \int_\ws e^{-ix\cdot\xi}f(x)\intd x, \quad
	\CF^{-1}[g(\xi)](x) = \frac{1}{(2\pi)^N}\int_{\ws}e^{ix\cdot\xi}g(\xi)\intd \xi,
\end{equation}
respectively. We first consider the divergence equation: $\di\Bu = g$ in $\dws$.

\begin{lemm}\label{lemm:div}
	Let $1<q<\infty$. For $g\in W_q^1(\dws)\cap\BW_q^{-1}(\ws)$, we set
\begin{equation}\label{sol:div}
	V(g) = (V_1(g),\dots,V_N(g))^T, \quad V_j(g) = -\CF^{-1}\left[\frac{i\xi_j}{|\xi|^2}\CF[g](\xi)\right](x)\quad (j=1,\dots,N).
\end{equation}
Then $V(g)\in W_q^1(\ws)^N\cap W_q^2(\dws)^N$ and $\Bu=V(g)$ solves the divergence equation: $\di\Bu = g$ in $\dws$. 
In addition, there are operators
\begin{equation*}
	V^1 \in \CL(L_q(\dws)^N, L_q(\dws)^{N^3}), \quad
	V^2 \in \CL(L_q(\dws), L_q(\dws)^{N^2}), \quad
	V^3 \in \CL(\wh W_q^{-1}(\ws),L_q(\dws)^N)
\end{equation*}
such that $R_\la V(g) = (V^1(\nabla g),V^2(\la^{1/2}g), V^3(\la g))$,
where the dual space of $\wh{W}_{q'}^1(\ws)$ with $q'=q/(q-1)$ is written by
$\wh W_q^{-1}(\ws)$ endowed with norm $\|\cdot\|_{\wh{W}_q^{-1}(\ws)}$. 
\end{lemm}

\begin{proof}
It is clear that $\Bu=V(g)$ solves the divergence equation: $\di\Bu=g$ in $\dws$
and that by the Fourier multiplier theorem of Mikhlin (cf. \cite[Appendix, Theorem 2]{Mikhlin65})
\begin{equation*}
	\|\nabla V(g)\|_{L_q(\ws)}\leq \ga_0\|g\|_{L_q(\ws)}, \quad
	\|\pa_k\nabla V(g)\|_{L_q(\ws)} \leq \ga_0\|\pa_k g\|_{L_q(\ws)} \quad (k=1,\dots,N-1).
\end{equation*}
Since $\di V(g)=g$ in $\dws$, it holds that $\pa_N^2 V(g)=\pa_N g-\pa_N\sum_{k=1}^{N-1}\pa_k V(g)$ in $\dws$, 
which, combined with the last inequalities, furnishes that
	$\|\pa_N^2 V(g)\|_{L_q(\dws)} \leq \ga_0\|\nabla g\|_{L_q(\dws)}$.

Next we estimate $V(g)$. Let $\ph\in C_0^\infty(\ws)^N$, and then
	$(V(g),\ph)_{\ws}
		= -(g,\CF[|\xi|^{-2}i\xi\cdot\CF^{-1}[\ph](\xi)])_\ws$.
The Fourier multiplier theorem again yields that
\begin{equation*}
	|(V(g),\ph)_{\ws}|
		\leq \|g\|_{\wh W_q^{-1}(\ws)}\left\|\nabla\CF\left[\frac{i\xi\cdot\CF^{-1}[\ph](\xi)}{|\xi|^2}\right]\right\|_{L_{q'}(\ws)}
		\leq \ga_0\|g\|_{\wh W_q^{-1}(\ws)}\|\ph\|_{L_{q'}(\ws)},
\end{equation*}
which implies that $\|V(g)\|_{L_q(\ws)}\leq \ga_0\|g\|_{\wh W_q^{-1}(\ws)}$.
We thus see that $V(g)\in W_q^1(\ws)^N\cap W_q^2(\dws)^N$ and the existence of operators $V^i$ $(i=1,2,3)$.
This completes the proof of the lemma.
\end{proof}

Note that $\jump{V(g)}=0$ on $\bdry$ since $V(g)\in W_q^1(\ws)^N$ by Lemma \ref{lemm:div}.
Setting $\Bu = V(g)+\Bv$ in \eqref{150608_1} and noting $\Di(\mu\BD(\Bv))=\mu\De\Bv$ by the condition (d) and by $\di\Bv=0$ in $\dws$,
we have
\begin{equation}\label{150608_2}
	\left\{\begin{aligned}
		\rho\la\Bv-\mu\De\Bv+\nabla\te &= \wt\Bf && \text{in $\dws$,} \\
		\di\Bv &= 0 && \text{in $\dws$,} \\
		\jump{(\mu\BD(\Bv)-\te\BI)\Bn_0} &= \jump{\wt\Bh} && \text{on $\BR_0^N$,} \\
		\jump{\Bv} &= 0 && \text{on $\BR_0^N$,}
	\end{aligned}\right.
\end{equation}
where
	$\wt\Bf = \rho\Bf-\rho\la V(g)+\Di(\mu\BD(V(g)))$ and $\wt\Bh = \Bh - \mu\BD(V(g))\Bn_0$.

The following theorem was essentially proved in \cite[Theorem 1.1, Theorem 1.2]{SS11b},
but we again show them here from viewpoint of the existence of $\CR$-bounded solution operator families of \eqref{150608_2}.

\begin{theo}\label{theo:ws2}
	Let $1<q<\infty$, $0<\ep<\pi/2$, and $\rho_\pm$ be positive constants. Suppose that the condition $({\rm d})$ holds.
	Then there exists an operator family 
		$\CS_I(\la) \in \Hol(\Si_\ep,\CL(\CY_{\CR,q}(\dws),W_q^2(\dws)^N))$
	such that, for any $\la\in\Si_\ep$ and $(\wt\Bf,\wt\Bh)\in Y_{\CR,q}(\dws)$,
	$\Bv=\CS_I(\la)G_{\CR,\la}(\wt\Bf,\wt\Bh)$ is a unique solution to the problem \eqref{150608_2}
	with some pressure term $\te$. In addition, 
	\begin{equation*}
		\CR_{\CL(\CY_{\CR,q}(\dws),L_q(\dws)^{\wt{N}})}
		\Big(\Big\{\Big(\la\frac{d}{d\la}\Big)^l\left(R_\la\CS_I(\la)\right)\mid \la\in\Si_\ep\Big\}\Big)\leq \ga_1
		\quad(l=0,1).
	\end{equation*}
\end{theo}

\begin{proof}
{\bf Step 1: Reduction to $\wt\Bf=0$.}
We first reduce \eqref{150608_2} to the case $\wt\Bf=0$.
To this end, we consider problems in $\ws$ as follows: 
\begin{align*}
	\rho_+\la\psi_+-\mu_+\De\psi_+ + \nabla \ph_+ &= \wt\Bf,\quad \di\psi_+=0 \quad \text{in $\ws$,} \\
	\rho_-\la\psi_--\mu_-\De\psi_-+\nabla\ph_- &= \wt\Bf, \quad \di\psi_-=0 \quad \text{in $\ws$.} \nonumber
\end{align*}
Then we have the following solution formulas (cf. \cite[Section 2]{SS11b}):
\begin{equation*}
	\psi_\pm =
		\CA_{\pm}(\la)\wt\Bf:=
		\CF^{-1}\left[\frac{\CF[\wt\Bf](\xi)-|\xi|^{-2}\xi<\xi,\CF[\wt\Bf](\xi)>}{\rho_\pm\la+\mu_\pm|\xi|^2}\right](x), \quad
	\ph_\pm =
		-\CF^{-1}\left[\frac{<i\xi,\CF[\wt\Bf](\xi)>}{|\xi|^2}\right](x).
\end{equation*}
By \cite[Theorem 3.3, proof of Theorem 3.2]{ES13}, 
\begin{equation}\label{151010_1}
	\CR_{\CL(L_q(\ws)^N,L_q(\ws)^{\wt N})}
	\Big(\Big\{\Big(\la\frac{d}{d\la}\Big)^l\left(R_\la\CA_\pm(\la)\right)\mid\la\in\Si_\ep\Big\}\Big)\leq \ga_1 \quad (l=0,1).
\end{equation} 
Here, we set
\begin{equation}\label{sol:ws}
	\psi=\CA(\la)\wt\Bf:=(\CA_+(\la)\wt\Bf)\chi_{\uhs} + (\CA_-(\la)\wt\Bf)\chi_{\lhs}, \quad
	\ph = \ph_+\chi_{\uhs} + \ph_-\chi_{\lhs}.
\end{equation}
Note that $\jump{\ph}=0$ on $\BR_0^N$ and that 
$D^\al \CA(\la)\wt\Bf = (D^\al\CA_+(\la)\wt\Bf)\chi_{\uhs}+(D^\al\CA_-(\la)\wt\Bf)\chi_{\lhs}$ in $\dws$
for any multi-index $\al\in\BN_0^N$ with $|\al|\leq 2$.
Thus, by \eqref{151010_1}, Proposition \ref{prop:R}, 
and the definition of $\CR$-boundedness (cf. Definition \ref{defi:R}),
\begin{equation}\label{151010_3}
	\CR_{\CL(L_q(\dws)^N,L_q(\dws)^{\wt N})}
	\Big(\Big\{\Big(\la\frac{d}{d\la}\Big)^l\left(R_\la\CA(\la)\right)\mid\la\in\Si_\ep\Big\}\Big)\leq \ga_1 \quad (l=0,1),
\end{equation}
and also setting $\Bv=\CA(\la)\wt\Bf+\Bw$ and $\te = \ph+\ka$ in \eqref{150608_2} yields that
\begin{equation*}
	\left\{\begin{aligned}
		\rho\la\Bw -\mu\De\Bw+\nabla\ka &= 0 && \text{in $\dws$,} \\
		\di\Bw &= 0 && \text{in $\dws$,} \\
		\jump{(\mu\BD(\Bw)-\ka\BI)\Bn_0} &= \jump{\wt\Bh}-\jump{\mu\BD(\CA(\la)\wt\Bf)\Bn_0} && \text{on $\BR_0^N$,} \\
		\jump{\Bw} &= -\jump{\CA(\la)\wt\Bf} && \text{on $\BR_0^N$.}
	\end{aligned}\right.
\end{equation*}
To analyze this system, it is enough to consider the equations:
\begin{equation}\label{150608_5}
	\left\{\begin{aligned}
		\rho\la\Bu -\mu\De\Bu + \nabla\te &= 0 && \text{in $\dws$,} \\
		\di\Bu &= 0 && \text{in $\dws$,} \\
		\jump{(\mu\BD(\Bu)-\te\BI)\Bn_0} &= \jump{\Bh} && \text{on $\BR_0^N$,} \\
		\jump{\Bu} &= \jump{\Bk} && \text{on $\BR_0^N$}
	\end{aligned}\right.
\end{equation}
for given $\Bh = (h_1,\dots,h_N)^T\in W_q^1(\dws)^N$ and $\Bk = (k_1,\dots,k_N)^T\in W_q^2(\dws)^N$ with $k_N = -\psi_N$,
where $\psi_N$ is the $N$th component of $\psi$ defined as \eqref{sol:ws}.

\noindent{\bf Step 2: Solution formulas of \eqref{150608_5}.}
We rewrite \eqref{150608_5} as follows:
\begin{equation}\label{150608_6}
	\left\{\begin{aligned}
		\rho_\pm \la \Bu_\pm -\mu_\pm\De\Bu_\pm +\nabla\te_\pm &= 0 && \text{in $\BR_\pm^N$,} \\
		\di\Bu_\pm &=0 && \text{in $\BR_\pm^N$,} \\
		\jump{\mu(\pa_N u_j + \pa_j u_N)} &= - \jump{h_j} && \text{on $\BR_0^N$} \\
		\jump{2\mu \pa_N u_N-\te} &= -\jump{h_N} && \text{on $\BR_0^N$,} \\
		\jump{u_J} &= \jump{k_J} && \text{on $\BR_0^N$,}
	\end{aligned}\right.
\end{equation}
where $\Bu=(u_1,\dots,u_N)^T$, $\Bu_\pm = \Bu\chi_{\BR_\pm^N}$, and $\te_\pm = \te\chi_{\BR_\pm^N}$.
Here and subsequently, $j$ and $J$ run from $1$ to $N-1$ and $1$ to $N$, respectively, and we set
	$y' = (y_1,\dots,y_{N-1})$ for $y=(y_1,\dots,y_N)$.

Let $\wh{f}(\xi',x_N)$ and $\CF_{\xi'}^{-1}[g(\xi',x_N)](x')$
be the partial Fourier transform with respect to $x'$ and its inverse formula defined by
\begin{equation*}
	\wh{f}(\xi',x_N) = \int_{\BR^{N-1}}e^{-ix'\cdot\xi'}f(x',x_N)\intd x',\quad
	\CF_{\xi'}^{-1}[g(\xi',x_N)](x') = \frac{1}{(2\pi)^{N-1}}\int_{\BR^{N-1}}e^{ix'\cdot\xi'}g(\xi',x_N)\intd\xi'.
\end{equation*}
Apply the partial Fourier transform to \eqref{150608_6}, and we have
\begin{align}
	(\rho_\pm\la+\mu_\pm|\xi'|^2-\mu_\pm \pa_N^2)\wh{u}_{\pm j}(\xi',x_N) + i\xi_j\wh{\te}_\pm(\xi',x_N) &= 0, \quad \pm x_N>0, \label{ode:1} \displaybreak[0] \\
	(\rho_\pm\la+\mu_\pm|\xi'|^2-\mu_\pm \pa_N^2)\wh{u}_{\pm N}(\xi',x_N)+\pa_N\wh{\te}_\pm(\xi',x_N)&=0, \quad \pm x_N>0, \label{ode:2} \displaybreak[0] \\
	\sum_{j=1}^{N-1}i\xi_j\wh{u}_{\pm j}(\xi',x_N) + \pa_N\wh{u}_{\pm N}(\xi',x_N) &= 0, \quad \pm x_N>0, \label{ode:3} \displaybreak[0] \\
	\jump{\mu(i\xi_j\wh{u}_N+\pa_N\wh{u}_j)}(\xi',0) &= -\jump{\wh{h}_j}(\xi',0), \label{ode:4} \displaybreak[0] \\
	\jump{2\mu \pa_N\wh{u}_N-\wh{\te}}(\xi',0) &= -\jump{\wh{h}_N}(\xi',0), \label{ode:5} \displaybreak[0] \\
	\jump{\wh{u}_J}(\xi',0) &= \jump{\wh{k}_J}(\xi',0) \label{ode:6}.
\end{align}
Set
	$A=|\xi'|$ and $B_\pm = \sqrt{(\rho_\pm/\mu_\pm)\la+|\xi'|^2}$.
By \eqref{ode:1}-\eqref{ode:3}, we have $(\pa_N^2 - A^2)\wh{\te}_\pm (\xi',x_N)=0$ for $\pm x_N>0$,
and applying $\pa_N^2-A^2$ to \eqref{ode:1} and \eqref{ode:2} yields that
$(\rho_\pm\la+\mu_\pm|\xi'|^2-\mu_\pm \pa_N^2)(\pa_N^2-A^2)\wh{u}_{\pm J}(\xi',x_N)=0$
for $\pm x_N>0$.
Thus, we will look for solutions to \eqref{ode:1}-\eqref{ode:6} of the forms:
\begin{equation*}
	\wh{u}_{\pm J}(\xi',x_N) = \al_{\pm J}(e^{\mp Ax_N}-e^{\mp B_{\pm}x_N}) + \beta_{\pm J}e^{\mp B_\pm x_N},
	\quad \wh{\te}_\pm(\xi',x_N) = \ga_\pm e^{\mp Ax_N} \enskip(\pm x_N>0).
\end{equation*}
Inserting the above formulas into \eqref{ode:1}-\eqref{ode:6}, we have the following relations:
\begin{align}
	&\mu_\pm(B_\pm^2-A^2)\al_{\pm j}+i\xi_j\ga_\pm =0, \label{ode:7} \displaybreak[0] \\
	&\mu_\pm(B_\pm^2-A^2)\al_{\pm N}\mp A\ga_\pm =0 \label{ode:8} \displaybreak[0] \\
	&i\xi'\cdot\al_\pm'\mp A\al_{\pm N}=0, \quad
	-i\xi'\cdot\al_\pm'+i\xi'\cdot\beta_\pm'\pm B_\pm\al_{\pm N}\mp B_\pm\beta_{\pm N} =0, \label{ode:9} \displaybreak[0] \\
	&\mu_+\left\{i\xi_j\beta_{+N}+(-A+B_+)\al_{+j}-B_+\beta_{+j}\right\}
		-\mu_-\left\{i\xi_j\beta_{-N}+(A-B_-)\al_{-j}+B_-\beta_{-j}\right\} = -\jump{\wh{h}_j}(\xi',0), \label{ode:10} \displaybreak[0] \\
	&\left[2\mu_+\left\{(-A+B_+)\al_{+N}-B_+\beta_{+N}\right\}-\ga_+\right] -
		\left[2\mu_-\left\{(A-B_-)\al_{-N}+B_-\beta_{-N}\right\}-\ga_-\right] = -\jump{\wh{h}_N}(\xi',0), \label{ode:11} \displaybreak[0] \\
	&\beta_{+J}-\beta_{-J} = \jump{\wh{k}_J}(\xi',0), \label{ode:12}
\end{align}
where we have set $\al_\pm=(\al_{\pm 1},\dots,\al_{\pm N})$ and $\beta_\pm =(\beta_{\pm 1},\dots,\beta_{\pm N})$.

From now on, we solve the equations \eqref{ode:7}-\eqref{ode:12}.
First, we write $i\xi'\cdot\al_\pm'$, $\al_{\pm N}$, and $\ga_\pm$ by using $i\xi'\cdot\beta_\pm'$ and $\beta_{\pm N}$.
By \eqref{ode:9}, we have
\begin{equation}\label{alpha}
	\al_{\pm N} = \pm\frac{-i\xi'\cdot\beta_\pm'\pm B_\pm\beta_{\pm N}}{B_\pm -A}, \quad
	i\xi'\cdot\al_\pm' = \frac{A(-i\xi'\cdot\beta_\pm'\pm B_\pm\beta_{\pm N})}{B_\pm-A},
\end{equation}
which, combined with \eqref{ode:8}, furnishes that
\begin{equation}\label{gamma}
	\ga_\pm = \frac{\mu_\pm(B_\pm +A)}{A}(-i\xi'\cdot\beta_\pm'\pm B_\pm\beta_{\pm N}).
\end{equation}

Next, we give exact formulas of $\al_{\pm J}$ and $\beta_{\pm J}$.
By \eqref{ode:10} and \eqref{alpha},
\begin{align}\label{ode:13}
	&\mu_+\left\{-(B_++A)i\xi'\cdot\beta_+'+A(B_+-A)\beta_{+N}\right\} \\
		&\quad  -\mu_-\left\{(B_-+A)i\xi'\cdot\beta_-' +A(B_- -A)\beta_{-N}\right\}=-i\xi'\cdot\jump{\wh{\Bh}'}(\xi',0). \nonumber
\end{align}
In addition, by \eqref{ode:11}, \eqref{alpha}, and \eqref{gamma},
\begin{align}\label{ode:14}
	&\mu_+\left\{(B_+-A)i\xi'\cdot\beta_+' -B_+(B_++A)\beta_{+N}\right\} \\
		&\quad -\mu_-\left\{(B_--A)i\xi'\cdot\beta_-+B_-(B_-+A)\beta_{-N}\right\} = -A\jump{\wh{h}_N}(\xi',0). \nonumber
\end{align}
It holds by \eqref{ode:12} that
\begin{equation*}
	i\xi'\cdot\beta_-' = i\xi'\cdot\beta_+'-i\xi'\cdot\jump{\wh{\Bk}'}(\xi',0), \quad
	\beta_{-N} = \beta_{+N} -\jump{\wh{k}_N}(\xi',0),
\end{equation*}
which, inserted into \eqref{ode:13} and \eqref{ode:14}, furnishes that
\begin{align*}
	\left\{\mu_+(B_++A)+\mu_-(B_-+A)\right\}i\xi'\cdot\beta_+' 
		+\left\{-\mu_+A(B_+-A)+\mu_-A(B_--A)\right\}\beta_{+N} &= P(\Bh,\Bk), \\
	\left\{-\mu_+(B_+-A)+\mu_-(B_--A)\right\}i\xi'\cdot\beta_+'
		+\left\{\mu_+B_+(B_++A)+\mu_-B_-(B_-+A)\right\}\beta_{+N} &= Q(\Bh,\Bk),
\end{align*}
where
\begin{align*}
	P(\Bh,\Bk)&=i\xi'\cdot\jump{\wh{\Bh}'}(\xi',0)+\mu_-(B_-+A)i\xi'\cdot\jump{\wh{\Bk}'}(\xi',0) +\mu_-A(B_--A)\jump{\wh{k}_N}(\xi',0), \\
	Q(\Bh,\Bk)&=A\jump{\wh{h}_N}(\xi',0)+\mu_-(B_--A)i\xi'\cdot\jump{\wh{\Bk}'}(\xi',0)+\mu_-B_-(B_-+A)\jump{\wh{k}_N}(\xi',0). \nonumber
\end{align*}
We often denote $P(\Bh,\Bk)$ and $Q(\Bh,\Bk)$ by $P$ and $Q$ for short in the following.
Let
\begin{equation*}
	\BL=
	\begin{pmatrix}
		\mu_+(B_++A)+\mu_-(B_-+A) & -\mu_+A(B_+-A)+\mu_-A(B_--A) \\
		-\mu_+(B_+-A)+\mu_-(B_--A) & \mu_+B_+(B_++A)+\mu_-B_-(B_-+A)
	\end{pmatrix},
\end{equation*}
and then
\begin{align*}
	&\det \BL =-(\mu_+-\mu_-)^2A^3
		+\left\{(3\mu_+-\mu_-)\mu_+B_++(3\mu_--\mu_+)\mu_-B_-\right\}A^2 \\
		&\quad+\left\{(\mu_+B_++\mu_-B_-)^2+\mu_+\mu_-(B_++B_-)^2\right\}A
		+(\mu_+B_++\mu_-B_-)(\mu_+B_+^2+\mu_-B_-^2).
\end{align*}
The inverse matrix $\BL^{-1}$ of $\BL$ is given by
\begin{equation*}
	\BL^{-1}=\frac{1}{\det\BL}
		\begin{pmatrix}
			L_{11} & L_{12} \\
			L_{21} & L_{22}
		\end{pmatrix}
\end{equation*}
with
\begin{align}\label{L_ij}
	L_{11} &= \mu_+B_+(B_++A)+\mu_-B_-(B_-+A),
	&L_{12} &= \mu_+A(B_+-A)-\mu_-A(B_--A), \displaybreak[0] \\
	L_{21} &= \mu_+(B_+-A)-\mu_-(B_--A),
	&L_{22} &= \mu_+(B_++A)+\mu_-(B_-+A). \nonumber
\end{align}
Thus we have
\begin{align}\label{beta}
	&i\xi'\cdot\beta_+' = \frac{1}{\det\BL}\left(L_{11}P+L_{12}Q\right),\quad \beta_{+N} = \frac{1}{\det\BL}\left(L_{21}P+L_{22}Q\right), \\
	&i\xi'\cdot\beta_-' = \frac{1}{\det\BL}\left(L_{11}P+L_{12}Q\right)-i\xi'\cdot\jump{\wh{\Bk}'}(\xi',0), \quad 
	\beta_{-N} = \frac{1}{\det\BL}\left(L_{21}P+L_{22}Q\right)-\jump{\wh{k}_N}(\xi',0). \nonumber
\end{align}
These relations yields that
\begin{align*}
	\CF_+(\Bh,\Bk):=&-i\xi'\cdot\beta_+'+B_+\beta_{+N} =
		-\frac{1}{\det\BL}\left\{(L_{11}-B_+L_{21})P+(L_{12}-B_+L_{22})Q\right\}, \\
	\CF_-(\Bh,\Bk):=&-i\xi'\cdot\beta_-'-B_-\beta_{-N} \nonumber \\
		=& -\frac{1}{\det\BL}\left\{(L_{11}+B_-L_{21})P+(L_{12}+B_-L_{22})Q\right\}
		+i\xi'\cdot\jump{\wh{\Bk}'}(\xi',0) + B_-\jump{\wh{k}_N}(\xi',0), \nonumber
\end{align*}
which, inserted into \eqref{alpha} and \eqref{gamma}, furnishes that
\begin{equation}\label{150609_2}
	\al_{\pm N} = \pm \frac{\CF_{\pm}(\Bh,\Bk)}{B_\pm-A}, \quad \ga_\pm = \frac{\mu_\pm(B_\pm+A)\CF_{\pm}(\Bh,\Bk)}{A}.
\end{equation}
By \eqref{ode:7} and \eqref{150609_2}, we have
\begin{equation}\label{150609_3}
	\al_{\pm j} = -\frac{i\xi_j\CF_\pm(\Bh,\Bk)}{A(B_\pm-A)},
\end{equation}
and furthermore, by \eqref{ode:10} and \eqref{ode:12},
\begin{align*}
	&\mu_+B_+\beta_{+j}+\mu_-B_-\beta_{-j} =\jump{\wh{h}_j}(\xi',0) + \mu_-i\xi_j\jump{\wh{k}_N}(\xi',0) \\
		&\quad +
		\frac{(\mu_+-\mu_-)i\xi_j}{\det\BL}\left(L_{21}P+L_{22}Q\right)
		-\frac{i\xi_j}{A}\left(\mu_+\CF_+(\Bh,\Bk)+\mu_-\CF_-(\Bh,\Bk)\right), \\
	&\beta_{+j}-\beta_{-j} = \jump{\wh{k}_j}(\xi',0).
\end{align*}
The last relations imply that
\begin{align}\label{150609_4}
	\beta_{\pm j}=&
		\frac{1}{\mu_+B_++\mu_-B_-}\left(\jump{\wh{h}_j}(\xi',0) + \mu_-i\xi_j\jump{\wh{k}_N}(\xi',0)+
		\frac{(\mu_+-\mu_-)i\xi_j}{\det\BL}\left(L_{21}P+L_{22}Q\right)\right. \\
		&\left.-\frac{i\xi_j}{A}\left(\mu_+\CF_+(\Bh,\Bk)+\mu_-\CF_-(\Bh,\Bk)\right)
		\pm\mu_{\mp}B_\mp\jump{\wh{k}_j}(\xi',0)\right). \nonumber
\end{align}
By the symbols \eqref{beta}-\eqref{150609_4}, we can give solution formulas of \eqref{150608_5} as follows:
\begin{align}\label{sol-formula}
	u_{\pm J}(x) &= -\CF_{\xi'}^{-1}[\al_{\pm J}(B_\pm-A)\CM_\pm (\pm x_N)](x') + \CF_{\xi'}^{-1}[\beta_{\pm J}e^{\mp B_\pm x_N}](x'), \\
	\te_\pm(x) &= \CF_{\xi'}^{-1}[\ga_\pm e^{\mp A x_N}](x'), \quad
	\CM_\pm(a) = \frac{e^{- B_\pm a}-e^{- A a}}{B_\pm-A}. \nonumber
\end{align}

\noindent{\bf Step 3: Construction of solution operators for \eqref{sol-formula}.}
Setting
\begin{align*}
	&P'(\Bh,\Bk') = i\xi'\cdot\jump{\wh{\Bh}'}(\xi',0)+\mu_-(B_-+A)i\xi'\cdot\jump{\wh{\Bk}'}(\xi',0), \displaybreak[0] \\
	&Q'(\Bh,\Bk')= A\jump{\wh{h}_N}(\xi',0)+\mu_-(B_--A)i\xi'\cdot\jump{\wh{\Bk}'}(\xi',0), \displaybreak[0] \\
	&P_N(k_N) =\mu_-A(B_--A)\jump{\wh{k}_N}(\xi',0), \quad Q_N(k_N) = \mu_-B_-(B_-+A)\jump{\wh{k}_N}(\xi',0), \displaybreak[0] \\
	&\CF_+'(\Bh,\Bk') = -\frac{1}{\det\BL}
		\left\{(L_{11}-B_+L_{21})P'(\Bh,\Bk')+(L_{12}-B_+L_{22})Q'(\Bh,\Bk')\right\}, \displaybreak[0] \\
	&\CF_-'(\Bh,\Bk') = -\frac{1}{\det\BL}\left\{(L_{11}+B_-L_{21})P'(\Bh,\Bk')+(L_{12}+B_-L_{22})Q'(\Bh,\Bk')\right\}
		+i\xi'\cdot\jump{\wh{\Bk}'}(\xi',0), \displaybreak[0] \\
	&\CF_{+N}(k_N) = -
		\frac{\mu_-\jump{\wh{k}_N}(\xi',0)}{\det\BL}
		\left\{A(B_--A)(L_{11}-B_+L_{21})+B_-(B_-+A)(L_{12}-B_+ L_{22})\right\}, \displaybreak[0] \\
	&\CF_{-N}(k_N) = -\frac{\mu_-\jump{\wh{k}_N}(\xi',0)}{\det\BL}
		\left\{A(B_--A)(L_{11}+B_-L_{21})+B_-(B_-+A)(L_{12}+B_-L_{22})\right\}
		+B_-\jump{\wh{k}_N}(\xi',0),
\end{align*}
we see that
\begin{align*}
	&P(\Bh,\Bk) = P'(\Bh,\Bk') + P_N(k_N), \quad Q(\Bh,\Bk) = Q'(\Bh,\Bk') + Q_N(k_N), \displaybreak[0] \\
	&\CF_+(\Bh,\Bk) = \CF_+'(\Bh,\Bk') + \CF_{+N}(k_N), \quad \CF_-(\Bh,\Bk) = \CF_-'(\Bh,\Bk') + \CF_{-N}(k_N).
\end{align*}
We also define operators $S_{\pm J}(\la)$ and $T_{\pm J}(\la)$ by
\begin{align}\label{151012_1}
	&S_{\pm j}(\la)(\Bh,\Bk') = 
		\CF_{\xi'}^{-1}\left[\frac{e^{\mp B_\pm x_N}}{\mu_+B_++\mu_-B_-}\jump{\wh{h}_j}(\xi',0)\right](x') 
		\pm\CF_{\xi'}^{-1}\left[\frac{\mu_\mp B_\mp e^{\mp B_\pm x_N}}{\mu_+B_+ + \mu_-B_-}\jump{\wh{k}_j}(\xi',0)\right](x') \displaybreak[0] \\
		&\quad+ \CF_{\xi'}^{-1}\left[\left(\frac{i\xi_j}{A}\right)\frac{\CF_{\pm}'(\Bh,\Bk')}{A}A\CM_\pm(\pm x_N)\right](x') \nonumber  \displaybreak[0] \\
		&\quad+(\mu_+-\mu_-)\CF_{\xi'}^{-1}\left[\left(\frac{i\xi_j}{\mu_+B_++\mu_-B_-}\right)
			\frac{L_{21}P'(\Bh,\Bk')+L_{22}Q'(\Bh,\Bk')}{A\det\BL}
			A e^{\mp B_\pm x_N}\right](x') \nonumber \displaybreak[0] \\
		&\quad-\CF_{\xi'}^{-1}\left[\left(\frac{i\xi_j}{A}\right)
			\frac{\mu_+\CF_+'(\Bh,\Bk')+\mu_-\CF_-'(\Bh,\Bk')}{(\mu_+B_++\mu_-B_-)A}
			A e^{\mp B_\pm x_N}\right](x'), \nonumber \displaybreak[0] \\
	&S_{\pm N}(\la)(\Bh,\Bk') = \mp \CF_{\xi'}^{-1}
		\left[\frac{\CF_{\pm}'(\Bh,\Bk')}{A}A\CM_{\pm}(\pm x_N)\right](x') \nonumber \displaybreak[0] \\
		&\quad+\CF_{\xi'}^{-1}\left[\frac{L_{21}P'(\Bh,\Bk')+L_{22}Q'(\Bh,\Bk')}{A\det\BL} Ae^{\mp B_\pm x_N}\right](x'), \nonumber \displaybreak[0] \\
	&T_{\pm j}(\la)k_N = \CF_{\xi'}^{-1}\left[\left(\frac{i\xi_j}{A}\right)\frac{\CF_{\pm N}(k_N)}{A}A\CM_{\pm}(\pm x_N)\right](x') \nonumber \displaybreak[0] \\
	&\quad + \CF_{\xi'}^{-1}\left[\frac{\mu_- i\xi_j}{\mu_+B_+ + \mu_-B_-}e^{\mp B_\pm x_N}\jump{\wh{k}_N}(\xi',0)\right](x') \nonumber \displaybreak[0] \\
	&\quad + (\mu_+-\mu_-)\CF_{\xi'}^{-1}\left[\left(\frac{i\xi_j}{\mu_+B_+ + \mu_- B_-}\right)
		\frac{L_{21}P_N(k_N)+L_{22}Q_N(k_N)}{A\det\BL}Ae^{\mp B_\pm x_N}\right](x') \nonumber \displaybreak[0] \\
	&\quad -\CF_{\xi'}^{-1}\left[\left(\frac{i\xi_j}{A}\right)
		\frac{\mu_+\CF_{+N}(k_N)+\mu_-\CF_{-N}(k_N)}{(\mu_+B_+ + \mu_- B_-)A}Ae^{\mp B_\pm x_N}\right](x'), \nonumber \displaybreak[0] \\
	&T_{\pm N}(\la)k_N = \mp \CF_{\xi'}^{-1}\left[\frac{\CF_{\pm N}(k_N)}{A}A\CM_\pm(\pm x_N)\right](x')
		 +\CF_{\xi'}^{-1}\left[\frac{L_{21}P_N(k_N)+L_{22}Q_N(k_N)}{A\det\BL}Ae^{\mp B_\pm x_N}\right](x') \nonumber \displaybreak[0] \\
	&\quad +\left(\frac{\pm 1-1}{2}\right)\CF_{\xi'}^{-1}\left[e^{B_- x_N}\jump{\wh{k}_N}(\xi',0)\right](x'). \nonumber 
\end{align}
Then $u_{\pm J}=S_{\pm J}(\la)(\Bh,\Bk')+T_{\pm J}(\la)k_N$.

\noindent{\bf Step 4: $\CR$-boundedness of solution operator families \eqref{151012_1}.}
We show the $\CR$-boundedness of the operator families \eqref{151012_1}.
To this end, we introduce two classes of multipliers.
Let $0<\ep<\pi/2$ and $\ga_0\geq0$,
and let $m(\xi',\la)$ be a function defined on $(\BR^{N-1}\setminus\{0\})\times \Si_{\ep,\ga_0}$,
which is infinitely many times differentiable with respect to $\xi'\in\BR^{N-1}\setminus\{0\}$
and is holomorphic with respect to $\la\in\Si_{\ep,\ga_0}$.
Here we have set $\Si_{\ep,0}=\Si_\ep$.
If there exists a real number $s$ such that,
for any multi-index $\al'=(\al_1,\dots,\al_{N-1})\in\BN_0^{N-1}$
and $(\xi',\la)\in(\BR^{N-1}\setminus\{0\})\times\Si_{\ep,\ga_0}$, there hold the estimates:
\begin{equation*}
	|D_{\xi'}^{\al'}m(\xi',\la)|
		\leq C_{s,\al',\ep,\ga_0}(|\la|^{1/2}+A)^{s-|\al'|}, \quad
	\left|D_{\xi'}^{\al'}\left(\la\frac{d}{d\la}m(\xi',\la)\right)\right|
		\leq C_{s,\al',\ep,\ga_0}(|\la|^{1/2}+A)^{s-|\al'|}
\end{equation*}
with some positive constant $C_{s,\al',\ep,\ga_0}$,
then $m(\xi',\la)$ is called a multiplier of order $s$ with type $1$.
If there exists a real number $s$ such that, for any multi-index $\al'=(\al_1,\dots,\al_{N-1})\in\BN_0^{N-1}$
and $(\xi',\la)\in(\BR^{N-1}\setminus\{0\})\times\Si_{\ep,\ga_0}$,
there holds the estimates:
\begin{equation*}
	|D_{\xi'}^{\al'}m(\xi',\la)| \leq C_{s,\al',\ep,\ga_0}(|\la|^{1/2}+A)^s A^{-|\al'|}, \quad
	\left|D_{\xi'}^{\al'}\left(\la\frac{d}{d\la}m(\xi',\la)\right)\right|\leq C_{s,\al',\ep,\ga_0}(|\la|^{1/2}+A)^s A^{-|\al'|}
\end{equation*}
with some positive constant $C_{s,\al',\ep,\ga_0}$,
then $m(\xi',\la)$ is called a multiplier of order $s$ with type $2$.
In what follows, we denote the set of all multipliers defined on $(\BR^{N-1}\setminus\{0\})\times\Si_{\ep,\ga_0}$
of order $s$ with type $l$ ($l=1,2$) by $\BBM_{s,l,\ep,\ga_0}$.
We here give typical examples of multiplies as follows:
the Riesz kernel $\xi_j/|\xi'|$ ($j=1,\dots,N-1$) is a multiplier of order $0$ with type $2$.
Functions $\xi_j$ and $\la^{1/2}$ are multiplies of order $1$ with type $1$.
We also introduce the following two fundamental lemmas (cf. \cite[Lemma 4.6, Lemma 4.8]{SS11b}). 

\begin{lemm}\label{lemm:1}
Let $s_1,s_2\in\BR$, $0<\ep<\pi/2$, and $\ga_0\geq 0$.
\begin{enumerate}[$(1)$]
\item
	Given $m_i\in\BBM_{s_i,1,\ep,\ga_0}$ $(i=1,2)$, we have $m_1m_2\in\BBM_{s_1+s_2,1,\ep,\ga_0}$.
\item
	Given $l_i\in\BBM_{s_i,i,\ep,\ga_0}$ $(i=1,2)$, we have $l_1l_2\in\BBM_{s_1+s_2,2,\ep,\ga_0}$.
\item
	Given $n_i\in\BBM_{s_i,2,\ep,\ga_0}$ $(i=1,2)$, we have $n_1n_2\in\BBM_{s_1+s_2,2,\ep,\ga_0}$.
\end{enumerate}
\end{lemm}


\begin{lemm}\label{lemm:2}
Let $s\in\BR$ and $0<\ep<\pi/2$. Then the following assertions hold:
\begin{enumerate}[$(1)$]
\item
$B_\pm^s\in\BBM_{s,1,\ep,0}$, $(A+B_\pm)^s\in\BBM_{s,2,\ep,0}$, and $(\det\BL)^s\in\BBM_{3s,2,\ep,0}$.
\item
$A^s\in\BBM_{s,2,\ep,0}$, provided that $s\geq0$.
\item
For real numbers $a$, $b$ satisfying $a+b>0$, we have $(aB_+ + bB_-)^s\in\BBM_{s,1,\ep,0}$.
\item
$L_{11}$, $L_{12}$, $L_{21}$, and $L_{22}$ defined as \eqref{L_ij} satisfy
$L_{11}, L_{12}\in \BBM_{2,2,\ep,0}$ and $L_{21}, L_{22}\in\BBM_{1,2,\ep,0}$.
\end{enumerate}
\end{lemm}


We start with the following lemma to show the $\CR$-boundedness of the operators $S_{\pm J}(\la)$, $T_{\pm J}(\la)$.

\begin{lemm}\label{lemm:Rbound1}
	Let $0<\ep<\pi/2$, $\ga_0\geq0$, and $1<q<\infty$. Given multiplies
	\begin{equation*}
		m_1\in\BBM_{-1,1,\ep,\ga_0}, \quad m_2\in \BBM_{-2,2,\ep,\ga_0}, \quad
		m_3\in\BBM_{-1,2,\ep,\ga_0}, \quad m_4 \in \BBM_{0,1,\ep,\ga_0}, \quad m_5\in \BBM_{0,2,\ep,\ga_0},
	\end{equation*}
 	we define operators $K_{\pm i}(\la)$ on $W_q^1(\dws)$ and $L_{\pm i}(\la)$ on $W_q^2(\dws)$ $(i=1,2,3)$ by the formulas:
	\begin{align*}
		[K_{\pm1}(\la)f](x)
			&= \CF_{\xi'}^{-1}\left[m_1(\xi',\la)e^{\mp B_\pm x_N}\jump{\wh f}(\xi',0)\right](x'), \displaybreak[0] \\
		[K_{\pm2}(\la)f](x)
			&= \CF_{\xi'}^{-1}\left[m_2(\xi',\la) A e^{\mp B_\pm x_N}\jump{\wh f}(\xi',0)\right](x'), \displaybreak[0] \\
		[K_{\pm3}(\la)f](x)
			&= \CF_{\xi'}^{-1}\left[m_3(\xi',\la) A \CM_\pm(\pm x_N)\jump{\wh f}(\xi',0)\right](x'), \displaybreak[0] \\
		[L_{\pm 1}(\la)g](x)
			&= \CF_{\xi'}^{-1}\left[m_4(\xi',\la) e^{\mp B_\pm x_N}\jump{\wh g}(\xi',0)\right](x'), \displaybreak[0] \\
		[L_{\pm2}(\la)g](x)
			&= \CF_{\xi'}^{-1}\left[m_3(\xi',\la) A e^{\mp B_\pm x_N}\jump{\wh g}(\xi',0)\right](x'), \displaybreak[0] \\
		[L_{\pm3}(\la)g](x)
			&= \CF_{\xi'}^{-1}\left[m_5(\xi',\la) A \CM_\pm(\pm x_N)\jump{\wh g}(\xi',0)\right](x')
	\end{align*}
	for $\pm x_N>0$ and $\la\in\Si_{\ep,\ga_0}$.
	Then there exist operator families $\wt K_{\pm i}(\la)$, $\wt{L}_{\pm i}(\la)$ with
	\begin{equation*}
		\wt K_{\pm i}(\la) \in \Hol(\Si_{\ep,\ga_0},\CL(L_q(\dws)^{N+1},W_q^2(\BR_\pm^N))), \quad 
		\wt L_{\pm i}(\la) \in \Hol(\Si_{\ep,\ga_0},\CL(L_q(\dws)^{N^2+N+1},W_q^2(\BR_\pm^N)))
	\end{equation*}
	such that
		$K_{\pm i}(\la)f = \wt {K}_{\pm i}(\la)(\nabla f,\la^{1/2}f)$, 
		$L_{\pm j}(\la)g = \wt{L}_{\pm i}(\la)(\nabla^2 g,\la^{1/2}\nabla g,\la g)$, and
	\begin{align*}
		\CR_{\CL(L_q(\dws)^{N+1},L_q(\BR_\pm^N)^{N^2+N+1})}
			\Big(\Big\{\Big(\la\frac{d}{d\la}\Big)^l\Big(R_\la\wt K_{\pm i}(\la)\Big)
			\mid \la\in \Si_{\ep,\ga_0}\Big\}\Big) &\leq \ga_1, \\
		\CR_{\CL(L_q(\dws)^{N^2+N+1},L_q(\BR_\pm^N)^{N^2+N+1})}
			\Big(\Big\{\Big(\la\frac{d}{d\la}\Big)^l\Big(R_\la\wt L_{\pm i}(\la)\Big)
			\mid \la\in \Si_{\ep,\ga_0}\Big\}\Big) &\leq \ga_1
	\end{align*}
	for $l=0,1$ and $i=1,2,3$.
\end{lemm}

\begin{proof}
It was essentially proved in \cite[Lemma 5.1, 5.2, 5.3, 5.4]{SS11b}.
\end{proof}

Lemma \ref{lemm:Rbound1} enables us to obtain the  following lemma. 

\begin{lemm}\label{lemm:Rbound2}
	Let $0<\ep<\pi/2$ and $1<q<\infty$.
	Given a multiplier $m_0\in\BBM_{0,2,\ep,0}$, we define operators $\CK_{\pm i}(\la)$
	on $W_q^1(\dws)^N\times W_q^2(\dws)^{N-1}$ $(i=1,2,3)$ by the formulas:
	\begin{align*}
		[\CK_{\pm 1}(\la)(\Bh,\Bk')](x)&=
			\CF_{\xi'}^{-1}\left[m_0(\xi',\la)\frac{L_{21}P'(\Bh,\Bk')+L_{22}Q'(\Bh,\Bk')}{A\det\BL}A e^{\mp B_\pm x_N}\right](x'), \displaybreak[0] \\
		[\CK_{\pm 2}(\la)(\Bh,\Bk')](x)&=
			\CF_{\xi'}^{-1}\left[m_0(\xi',\la)\frac{\CF_\pm'(\Bh,\Bk')}{A}A\CM_\pm (\pm x_N)\right](x'), \displaybreak[0] \\
		[\CK_{\pm 3}(\la)(\Bh,\Bk')](x)&=
			\CF_{\xi'}^{-1}\left[m_0(\xi',\la)\frac{\CF_\pm'(\Bh,\Bk')}{(\mu_+ B_+ + \mu_- B_-)A}Ae^{\mp B_\pm x_N}\right](x')
	\end{align*}
	for $\pm x_N>0$ and $\la\in\Si_\ep$.
	Then there exist operator families 
		$\wt\CK_{\pm i}(\la)\in\Hol(\Si_\ep,\CL(L_q(\dws)^\CN,W_q^2(\BR_\pm^N)))$
	such that
		$\CK_{\pm i}(\la)(\Bh,\Bk') = \wt\CK_{\pm i}(\la)(\nabla\Bh,\la^{1/2}\Bh,\nabla^2\Bk',\la^{1/2}\nabla\Bk',\la\Bk')$
	and
	\begin{equation*}
		\CR_{\CL(L_q(\dws)^{\CN},L_q(\BR_\pm^N)^{N^2+N+1})}
		\Big(\Big\{\Big(\la\frac{d}{d\la}\Big)^l\Big(R_\la\wt\CK_{\pm i}(\la)\Big)\mid \la\in\Si_\ep\Big\}\Big) \leq \ga_1
	\end{equation*}
	for $l=0,1$ and $i=1,2,3$, where $\CN=N^2+N+N^2(N-1)+N(N-1)+(N-1)$.
\end{lemm}

\begin{proof}
We only show the case $\CK_{\pm 1}(\la)$. Note that
\begin{align*}
	&[\CK_{\pm 1}(\la)(\Bh,\Bk')](x)
		= \CF_{\xi'}^{-1}\left[m_0(\xi',\la)\frac{L_{21}}{\det\BL}Ae^{\mp B_\pm x_N}\frac{i\xi'}{A}\cdot\jump{\wh{\Bh}'}(\xi',0)\right](x') \\
	&\quad + \CF_{\xi'}^{-1}\left[m_0(\xi',\la)\frac{L_{22}}{\det\BL}Ae^{\mp B_\pm x_N}\jump{\wh{h}_N}(\xi',0)\right](x') \\
	&\quad + \mu_- \CF_{\xi'}^{-1}\left[m_0(\xi',\la)\frac{(B_-+A)L_{21}+(B_--A)L_{22}}{\det\BL}Ae^{\mp B_\pm x_N}
		\frac{i\xi'}{A}\cdot\jump{\wh{\Bk}'}(\xi',0)\right](x') \\
	&=:[\CK_{\pm 1}^1(\la)\Bh'](x)+[\CK_{\pm 1}^2(\la)h_N](x) + [\CK_{\pm 1}^3(\la)\Bk'](x).
\end{align*}
By Lemma \ref{lemm:1} and Lemma \ref{lemm:2},
\begin{equation*}
	m_0\frac{L_{21}}{\det\BL},\,m_0\frac{L_{22}}{\det\BL}\in\BBM_{-2,2,\ep,0}, \quad
	m_0\frac{(B_-+A)L_{21}+(B_--A)L_{22}}{\det\BL}\in\BBM_{-1,2,\ep,0},
\end{equation*}
which, combined with Lemma \ref{lemm:Rbound1}, furnishes that
there exist operator families $\wt{\CK}_{\pm 1}^i(\la)$ ($i=1,2,3$) with
\begin{align*}
	\wt{\CK}_{\pm 1}^1(\la)&\in\Hol(\Si_\ep,\CL(L_q(\dws)^{(N-1)(N+1)},W_q^2(\BR_\pm^N))), \\
	\wt{\CK}_{\pm 1}^2(\la)&\in\Hol(\Si_\ep,\CL(L_q(\dws)^{N+1},W_q^2(\BR_\pm^N))), \\
	\wt{\CK}_{\pm 1}^3(\la)&\in\Hol(\Si_\ep,\CL(L_q(\dws)^{(N-1)(N^2+N+1)},W_q^2(\BR_\pm^N)))
\end{align*}
such that
\begin{align*}
	\CK_{\pm 1}^1(\la)\Bh' &= \wt\CK_{\pm 1}^1(\la)(\nabla\Bh',\la^{1/2}\Bh'), \quad
	\CK_{\pm 1}^2(\la)h_N = \wt\CK_{\pm 1}^2(\la)(\nabla h_N,\la^{1/2} h_N), \\
	\CK_{\pm 1}^3(\la)\Bk' &= \wt\CK_{\pm 1}^3(\la)(\nabla^2\Bk',\la^{1/2}\nabla\Bk',\la\Bk')
\end{align*}
and
\begin{align*}
	\CR_{\CL(L_q(\dws)^{(N-1)(N+1)},L_q(\BR_\pm^N)^{N^2+N+1})}
	\Big(\Big\{\Big(\la\frac{d}{d\la}\Big)^l\Big(R_\la\wt\CK_{\pm 1}^1(\la)\Big)\mid \la\in\Si_\ep\Big\}\Big)&\leq\ga_1, \displaybreak[0] \\
	\CR_{\CL(L_q(\dws)^{N+1},L_q(\BR_\pm^N)^{N^2+N+1})}
	\Big(\Big\{\Big(\la\frac{d}{d\la}\Big)^l\Big(R_\la\wt\CK_{\pm 1}^2(\la)\Big)\mid \la\in\Si_\ep\Big\}\Big)&\leq\ga_1, \displaybreak[0] \\
	\CR_{\CL(L_q(\dws)^{(N-1)(N^2+N+1)},L_q(\BR_\pm^N)^{N^2+N+1})}
	\Big(\Big\{\Big(\la\frac{d}{d\la}\Big)^l\Big(R_\la\wt\CK_{\pm 1}^3(\la)\Big)\mid \la\in\Si_\ep\Big\}\Big)&\leq\ga_1
\end{align*}
for $l=0,1$. Thus setting
\begin{align*}
	&\wt\CK_{\pm1}(\la)(\nabla\Bh,\la^{1/2}\Bh,\nabla^2\Bk',\la^{1/2}\nabla\Bk',\la\Bk') \\
		& =\wt\CK_{\pm 1}^1(\la)(\nabla\Bh',\la^{1/2}\Bh') + \wt\CK_{\pm 1}^2(\la)(\nabla h_N,\la^{1/2} h_N)
		+\wt\CK_{\pm 1}^3(\la)(\nabla^2\Bk',\la^{1/2}\nabla\Bk',\la\Bk')
\end{align*}
implies, by Proposition \ref{prop:R}, 
that we have obtained the required operator $\wt\CK_{\pm1}(\la)$ of Lemma \ref{lemm:Rbound2}.
\end{proof}

To treat $T_{\pm J}(\la)$, we use the following lemma.

\begin{lemm}\label{lemm:Rbound3}
	Let $0<\ep<\pi/2$ and $1<q<\infty$.
	Suppose that $k_N$ is given by $k_N=-\psi_N$,
	where $\psi_N$ is the $N$th component of $\psi=\CA(\la)\wt\Bf$ $(\la\in\Si_\ep)$ defined as \eqref{sol:ws}.
	Given a multiplier $m_0\in\BBM_{0,2,\ep,0}$, we define operators $\CK_{\pm i}(\la)$ $(i=4,5,6)$ by the formulas:
	\begin{align*}
		[\CK_{\pm 4}(\la)k_N](x)
			&= \CF_{\xi'}^{-1}\left[m_0(\xi',\la)\frac{L_{21}P_N(k_N)+L_{22}Q_N(k_N)}{A\det\BL}Ae^{\mp B_\pm x_N}\right](x'), \\
		[\CK_{\pm 5}(\la)k_N](x)
			&= \CF_{\xi'}^{-1}\left[m_0(\xi',\la)\frac{\CF_{\pm N}(k_N)}{A} A \CM_\pm (\pm x_N)\right](x'), \\
		[\CK_{\pm 6}(\la)k_N](x)
			&= \CF_{\xi'}^{-1}\left[m_0(\xi',\la)\frac{\CF_{\pm N}(k_N)}{(\mu_+B_+ + \mu_-B_-)A} A e^{\mp B_\pm x_N}\right](x')
	\end{align*}
	for $\pm x_N>0$ and $\la\in\Si_\ep$.
	Then there exist operator families 
		$\wt\CK_{\pm i}(\la)\in \Hol(\Si_\ep,\CL(L_q(\dws)^N,W_q^2(\BR_\pm^N)))$
	such that
		$\CK_{\pm i}(\la)k_N = \wt\CK_{\pm i}(\la)\wt\Bf$ and
	\begin{equation*}
		\CR_{\CL(L_{q}(\dws)^N,L_q(\BR_\pm^N)^{N^2+N+1})}
		\Big(\Big\{\Big(\la\frac{d}{d\la}\Big)^l\Big(R_\la\wt\CK_{\pm i}(\la)\Big)\mid\la\in\Si_\ep\Big\}\Big) \leq \ga_1
	\end{equation*}
	for $l=0,1$ and $i=4,5,6$.
\end{lemm}

\begin{proof}
We only consider the case $\CK_{\pm 4}(\la)$.
First, we give some special formula of $\jump{\wh k_N}=-(\wh{\psi}_{+N}(\xi',0)-\wh{\psi}_{-N}(\xi',0))$. 
Let $\wt\Bf = (\wt{f}_1,\dots,\wt{f}_N)^T$.
Since
\begin{equation*}
\psi_{\pm N}(x)
	= \CF_{\xi}^{-1}\left[\frac{A^2}{|\xi|^2(\rho_\pm\la+\mu_\pm|\xi|^2)}\CF[\wt{f}_N](\xi)\right](x')
	-\sum_{j=1}^{N-1}\CF_{\xi}^{-1}\left[\frac{\xi_N\xi_j}{|\xi|^2(\rho_\pm\la+\mu_\pm|\xi|^2)}\CF[\wt{f}_j](\xi)\right](x),
\end{equation*}
it holds that
\begin{align*}
\wh{\psi}_{\pm N}(\xi',x_N) =&\int_{-\infty}^\infty A^2\wh{\wt{f}}_N(\xi',y_N)
	\left(\frac{1}{2\pi}\int_{-\infty}^\infty\frac{e^{i(x_N-y_N)\xi_N}}{|\xi|^2(\rho_\pm\la+\mu_\pm|\xi|^2)}\intd \xi_N\right)\intd y_N \displaybreak[0] \\
	&-\sum_{j=1}^{N-1}\int_{-\infty}^\infty\xi_j\wh{\wt{f}_j}(\xi',y_N)
	\left(\frac{1}{2\pi}\int_{-\infty}^\infty\frac{\xi_N e^{i(x_N-y_N)\xi_N}}{|\xi|^2(\rho_\pm\la+\mu_\pm|\xi|^2)}\intd \xi_N\right)\intd y_N.
\end{align*}
On the other hand, we have, by the residue theorem,
\begin{align*}
	\frac{1}{2\pi}\int_{-\infty}^\infty\frac{e^{ia\xi_N}}{|\xi|^2(\rho_\pm\la+\mu_\pm|\xi|^2)}\intd \xi_N
		&=\frac{1}{2\rho_\pm\la}\left(\frac{e^{-|a|A}}{A}-\frac{e^{-|a|B_\pm}}{B_\pm}\right), \\
	\frac{1}{2\pi}\int_{-\infty}^\infty\frac{\xi_N e^{ia \xi_N}}{|\xi|^2(\rho_\pm\la+\mu_\pm|\xi|^2)}\intd \xi_N
		&={\rm sign}(a)\frac{i}{2\rho_\pm\la}\left(e^{-|a|A}-e^{-|a|B_\pm}\right)
\end{align*}
for $a\in\BR$, where ${\rm sign}(a)=\pm 1$ when $\pm a>0$ and ${\rm sign}(a)=0$ when $a=0$.
Inserting these formulas into the above identity of $\wh{\psi}_{\pm N}(\xi',x_N)$ with $x_N=0$ yields that
\begin{align*}
	\wh{\psi}_{\pm N}(\xi',0) =&
		\int_{-\infty}^\infty\frac{A^2}{2\rho_\pm \la}\left(\frac{e^{-A|y_N|}}{A}-\frac{e^{-B_\pm|y_N|}}{B_\pm}\right)
		\wh{\wt{f}}_N(\xi',y_N)\intd y_N \\
		&-\sum_{j=1}^{N-1}\int_{-\infty}^\infty\frac{i\xi_j\,{\rm sign}(y_N)}{2\rho_\pm \la}
		\left(e^{-A|y_N|}-e^{-B_\pm |y_N|}\right)\wh{\wt{f}_j}(\xi',y_N)\intd y_N.
\end{align*}
By $\rho_\pm\la = \mu_\pm(B_\pm^2-A^2)$, we have
\begin{align*}
	\wh{\psi}_{\pm N}(\xi',0)
		=& \frac{1}{2\mu_\pm}\int_{-\infty}^\infty\frac{A}{B_\pm(B_\pm+A)}
		e^{-B_\pm |y_N|}\wh{\wt{f}}_N(\xi',y_N)\intd y_N \\
		&-\frac{1}{2\mu_\pm}\int_{-\infty}^\infty\frac{A}{B_\pm+A}
		\CM_\pm(|y_N|)\wh{\wt{f}}_N(\xi',y_N)\intd y_N \\
		&-\frac{1}{2\mu_\pm}\sum_{j=1}^{N-1}\int_{-\infty}^\infty\frac{i\xi_j}{B_\pm+A}{\rm sign}(y_N)
		\CM_\pm(|y_N|)\wh{\wt{f}_j}(\xi',y_N)\intd y_N.
\end{align*}
Thus, by Lemma \ref{lemm:1} and Lemma \ref{lemm:2},
there exist $\Bm_{\pm}\in(\BBM_{-2,2,\ep,0})^N$ and $\Bn_{\pm}\in(\BBM_{-1,2,\ep,0})^N$ such that
\begin{align*}
	&\jump{\wh k_N}(\xi',0) \\
		&=\sum_{\mathfrak{s}\in\{+,-\}}\left\{\int_{-\infty}^\infty Ae^{-B_\mathfrak{s} |y_N|}\Bm_\mathfrak{s}(\xi',\la)\cdot\wh{\wt \Bf}(\xi',y_N)\intd y_N
		+\int_{-\infty}^\infty A\CM_\mathfrak{s}(|y_N|)\Bn_\mathfrak{s}(\xi',\la)\cdot\wh{\wt \Bf}(\xi',y_N)\intd y_N\right\},
\end{align*}
which, combined with the formula of $\CK_{\pm 4}(\la)$, furnishes that
\begin{align*}
	&[\CK_{\pm 4}(\la)k_N](x)
		=\sum_{\mathfrak{s}\in\{+,-\}}\int_{-\infty}^\infty\CF_{\xi'}^{-1}
		\left[Ae^{\mp B_\pm x_N}e^{-B_\mathfrak{s}|y_N|}\frac{l(\xi',\la)m_0(\xi',\la)\Bm_\mathfrak{s}(\xi',\la)}{\det\BL}\cdot\wh{\wt\Bf}(\xi',y_N)\right](x')\displaybreak[0]\\
		&+\sum_{\mathfrak{s}\in\{+,-\}}\int_{-\infty}^\infty\CF_{\xi'}^{-1}
		\left[Ae^{\mp B_\pm x_N}\CM_\mathfrak{s}(|y_N|)\frac{l(\xi',\la)m_0(\xi',\la)\Bn_\mathfrak{s}(\xi',\la)}{\det\BL}\cdot\wh{\wt\Bf}(\xi',y_N)\right](x') 
		=:[\wt{\CK}_{\pm 4}(\la)\wt\Bf](x)
\end{align*}
with $l(\xi',\la) =\mu_-\{ L_{21}A(B_--A)+L_{22}B_-(B_-+A)\}$.
By Lemma \ref{lemm:1} and Lemma \ref{lemm:2},
\begin{equation*}
	\frac{l(\xi',\la)m_0(\xi',\la)\Bm_\pm(\xi',\la)}{\det\BL}\in(\BBM_{-2,2,\ep,0})^N, \quad
	\frac{l(\xi',\la)m_0(\xi',\la)\Bn_\pm(\xi',\la)}{\det\BL}\in(\BBM_{-1,2,\ep,0})^N,
\end{equation*}
which, combined with \cite[Lemma 5.6]{SS12} (cf. also \cite[Lemma B.2]{Saito15}), 
shows that $\wt K_{\pm 4}(\la)$ is the required operator in Lemma \ref{lemm:Rbound3}.
This completes the proof of the lemma.
%
%
\end{proof}

We apply 
Lemma \ref{lemm:Rbound1}, Lemma \ref{lemm:Rbound2}, and Lemma \ref{lemm:Rbound3} to \eqref{151012_1}
together with Proposition \ref{prop:R}, 
Lemma \ref{lemm:1}, and Lemma \ref{lemm:2}  
to see that
there exist operator families $\wt S_{\pm J}(\la)$, $\wt T_{\pm J}(\la)$ with
\begin{equation*}
	\wt{S}_{\pm J}(\la) \in\Hol(\Si_\ep,\CL(L_q(\dws)^{\CN},W_q^2(\BR_\pm^N))), \quad
	\wt T_{\pm J}(\la) \in \Hol(\Si_\ep,\CL(L_q(\dws)^N,W_q^2(\BR_\pm^N)))
\end{equation*}
such that
	$S_{\pm J}(\la)(\Bh,\Bk') =\wt{S}_{\pm J}(\la)(\nabla\Bh,\la^{1/2}\Bh,\nabla^2\Bk',\nabla\Bk',\la\Bk')$,
	$T_{\pm J}(\la)k_N = \wt{T}_{\pm J}(\la)\wt\Bf$, and
\begin{align*}
	\CR_{\CL(L_q(\dws)^{\CN},L_q(\BR_\pm)^{N^2+N+1})}
	\Big(\Big\{\Big(\la\frac{d}{d\la}\Big)^l\Big(R_\la \wt{S}_{\pm J}(\la)\Big)\mid \la\in\Si_\ep\Big\}\Big) &\leq \ga_1, \\
	\CR_{\CL(L_q(\dws)^N,L_q(\BR_\pm)^{N^2+N+1})}
	\Big(\Big\{\Big(\la\frac{d}{d\la}\Big)^l\Big(R_\la \wt{T}_{\pm J}(\la)\Big)\mid \la\in\Si_\ep\Big\}\Big) &\leq \ga_1 \quad (l=0,1),
\end{align*}
where $\CN$ is the same number as in Lemma \ref{lemm:Rbound2}.
Thanks to these properties and Proposition \ref{prop:R}, 
setting
\begin{align*}
	&S_\pm(\la)(\Bh,\Bk')=(S_{\pm 1}(\la)(\Bf,\Bk'),\dots,S_{\pm N}(\la)(\Bh,\Bk'))^T, \quad
		T_\pm(\la)k_N=(T_{\pm 1}(\la)k_N,\dots,T_{\pm N}(\la)k_N)^T, \\
	&S(\la)(\Bh,\Bk') = \left(S_+(\la)(\Bh,\Bk')\right)\chi_{\BR_+^N} + \left(S_-(\la)(\Bh,\Bk')\right)\chi_{\BR_-^N}, \quad
	T(\la)k_N = \left(T_+(\la)k_N\right)\chi_{\uhs} + \left(T_-(\la)k_N\right)\chi_{\lhs},
\end{align*}
we can construct an operator family 
	$\CB(\la)\in \Hol(\Si_\ep,\CL(L_q(\dws)^{\CN+N}),W_q^2(\dws)^N)$
such that
\begin{equation*}
	\CB(\la)(\nabla\Bh,\la^{1/2}\Bh,\nabla^2\Bk',\nabla\Bk',\la\Bk',\wt\Bf) = S(\la)(\Bh,\Bk') + T(\la)k_N,
\end{equation*}
which solves the problem \eqref{150608_5}, and 
\begin{equation*}
	\CR_{\CL(L_q(\dws)^{\CN+N},L_q(\dws)^{\wt N})}
	\Big(\Big\{\Big(\la\frac{d}{d\la}\Big)^l\Big(R_\la\CB(\la)\Big)\mid\la\in\Si_\ep\Big\}\Big) \leq \ga_1
	\quad (l=0,1).
\end{equation*}
Thus, we define an operator family $\CS_I(\la)$ as
\begin{equation*}
	\CS_I(\la)G_{\CR,\la}(\wt\Bf,\wt\Bh) = \CA(\la)\wt\Bf + \CB(\la)(\nabla\Bh,\la^{1/2}\Bh,\nabla^2\Bk',\nabla\Bk',\la\Bk',\wt\Bf)
\end{equation*}
with $\Bh = \wt\Bh-\mu\BD(\CA(\la)\wt\Bf)\Bn_0$ and $\Bk = -\CA(\la)\wt\Bf$,
which, combined with \eqref{151010_3} and Proposition \ref{prop:R},
shows that $\CS_I(\la)$ is the required operator in Theorem \ref{theo:ws2}.
This completes the proof of Theorem \ref{theo:ws2}.
\end{proof}

Since $R_\la V(g)=(V^1(\nabla g),V^2(\la^{1/2}g),V^3(\la g))$ as follows from Lemma \ref{lemm:div},
we have the following theorem by combining Theorem \ref{theo:ws2} with Lemma \ref{lemm:div} and by setting
\begin{align*}
	&Y_q = \{(\Bf,g,\Bh) \mid \Bf\in L_q(\dws)^N,g\in W_q^1(\dws)\cap \BW_q^{-1}(\ws),\Bh\in W_q^1(\dws)^N\}, \displaybreak[0] \\ 
	&\CY_q = \{(F_1,\dots,F_6) \mid F_1,F_4,F_6\in L_q(\dws)^N,F_2\in L_q(\dws)^{N^3},F_3,F_5\in L_q(\dws)^{N^2}\}.  \displaybreak[0] \\
	&G_\la(\Bf,g,\Bh) = (\Bf,V^1(\nabla g),V^2(\la^{1/2}g),V^3(\la g),\nabla\Bh,\la^{1/2}\Bh) 
		=(\Bf,R_\la V(g),\nabla\Bh,\la^{1/2}\Bh).
\end{align*}

\begin{theo}\label{theo:ws3}
	Let $1<q<\infty$, $0<\ep<\pi/2$, and $\rho_\pm$ be positive constants, and let $V$ be the same operator as in Lemma $\ref{lemm:div}$.
	Suppose that the condition $(d)$ holds.
	Then there exists an operator family 
		$\CT_I(\la) \in \Hol(\Si_\ep,\CL(\CY_q,W_q^2(\dws)^N))$
	such that $\Bu=V(g)+\CT_I(\la)G_\la(\Bf,g,\Bh)$ is a unique solution to the problem \eqref{150608_1} 
	with some pressure $\te$ for $\la\in\Si_\ep$ and $(\Bf,g,\Bh)\in Y_q$. 
	In addition, 
	\begin{equation*}
		\CR_{\CL(\CY_q,L_q(\dws)^{\wt N})}
			\Big(\Big\{\Big(\la\frac{d}{d\la}\Big)^l \Big(R_\la\CT_I(\la)\Big)\mid \la\in\Si_\ep\Big\}\Big) \leq \ga_1\quad
			(l=0,1).
	\end{equation*}
\end{theo}

\begin{proof}[Proof of Theorem \ref{theo:ws}]
Let $1<q<\infty$ and $q'=q/(q-1)$.
According to what was pointed out in Subsection \ref{sub:reduced},
we consider, as an auxiliary problem, the following weak problem:
\begin{equation}\label{weak:ws1}
	\la(g,\ph)_{\dws}+(\nabla g,\nabla\ph)_{\dws} =-(\Bf,\nabla\ph)_{\dws} \quad\text{for all $\ph\in W_{q'}^1(\ws)$,}  \quad  
		\jump{g} = <\jump{\Bh},\Bn_0> \quad \text{on $\BR_0^N$.} 
\end{equation}
Concerning this weak problem, we show the following proposition.

\begin{prop}\label{prop:w-w}
Let $0<\ep<\pi/2$ and $1<q<\infty$.
Suppose that $V$ is the same operator as in Lemma $\ref{lemm:div}$.
Then, for any $\la\in\Si_\ep$ and $(\Bf,\Bh)\in Y_{\CR,q}(\dws)$, 
the problem \eqref{weak:ws1} admits a unique solution $g\in W_q^1(\dws)\cap\BW_q^{-1}(\ws)$.
In addition, there exists an operator family 
	$\BV(\la) \in \Hol(\Si_\ep,\CL(\CY_{\CR,q}(\dws),W_q^2(\dws)^N))$
such that
\begin{equation}\label{160118_1}
	\CR_{\CL(\CY_{\CR,q}(\dws),L_q(\dws)^{\wt N})}
	\Big(\Big\{\Big(\la\frac{d}{d\la}\Big)^l\Big(R_\la\BV(\la)\Big)\mid\la\in\Si_{\ep}\Big\}\Big) \leq \ga_1
	\quad (l=0,1)
\end{equation}
and $V(g) = \BV(\la)(\Bf,\nabla\Bh,\la^{1/2}\Bh)$ for any $(\Bf,\Bh)\in Y_{\CR,q}(\dws)$,
where $g$ is the solution to \eqref{weak:ws1}. 
\end{prop}


\begin{proof}
We only show the existence of the $\CR$-bounded solution operator family $\BV(\la)$,
since the unique solvability of the weak problem \eqref{weak:ws1} was already mentioned in Proposition \ref{prop:weak1}.

It suffices to consider the case $\Bf\in C_0^\infty(\dws)^N$ in what follows, since $C_0^\infty(\dws)$ is dense in $L_q(\dws)$.
Then the $g$ satisfying \eqref{weak:ws1} is given by $g=\ph + \psi$ with
\begin{equation*}
	(\la-\De)\ph=\di\Bf \quad \text{in $\ws$,} \quad 
	\left\{\begin{aligned}
		&(\la-\De)\psi = 0 && \text{in $\dws$,} \\
		&\jump{\psi} = \jump{h}, \quad   
		\jump{\frac{\pa \psi}{\pa\Bn_0}} = 0 && \text{on $\BR_0^N$,}
	\end{aligned}\right.
\end{equation*}
where $h=<\Bh,\Bn_0>$ and $\pa \psi/\pa\Bn_0 = \Bn_0\cdot\nabla \psi= -\pa_N \psi$.


\noindent{\bf Step 1: Solution formulas.}
We give the exact solution formulas of $\ph$, $\psi$.
The $\ph$ is given by
\begin{equation}\label{sol:ww1}
	\ph= \CF^{-1}\left[\frac{\CF[\di\Bf](\xi)}{\la+|\xi|^2}\right](x) = \CF^{-1}\left[\frac{<i\xi,\CF[\Bf](\xi)>}{\la+|\xi|^2}\right](x).
\end{equation}

On the other hand, we rewrite the system for $\psi$ as follows:
\begin{equation}\label{160114_1}
	\left\{\begin{aligned}
		(\la-\De)\psi_\pm &= 0 && \text{in $\BR_\pm^N$,} \\
		\psi_+ - \psi_- &= \jump{h} && \text{on $\BR_0^N$,} \\
		\pa_N\psi_+-\pa_N\psi_- &= 0 && \text{on $\BR_0^N$,}
	\end{aligned}\right.
\end{equation}
where we have set $\psi_\pm = \psi\chi_{\BR_\pm^N}$.
Applying the partial Fourier transform with respect to $x'\in\tws$ to \eqref{160114_1} furnishes that  
\begin{equation*}
	\left\{\begin{aligned}
		\left\{\pa_N^2-(\la+|\xi'|^2)\right\}\wh{\psi}_\pm(\xi',x_N) &= 0, \quad \pm x_N>0, \\
		\wh\psi_+(\xi',0)-\wh\psi_-(\xi',0) &= \jump{\wh{h}}(\xi',0), \\
		(\pa_N\wh\psi_+)(\xi',0) - (\pa_N\wh\psi_-)(\xi',0) &= 0.
	\end{aligned}\right.
\end{equation*}
Solving this system as ordinary differential equations with respect to $x_N$ and setting $B=\sqrt{\la+A^2}$ with $A=|\xi'|$,
we obtain
	$\wh{\psi}_\pm (\xi',x_N) = \pm (1/2)\jump{\wh{h}}(\xi',0)e^{\mp B x_N}$ $(\pm x_N>0)$,
which implies that 
\begin{equation}\label{sol:ww2}
	\psi_\pm = \psi_\pm(x',x_N) =\pm\frac{1}{2}\CF_{\xi'}^{-1}\left[\jump{\wh{h}}(\xi',0)e^{\mp B x_N}\right](x') \quad (\pm x_N>0)
\end{equation}
solves the problem \eqref{160114_1}.
Hence, $\psi = \psi_+\chi_{\BR_+^N}+\psi_-\chi_{\BR_-^N}$. 

\noindent{\bf Step 2: Construction of $\CR$-bounded solution operator families.}
Since $V(\ph+\psi)=V(\ph)+V(\psi)$, we consider $V(\ph)$, $V(\psi)$ one by one.
First we construct a $\CR$-bounded solution operator family for $V(\ph)$.
By \eqref{sol:div} and \eqref{sol:ww1},
\begin{equation*}
	V(\ph) = \CF^{-1}\left[\frac{\xi<\xi,\CF[\Bf](\xi)>}{|\xi|^2(\la+|\xi|^2)}\right](x)=:\BV^1(\la)\Bf.
\end{equation*}
As was discussed in \eqref{151010_1}, we already know that 
\begin{align*}
	&	\BV^1(\la) \in \Hol(\Si_\ep,\CL(L_q(\ws)^N,W_q^2(\ws)^{N})), \\
	&\CR_{\CL(L_q(\ws)^N,L_q(\ws)^{\wt N})}\Big(\Big\{\Big(\la\frac{d}{d\la}\Big)^l \big(R_\la\BV^1(\la)\big)\mid \la\in\Si_\ep\Big\}\Big)\leq \ga_1
	\quad (l=0,1).
\end{align*}


Next, we consider the term $V(\psi)$.
By \eqref{sol:div}, we have, for $j=1,\dots,N-1$,  
\begin{align}\label{160117_1}
	&\wh{V_j(\psi)}(\xi',x_N) = -\int_{-\infty}^\infty i\xi_j \wh{\psi}(\xi',y_N)\left(\frac{1}{2\pi}\int_{-\infty}^\infty\frac{e^{i(x_N-y_N)\xi_N}}{|\xi|^2}\intd\xi_N\right)\intd y_N, \\
	&\wh{V_N(\psi)}(\xi',x_N) = -\int_{-\infty}^\infty  \wh{\psi}(\xi',y_N)\left(\frac{1}{2\pi}\int_{-\infty}^\infty\frac{i\xi_N e^{i(x_N-y_N)\xi_N}}{|\xi|^2}\intd\xi_N\right)\intd y_N. \nonumber
\end{align}
Since it holds, by the residue theorem, that
\begin{equation*}
	\frac{1}{2\pi}\int_{-\infty}^\infty\frac{e^{i a\xi_N}}{|\xi|^2}\intd\xi_N
		= \frac{e^{-|a|A}}{2A}, \quad
	\frac{1}{2\pi}\int_{-\infty}^\infty\frac{i\xi_N e^{ia\xi_N}}{|\xi|^2}\intd\xi_N
		=-{\rm sign}(a)\frac{e^{-|a|A}}{2} \quad (a\in\BR\setminus\{0\}),
\end{equation*}
we insert these formulas into \eqref{160117_1} in order to obtain
\begin{align*}
	&\wh{V_j(\psi)}(\xi',x_N) = -\frac{i\xi_j}{2A}\int_{-\infty}^\infty e^{-|x_N-y_N|A}\wh{\psi}(\xi',y_N)\intd y_N, \\
	&\wh{V_N(\psi)}(\xi',x_N) = \frac{1}{2}\int_{-\infty}^\infty{\rm sign}(x_N-y_N)e^{-|x_N-y_N|A}  \wh{\psi}(\xi',y_N)\intd y_N.
\end{align*} 
This combined with \eqref{sol:ww2} furnishes that
\begin{align*}
	&V_j(\psi)
		= -\CF_{\xi'}^{-1}\left[\frac{i\xi_j}{4A}\jump{\wh{h}}(\xi',0)\int_0^\infty \left(e^{-|x_N-y_N|A}-e^{-|x_N+y_N|A}\right)  e^{-By_N}\intd y_N\right](x'), \\
	&V_N(\psi)
		= \CF_{\xi'}^{-1}\left[\frac{\jump{\wh{h}}(\xi',0)}{4}\int_0^\infty \left({\rm sign}(x_N-y_N)e^{-|x_N-y_N|A} -{\rm sign}(x_N+y_N)e^{-|x_N+y_N|A}\right)
		e^{-By_N}\intd y_N\right](x').
\end{align*}
By direct calculations, we have the following lemma.

\begin{lemm}
Let $0<\ep<\pi/2$ and $\xi'\in\BR^{N-1}\setminus\{0\}$. We set
\begin{equation*}
	A=|\xi'|, \quad B=\sqrt{\la+|\xi'|^2}, \quad \CM(a) = \frac{e^{-Ba}-e^{-A a}}{B-A} \quad (\la\in\Si_\ep,\, a>0).
\end{equation*}
Then it holds that, for $\pm x_N>0$,
\begin{align*}
	&\int_0^\infty \left(e^{-|x_N-y_N|A}-e^{-|x_N+y_N|A}\right)  e^{-By_N}\intd y_N = \mp\frac{2A}{B+A}\CM(\pm x_N), \displaybreak[0] \\
	&\int_0^\infty \left({\rm sign}(x_N-y_N)e^{-|x_N-y_N|A} -{\rm sign}(x_N+y_N)e^{-|x_N+y_N|A}\right)e^{-By_N}\intd y_N \displaybreak[0] \\
	&= -\frac{2A}{B+A}\CM(\pm x_N) - \frac{2}{B+A}e^{\mp Bx_N}.
\end{align*}
\end{lemm}

This lemma yields that, for $\pm x_N>0$ and $j=1,\dots,N-1$,
\begin{align*}
	&[V_j(\psi)](x',x_N) = \pm \CF_{\xi'}^{-1}\left[\left(\frac{i\xi_j}{2A(B+A)}\right)A\CM(\pm x_N)\jump{\wh{h}}(\xi',0)\right](x'), \\
	&[V_N(\psi)](x',x_N) = -\frac{1}{2}\CF_{\xi'}^{-1}\left[\frac{A}{B+A}\CM(\pm x_N)\jump{\wh{h}}(\xi',0)\right](x')
		-\frac{1}{2}\CF_{\xi'}^{-1}\left[\frac{e^{\mp Bx_N}}{B+A}\jump{\wh{h}}(\xi',0)\right](x') \\
	&=:[V_N^1(\psi)](x',x_N)+[V_N^2(\psi)](x',x_N).
\end{align*}
By Lemma \ref{lemm:Rbound1} and $h=<\Bh,\Bn_0>$, there exist 
	$\BV_j(\la), \BV_N^1(\la)\in\Hol(\Si_\ep,\CL(L_q(\dws)^{N^2+N},W_q^2(\dws)))$
such that $V_j(\psi) = \BV_j(\la)(\nabla\Bh,\la^{1/2}\Bh)$,
$V_N^1(\psi) = \BV_N^1(\la)(\nabla\Bh,\la^{1/2}\Bh)$, and 
\begin{equation*}
	\CR_{\CL(L_q(\dws)^{N^2+N},L_q(\dws)^{2(N^2+N+1)})}
	\Big(\Big\{\Big(\la\frac{d}{d\la}\Big)^l\Big(R_\la\BV_j(\la),R_\la\BV_N^1(\la)\Big)\mid\la\in\Si_\ep\Big\}\Big)\leq \ga_1 \quad (l=0,1).
\end{equation*}

To treat $V_N^2(\psi)$, we show the following lemma.

\begin{lemm}\label{lemm:160118}
Let $0<\ep<\pi/2$ and $1<q<\infty$. We define a operator $K_\pm(\la)$ on $W_q^1(\dws)$ by the formulas:
\begin{equation*}
	 [K_\pm(\la)f](x) = \CF_{\xi'}^{-1}\left[\frac{e^{\mp Bx_N}}{B+A}\jump{\wh{f}}(\xi',0)\right](x') \quad (\pm x_N>0,\,\la\in\Si_\ep).
\end{equation*}
Then there exists an operator families 
	$\wt K_\pm(\la)\in \Hol(\Si_\ep,\CL(L_q(\dws)^{N+1},W_q^2(\BR_\pm^N)))$
such that $K_\pm(\la) f = \wt K_\pm(\la)(\nabla f,\la^{1/2}f) $ and 
	\begin{equation}\label{160206_1}
		\CR_{\CL(L_q(\dws)^{N+1},L_q(\BR_\pm^N)^{N^2+N+1})}
		\Big(\Big\{\Big(\la\frac{d}{d\la}\Big)^l\Big(R_\la\wt K_\pm(\la)\Big)\mid \la\in\Si_\ep \Big\}\Big)
		\leq \ga_1 \quad (l=0,1).
	\end{equation}
\end{lemm}

\begin{proof}
By using the relation:
$g(x_N)h(0) = -\int_0^\infty d/dy_N\left(g(x_N+y_N)h(y_N)\right)\intd y_N$  
for functions $g$, $h$ satisfying $g(x_N+y_N)h(y_N)\to 0$ as $y_N\to\infty$,
we rewrite $K_\pm(\la)f$ as
\begin{align*}
&K_\pm(\la)f = \pm \int_0^\infty \CF_{\xi'}^{-1}
\left[\frac{\la^{1/2}}{B(B+A)}e^{\mp Bx_N-B y_N}\left(\wh{\la^{1/2}f}(\xi',y_N)-\wh{\la^{1/2}f}(\xi',-y_N)\right)\right](x')\intd y_N \displaybreak[0] \\
&\quad\mp \sum_{j=1}^{N-1}
\int_0^\infty \CF_{\xi'}^{-1}
\left[\frac{i\xi_j}{B(B+A)}e^{\mp Bx_N- By_N}\left(\wh{\pa_j f}(\xi',y_N)-\wh{\pa_j f}(\xi',-y_N)\right)\right](x')\intd y_N \displaybreak[0] \\
&\quad\pm \int_0^\infty\CF_{\xi'}^{-1}\left[\frac{1}{B+A}e^{\mp Bx_N - By_N}\left(\wh{\pa_N f}(\xi',y_N) + \wh{\pa_N f}(\xi',-y_N)\right)\right](x')\intd y_N  
=: \wt K_\pm(\la)(\nabla f,\la^{1/2}f) 
\end{align*}
for $\pm x_N>0$, respectively, where we have used $B^2=\la+A^2$.
From now on, we show the estimate \eqref{160206_1}.
Noting $\la=(B+A)(B-A)$ and $B/(B+A) = 1 - A/(B+A)$, we have, for $k,l=1,\dots,N-1$ and $\pm x_N>0$,
\begin{align*}
	&\left(\pa_k\pa_l,\pa_k\pa_N,\pa_N\pa_l,\la^{1/2}\pa_k\right) \wt K_\pm(\la)(\nabla f,\la^{1/2}f) \\
		&\quad=	\pm \int_0^\infty \CF_{\xi'}^{-1}
		\left[\frac{\la^{1/2}m(\xi',\la)}{B}A e^{\mp Bx_N-B y_N}\left(\wh{\la^{1/2}f}(\xi',y_N)-\wh{\la^{1/2}f}(\xi',-y_N)\right)\right](x')\intd y_N \displaybreak[0] \\
		&\quad\mp \sum_{j=1}^{N-1}
		\int_0^\infty \CF_{\xi'}^{-1}
		\left[\frac{i\xi_j m(\xi',\la)}{B}A e^{\mp Bx_N- By_N}\left(\wh{\pa_j f}(\xi',y_N)-\wh{\pa_j f}(\xi',-y_N)\right)\right](x')\intd y_N \displaybreak[0] \\
		&\quad\pm \int_0^\infty\CF_{\xi'}^{-1}\left[m(\xi',\la) Ae^{\mp Bx_N - By_N}\left(\wh{\pa_N f}(\xi',y_N) + \wh{\pa_N f}(\xi',-y_N)\right)\right](x')\intd y_N, \displaybreak[0] \\
	&\left(\pa_N^2,\la^{1/2}\pa_N\right)\wt K_\pm(\la)(\nabla f,\la^{1/2}f) \displaybreak[0] \\
		&\quad=	\pm \int_0^\infty \CF_{\xi'}^{-1}
		\left[\left(\frac{\la^{1/2}}{B}\right)\left(n(\xi',\la) -A l(\xi',\la)\right)e^{\mp Bx_N-B y_N}
		\left(\wh{\la^{1/2}f}(\xi',y_N)-\wh{\la^{1/2}f}(\xi',-y_N)\right)\right](x')\intd y_N \displaybreak[0] \\
		&\quad\mp \sum_{j=1}^{N-1}
		\int_0^\infty \CF_{\xi'}^{-1}
		\left[\left(\frac{i\xi_j}{B}\right)\left(n(\xi',\la) -A l(\xi',\la)\right)e^{\mp Bx_N- By_N}
		\left(\wh{\pa_j f}(\xi',y_N)-\wh{\pa_j f}(\xi',-y_N)\right)\right](x')\intd y_N \displaybreak[0] \\
		&\quad\pm \int_0^\infty\CF_{\xi'}^{-1}\left[\left(n(\xi',\la) -A l(\xi',\la)\right)e^{\mp Bx_N - By_N}
		\left(\wh{\pa_N f}(\xi',y_N) + \wh{\pa_N f}(\xi',-y_N)\right)\right](x')\intd y_N, \displaybreak[0] \\
	&\la \wt K_\pm(\la)(\nabla f,\la^{1/2}f) \displaybreak[0] \\
		&\quad=	\pm \int_0^\infty \CF_{\xi'}^{-1}
		\left[\left(\frac{\la^{1/2}}{B}\right)(B-A)e^{\mp Bx_N-B y_N}\left(\wh{\la^{1/2}f}(\xi',y_N)-\wh{\la^{1/2}f}(\xi',-y_N)\right)\right](x')\intd y_N \displaybreak[0] \\
		&\quad\mp \sum_{j=1}^{N-1}
		\int_0^\infty \CF_{\xi'}^{-1}
		\left[\left(\frac{i\xi_j}{B}\right)(B-A)e^{\mp Bx_N- By_N}\left(\wh{\pa_j f}(\xi',y_N)-\wh{\pa_j f}(\xi',-y_N)\right)\right](x')\intd y_N \displaybreak[0] \\
		&\quad\pm \int_0^\infty\CF_{\xi'}^{-1}\left[(B-A)e^{\mp Bx_N - By_N}\left(\wh{\pa_N f}(\xi',y_N) + \wh{\pa_N f}(\xi',-y_N)\right)\right](x')\intd y_N,  
\end{align*}
where we have set
\begin{align*}
	m(\xi',\la) &= \left(-\frac{\xi_k\xi_l}{A(B+A)},\frac{\mp i\xi_kB}{A(B+A)},\frac{\mp i\xi_l B}{A(B+A)},\frac{i\xi_k\la^{1/2}}{A(B+A)}\right), \\
	n(\xi',\la) &=	\left(B,\mp \la^{1/2}\right), \quad l(\xi',\la) = \left(\frac{B}{B+A},\frac{\mp\la^{1/2}}{B+A}\right).
\end{align*}
Since $m(\xi',\la),\,l(\xi',\la) \in \BBM_{0,2,\ep,0}$ and $n(\xi',\la)\in \BBM_{1,1,\ep,0}$, 
applying \cite[Lemma 5.4] {SS12} with Lemma \ref{lemm:1}, Lemma \ref{lemm:2}
to the above formulas of $\wt K_\pm(\la)(\nabla f,\la^{1/2}f)$ furnishes \eqref{160206_1}.
This completes the proof of the lemma.
%
%
%
%
\end{proof}

By Lemma \ref{lemm:160118} and $h=<\Bh,\Bn_0>$, there exists 
$\BV_N^2(\la)\in\Hol(\Si_\ep,\CL(L_q(\dws)^{N^2+N},W_q^2(\dws)))$
such that $V_N^2(\psi) = \BV_N^2(\la)(\nabla\Bh,\la^{1/2}\Bh)$ and 
\begin{equation*}
\CR_{\CL(L_q(\dws)^{N^2+N},L_q(\dws)^{N^2+N+1})}
\Big(\Big\{\Big(\la\frac{d}{d\la}\Big)^l\big(R_\la\BV_N^2(\la)\big)\mid\la\in\Si_\ep\Big\}\Big)
\leq \ga_1 \quad(l=0,1).
\end{equation*}
Recalling Remark \ref{rema:div} \eqref{rema:div3}, we set, for $(H_2,H_3)\in L_q(\dws)^{N^2}\times L_q(\dws)^N$,
\begin{equation*}
	\BV^2(\la)(H_2,H_3) =
		\left(\BV_1(\la)(H_2,H_3),\dots,\BV_{N-1}(\la)(H_2,H_3),\BV_N^1(\la)(H_2,H_3)+\BV_N^2(\la)(H_2,H_3)\right)^T.
\end{equation*}
Then $\BV(\la)\BH=\BV^1(\la)H_1+\BV^2(\la)(H_2,H_3)$ with $\BH=(H_1,H_2,H_3)\in \CY_{\CR,q}(\dws)$ satisfies \eqref{160118_1}.
Moreover, for $(\Bf,\Bh)\in Y_{\CR,q}(\dws)$, $V(g) = \BV(\la)(\Bf,\nabla\Bh,\la^{1/2}\Bh)$ 
with the solution $g$ of \eqref{weak:ws1}. 
\end{proof}

We set $\BS_I(\la)\BH = \BV(\la)\BH + \CT_I(\la)\left(H_1,R_\la\BV(\la)\BH,H_2,H_3\right)$
for $\BH=(H_1,H_2,H_3)\in \CY_{\CR,q}(\dws)$.
Then,  Theorem \ref{theo:ws3} and Proposition \ref{prop:w-w}, together with Proposition \ref{prop:R},
shows that $\BS_I(\la)$ is the required operator in Theorem \ref{theo:ws}.
This completes the proof of Theorem \ref{theo:ws}.
\end{proof}


\section{Reduced Stokes resolvent equations on a bent space}\label{sec:bent}
Let $\Phi:\BR_x^N\to\BR_y^N$ be a bijection of $C^1$ class and let $\Phi^{-1}$ be its inverse map,
where subscripts $x$, $y$ denote their variables, here and subsequently.
Writing $(\nabla_x\Phi)(x) = \BA +\BB(x)$ and $(\nabla_y\Phi^{-1})(\Phi(x))=\BA_{-1}+\BB_{-1}(x)$,
we assume that $\BA$ and $\BA_{-1}$ are orthonormal matrices with constant coefficients and $\det\BA=\det\BA_{-1}=1$,
and also assume that
$\BB(x)$ and $\BB_{-1}(x)$ are matrices of functions in $W_r^1(\ws)$ with $N<r<\infty$ such that
\begin{equation}\label{M1M2}
	\|(\BB,\BB_{-1})\|_{L_\infty(\ws)} \leq M_1, \quad
	\|\nabla_x(\BB,\BB_{-1})\|_{L_r(\ws)}\leq M_2.
\end{equation}
We will choose $M_1$ small enough eventually, so that we may assume that $0<M_1\leq 1 \leq M_2$ in the following.

\begin{rema}\label{rema:AB}
Since $x = \Phi^{-1}(\Phi(x))$, we have $\BI = (\nabla_y\Phi^{-1})(\nabla_x\Phi)$.
This implies that $(\nabla_y\Phi^{-1})^{-1}=(\nabla_x\Phi)$, which is equivalent to $(\BA_{-1}+\BB_{-1}(x))^{-1}=\BA+\BB(x)$.
\end{rema}

Set $\Om_\pm = \Phi(\BR_\pm^N)$ and $\Ga = \Phi(\BR_0^N)$, and let $\wt\Bn=\wt\Bn(y)$ be the unit normal vector on $\Ga$,
which points from $\Om_+$ to $\Om_-$.
In addition, setting $\Phi^{-1}=(\Phi_{-1,1},\dots,\Phi_{-1,N})^T$, we see that $\Ga$ is represented by $\Phi_{-1,N}(y)=0$,
since $\Ga = \Phi(\{x_N=0\})= \Phi\circ\Phi^{-1}(\{y\in\ws\mid\Phi_{-1,N}(y)=0\})$ by $x_N=\Phi_{-1,N}(y)$.
This representation implies that
\begin{equation}\label{normal}
	\wt\Bn(\Phi(x)) = - \frac{\nabla_y\Phi_{-1,N}}{|\nabla_y\Phi_{-1,N}|}
		= - \frac{(A_{N1}+B_{N1}(x),\dots,A_{NN}+B_{NN}(x))^T}{(\sum_{i=1}^N(A_{Ni}+B_{Ni}(x))^2)^{1/2}}
		= \frac{(\BA_{-1}+\BB_{-1}(x))^T\Bn_0}{|(\BA_{-1}+\BB_{-1}(x))^T\Bn_0|}
\end{equation}
with $\Bn_0=(0,\dots,0,-1)^T$,
where we have set $\BA_{-1}=(A_{ij})$ and $\BB_{-1}(x)=(B_{ij}(x))$.
In particular, $\wt\Bn$ is defined on $\ws$ by \eqref{normal}.
Since $\sum_{i=1}^N(A_{Ni}+B_{Ni}(x))^2=1+\sum_{i=1}^N(2 A_{Ni}B_{Ni}(x)+B_{Ni}(x)^2)$
by the fact that $\BA_{-1}$ is a orthonormal matrix,
we see by \eqref{M1M2} and \eqref{normal} that $\|\nabla_x\wt\Bn\|_{L_r(\BR^N)}\leq C_N M_2$.
Let $\wt\mu_\pm=\wt\mu_\pm(y)$ be a viscosity coefficient that is defined on $\ws$ and satisfies conditions:
\begin{equation}\label{visco:1}
	\frac{1}{2}\mu_{\pm 1}\leq \wt\mu_\pm(y) \leq \frac{3}{2}\mu_{\pm 2} \quad (y\in\ws), \quad 
	|\wt\mu_\pm(y) - \mu_{\pm0}|\leq M_1 \quad (y\in\ws), \quad
	\|\nabla_y\wt\mu_\pm\|_{L_r(\ws)}\leq C_{r}, 
\end{equation}
where $\mu_{\pm 0}$ are some constant with $\mu_{\pm 1} \leq \mu_{\pm 0} \leq \mu_{\pm2}$, respectively,
for the same constants $\mu_{\pm 1}$, $\mu_{\pm 2}$ as in Theorem \ref{theo:main}.
In addition, we set
\begin{align}\label{visco:2}
	&\wt\mu(y)=\wt\mu_+(y)\chi_{\Om_+}(y) + \wt\mu_-(y)\chi_{\Om_-}(y), \quad 
	\wt\rho(y) = \rho_+\chi_{\Om_+}(y)+\rho_-\chi_{\Om_-}(y) \quad 
		(\rho_\pm:\text{ positive constants}), 
\end{align}
and also set $\mu_\pm(x) =\wt\mu_\pm(\Phi(x))$, $\mu(x)=\wt\mu(\Phi(x))$, $\rho(x) = \wt\rho(\Phi(x))$, and
$\mu_{0} (x) = \wt\mu_{0}(\Phi(x))$ with $\wt\mu_0(y) = \mu_{+0}\chi_{\Om_+}(y)+ \mu_{-0}\chi_{\Om_-}(y)$.
It then holds that
\begin{align}\label{visco:3}
	&\rho = \rho(x) = \rho_+\chi_{\BR_+^N}(x) + \rho_-\chi_{\BR_-^N}(x), \quad \mu_0 = \mu_0(x) = \mu_{+0}\chi_{\uhs}(x) + \mu_{-0}\chi_{\lhs}(x), \\
	&\mu(x) =  \mu_+(x)\chi_{\uhs}(x)+\mu_-(x)\chi_{\lhs}(x), \quad 
	|\mu(x)-\mu_0|\leq M_1 \quad (x\in\dws), \quad \|\nabla_x\mu\|_{L_r(\dws)}\leq C_r.\nonumber
\end{align}

First we consider the two-phase reduced Stokes equation in $\dot\Om=\Om_+\cup\Om_-$ with an interface condition:
\begin{equation}\label{reduce:2}
	\left\{\begin{aligned}
		\la\wt\Bu-\wt\rho^{\,-1}\Di\wt\BT(\wt\Bu,\wt K_I(\wt\Bu)) &=\wt\Bf && \text{in $\dot\Om$,} \\
		\jump{\wt\BT(\wt\Bu,\wt K_I(\wt\Bu))\wt\Bn} &= \jump{\wt\Bh} && \text{on $\Ga$,} \\
		\jump{\wt\Bu} &= 0 && \text{on $\Ga$.}
	\end{aligned}\right.
\end{equation}
Here $\wt\BT(\wt\Bu,\wt K_I(\wt\Bu))=\wt\mu\BD(\wt\Bu)-\wt K_I(\wt\Bu)\BI$ and
$\wt K_I(\wt\Bu)$ is a unique solution to the following weak problem:
\begin{align}
	&(\wt\rho^{\,-1}\nabla \wt K_I(\wt\Bu),\nabla\wt\ph)_{\dot{\Om}}
	=(\wt\rho^{\,-1}\Di(\wt\mu\BD(\wt\Bu))-\nabla\di\wt\Bu,\nabla\wt\ph)_{\dot{\Om}} \quad \text{for all $\wt\ph\in \wh{W}_{q'}^1(\BR_y^N)$,} \label{150824_1} \\
	&\jump{\wt K_I(\wt\Bu)} = <\jump{\wt\mu\BD(\wt\Bu)\wt\Bn},\wt\Bn> - \jump{\di\wt\Bu} \quad \text{on $\Ga$.} \label{150824_1-2}
\end{align}
We then have the following theorem.
\begin{theo}\label{theo:bent}
	Let $0<\ep<\pi/2$, $1<q<\infty$, $N<r<\infty$, and $\max(q,q')\leq r$ with $q'=q/(q-1)$.
	Suppose that \eqref{M1M2}, \eqref{visco:1}, and \eqref{visco:2} hold. 
	Let $Z_{\CR,q}(G)$ and $\CZ_{\CR,q}(G)$, with an open set $G$ of $\ws$, be defined as 
	\begin{align*}
		Z_{\CR,q}(G) &= L_q(G)^N\times W_q^1(G)^N, \\
		\CZ_{\CR,q}(G) &= \{(H_1,\dots,H_4) \mid H_1,H_3\in L_q(G)^N,\, H_2\in L_q(G)^{N^2},\, H_4\in W_q^1(G)^N\}, 
	\end{align*}
	while $\mu^* := (1/2)\min(\mu_{+1}, \mu_{-1}, \mu_{+2},\mu_{-2})$.
	Then there exist $0<M_1<\min(1,\mu^*)$, $\la_0\geq1$, and
	$\wt\BS_I(\la) \in \Hol(\Si_{\ep,\la_0},\CL(\CZ_{\CR,q}(\dot\Om), W_q^2(\dot{\Om})^N))$
	such that, for any $\la\in\Si_{\ep,\la_0}$ and $(\wt\Bf,\wt\Bh)\in Z_{\CR,q}(\dot\Om)$,
	$\wt\Bu=\wt\BS_I(\la)H_{\CR,\la}(\wt\Bf,\wt\Bh)$ is a unique solution to the problem \eqref{reduce:2}, and furthermore,
	\begin{equation}\label{150825_12}
		\CR_{\CL(\CZ_{\CR,q}(\dot\Om),L_q(\dot{\Om})^{\wt N})}
		\Big(\Big\{\Big(\la\frac{d}{d\la}\Big)^l (R_\la\wt\BS_I(\la))\mid\la\in\Si_{\ep,\la_0}\Big\}\Big) \leq\ga_2
		\quad(l=0,1)
	\end{equation}
	with some positive constant $\ga_2$.
	Here and subsequently, $\wt N = N^3+N^2+N$, $R_\la\Bu = (\nabla^2\Bu,\la^{1/2}\nabla\Bu,\la\Bu)$, and
	$H_{\CR,\la}(\wt\Bf,\wt\Bh) = (\wt\Bf,\nabla\wt\Bh,\la^{1/2}\wt\Bh,\wt\Bh)$;
	 $M_1$ is a constant depending on $N$, $q$, $r$, $\ep$, $\rho_+$ $\rho_-$, $\mu_{+ 1}$,  $\mu_{- 1}$, $\mu_{+ 2}$, and $\mu_{- 2}$;
	$\la_0$ is a constant depending on $M_2$, $N$, $q$, $r$, $\ep$, $\rho_+$ $\rho_-$, $\mu_{+ 1}$,  $\mu_{- 1}$, $\mu_{+ 2}$, and $\mu_{- 2}$;
	$\ga_2$ denotes a generic constant depending on $M_2$, $\la_0$, $N$, $q$, $r$, $\ep$, $\rho_+$ $\rho_-$, $\mu_{+ 1}$,  $\mu_{- 1}$, $\mu_{+ 2}$, and $\mu_{- 2}$.
\end{theo}

The remaining part of this section is mainly devoted to the proof of Theorem \ref{theo:bent}.
We rewrite the problem \eqref{reduce:2} as follows:
\begin{equation}\label{reduce:3}
	\left\{\begin{aligned}
		\la\wt\Bu-\wt\rho^{\,-1}\wt\mu\Di\BD(\wt\Bu)+\wt\rho^{\,-1}\nabla\wt\te-\wt\rho^{\,-1}\BD(\wt\Bu)\nabla\wt\mu
			&=\wt\Bf && \text{in $\dot\Om$,} \\
		\jump{(\wt\mu\BD(\wt\Bu)-\wt\te\BI)\wt\Bn} &= \jump{\wt\Bh} && \text{on $\Ga$,} \\
		\jump{\wt\Bu} &=0 && \text{on $\Ga$}
	\end{aligned}\right.
\end{equation}
with $\wt\te=\wt K_I(\wt\Bu)$. 
By the change of variable: $y=\Phi(x)$,
we transform the problem \eqref{reduce:3} to some problem on $\dws$ with $\Bu(x)=\wt\Bu(y)$ and $\te(x)=\wt\te(y)$.
Here we note the following fundamental relations:
\begin{align*}
	&\frac{\pa}{\pa y_j} = \sum_{k=1}^N(A_{kj}+B_{kj}(x))\frac{\pa}{\pa x_k}, \quad \nabla_y = (\BA_{-1}+\BB_{-1}(x))^T\nabla_x, \displaybreak[0] \\
	&\frac{\pa^2}{\pa y_j\pa y_k} =\sum_{l,m=1}^N A_{lj}A_{mk}\frac{\pa^2}{\pa x_l\pa x_m} 
	+\sum_{l,m=1}^N\left(A_{lj}B_{mk}(x)+A_{mk}B_{lj}(x)+B_{lj}(x)B_{mk}(x)\right)\frac{\pa^2}{\pa x_l\pa x_m} \nonumber \displaybreak[0] \\
	&+\sum_{l,m=1}^N\left(A_{lj}+B_{lj}(x)\right)\left(\frac{\pa}{\pa x_l}B_{mk}(x)\right)\frac{\pa}{\pa x_m}, \nonumber
\end{align*}
and furthermore,
\begin{align*}
	&	\De_y = \De_x +\sum_{k,l,m=1}^N\left(A_{lk}+B_{lk}(x)\right)\left(\frac{\pa}{\pa x_l}B_{mk}(x)\right)\frac{\pa}{\pa x_m} \displaybreak[0] \\
	&\quad+\sum_{k,l,m=1}^N\left(A_{lk}B_{mk}(x)+A_{mk}B_{lk}(x)+B_{lk}(x)B_{mk}(x)\right)\frac{\pa^2}{\pa x_l \pa x_m}, \nonumber \displaybreak[0] \\
	&\nabla_y\di_y\wt\Bu = (\BA_{-1}+\BB_{-1}(x))^T\nabla_x\di_x(\BA_{-1}\Bu) 
	+(\BA_{-1}+\BB_{-1}(x))^T\sum_{j,k=1}^N\nabla_x\left(B_{kj}(x)\frac{\pa}{\pa x_k} u_j\right), \displaybreak[0] \nonumber \\
	&\BD(\wt\Bu) = \nabla_x\Bu(\BA_{-1}+\BB_{-1}(x))+(\BA_{-1}+\BB_{-1}(x))^T (\nabla_x\Bu)^T. \nonumber
\end{align*}
Thus the first equation of \eqref{reduce:3} is reduced to
\allowdisplaybreaks{\begin{align*}
	&\wt\Bf = \la\wt\Bu-\frac{\wt\mu}{\wt\rho}\left(\De_y\wt\Bu+\nabla_y\di_y\wt\Bu\right)
		+\frac{1}{\wt\rho}\nabla_y\wt\te-\frac{1}{\wt\rho}\BD(\wt\Bu)\nabla_y\wt\mu  \\
	&= \la\Bu-\frac{\mu}{\rho}\left(\De\Bu+\BA_{-1}^T\nabla\di(\BA_{-1}\Bu)\right)
		+\frac{1}{\rho}\BA_{-1}^T\nabla\te
		-\frac{\mu}{\rho}\left\{
 		\sum_{k,l,m=1}^N\left(A_{lk}+B_{lk}(x)\right)\left(\frac{\pa}{\pa x_l}B_{mk}(x)\right)\frac{\pa}{\pa x_m}\Bu
 		\right. \\
		&+\sum_{k,l,m=1}^N\left(A_{lk}B_{mk}(x)+A_{mk}B_{lk}(x)+B_{lk}(x)B_{mk}(x)\right)\frac{\pa^2}{\pa x_l\pa x_m}\Bu \\
		&\left.+\BB_{-1}(x)^T\nabla\di(\BA_{-1}\Bu)
		+(\BA_{-1}+\BB_{-1}(x))^T\sum_{j,k=1}^N\nabla\left(B_{kj}(x)\frac{\pa}{\pa x_k}u_j\right)\right\}
		+\frac{1}{\rho}\BB_{-1}(x)^T\nabla\te \\
		&-\frac{1}{\rho}\left\{(\nabla\Bu)(\BA_{-1}+\BB_{-1}(x))+(\BA_{-1}+\BB_{-1}(x))^T (\nabla\Bu)^T\right\}
		(\BA_{-1}+\BB_{-1}(x))^T\nabla\mu,
\end{align*}}
and we have, by setting $\Bv=\BA_{-1}\Bu$ and $\Bf = \BA_{-1}\wt\Bf\circ\Phi$,
\begin{equation}\label{150612_1}
	\la\Bv -\rho^{-1}\mu\Di\BD(\Bv)+\rho^{-1}\nabla\te+\rho^{-1}\CF^1(\Bv)+\rho^{-1}\CP^1\nabla\te
		 =\Bf \quad \text{in $\dws$.}
\end{equation}
Here we have the following information for $\CF^1(\Bv)$ and $\CP^1$:
\begin{align}\label{151016_1}
	&\CF^1(\Bv) =  
		\mu(\CR^1\nabla^2\Bv + \CS^1\nabla\Bv) + (\CT^1\nabla\Bv)\nabla\mu, \quad
	\|(\CP^1,\CR^1)\|_{L_\infty(\ws)}\leq C_N M_1, \\
	&\|(\nabla\CP^1,\nabla\CR^1,\CS^1)\|_{L_r(\ws)}\leq C_N M_2, \quad
	\|(\CT^1,\nabla\CT^1)\|_{L_\infty(\ws)\times L_r(\ws)}\leq C_N M_2. \nonumber
\end{align}

Next we consider the interface condition of \eqref{reduce:3}. 
By \eqref{normal}, 
\begin{equation*}
	|(\BA_{-1}+\BB_{-1}(x))^T\Bn_0|\jump{\wt\Bh} = \jump{(\wt\mu\BD(\wt\Bu)-\wt\te\BI)(\BA_{-1}+\BB_{-1}(x))^T\Bn_0},
\end{equation*}
which, multiplied by $(\BA_{-1}+\BB_{-1}(x))^{-T}=\{(\BA_{-1}+\BB_{-1}(x))^T\}^{-1}$, furnishes that
\begin{align*}
	&|(\BA_{-1}+\BB_{-1}(x))^T\Bn_0|(\BA_{-1}+\BB_{-1}(x))^{-T}\jump{\wt\Bh} \\
		&= \jump{\wt\mu(\BA_{-1}+\BB_{-1}(x))^{-T}\BD(\wt\Bu)(\BA_{-1}+\BB_{-1}(x))^T\Bn_0} -\jump{\wt\te\Bn_0} \\
		&=\jump{(\mu\BD(\Bv)-\te\BI)\Bn_0}
			+\jump{\mu\{((\BA_{-1}+\BB_{-1}(x))^{-T}\BA_{-1}^T-\BI\}(\nabla\Bv)\Bn_0} + \jump{\mu(\nabla\Bv)^T\BA_{-1}\BB_{-1}(x)^T\Bn_0} \\
			&+ \jump{\mu(\BA_{-1}+\BB_{-1}(x))^{-T}\BA_{-1}^T(\nabla\Bv)
			(\BA_{-1}\BB_{-1}(x)^T+\BB_{-1}(x)\BA_{-1}^{T}+\BB_{-1}(x)\BB_{-1}(x)^T)\Bn_0}.
\end{align*}
Since $(\BA_{-1}+\BB_{-1}(x))^{-T}\BA_{-1}^T = (\BI+\BB_{-1}(x)\BA_{-1}^T)^{-T}$
and $(\BA_{-1}+\BB_{-1}(x))^{-1}=\BA +\BB(x)$ by Remark \ref{rema:AB},
it holds that, by \eqref{M1M2},
\begin{align*}
	&\|(\BA_{-1}+\BB_{-1}(\cdot))^{-T}\BA_{-1}^T-\BI\|_{L_\infty(\ws)} \leq C_N M_1, \quad
	\|\nabla\{(\BA_{-1}+\BB_{-1}(\cdot))^{-T}\BA_{-1}^T-\BI\}\|_{L_r(\ws)} \leq C_N M_2, \\
	&\|(\BA_{-1}+\BB_{-1}(\cdot))^{-T}\BA_{-1}^T\|_{L_\infty(\ws)}+
	\|\nabla\{(\BA_{-1}+\BB_{-1}(\cdot))^{-T}\BA_{-1}^T\}\|_{L_r(\ws)} \leq C_N M_2.
\end{align*}
We thus see, by setting $\Bh = |(\BA_{-1}+\BB_{-1}(x))^T\Bn_0|(\BA_{-1}+\BB_{-1}(x))^{-T}\wt\Bh\circ\Phi$, that
\begin{equation}\label{151016_5}
	\jump{(\mu\BD(\Bv)-\te\BI)\Bn_0}+\jump{\CF^2(\Bv)\Bn_0} = \jump{\Bh} \quad \text{on $\BR_0^N$,}
\end{equation}
where $\CF^2(\Bv)$ satisfies the following properties:
\begin{equation}\label{151016_6}
	\CF^2(\Bv) = \mu\CR^2\nabla\Bv, \quad \|\CR^2\|_{L_\infty(\ws)} \leq C_NM_1, \quad \|\nabla\CR^2\|_{L_r(\ws)} \leq C_N M_2.
\end{equation}

Finally, we consider the weak problem \eqref{150824_1}-\eqref{150824_1-2}.
Let $({\rm LHS})$ and $({\rm RHS})$ stand for the left-hand side and the right-hand side of \eqref{150824_1}, respectively.
Then, for any $\wt\ph\in \wh{W}_{q'}^1(\BR_y^N)$ and for $\ph(x)=\wt\ph(\Phi(x))$, we have
\begin{align}\label{151016_8}
	&\text{(LHS)} = 
		(\rho^{-1}(\BA_{-1}+\BB_{-1})^T\nabla\te,|\det\nabla\Phi|(\BA_{-1}+\BB_{-1})^T\nabla\ph)_{\dws} \\
		&\quad=(\rho^{-1}\nabla\te,\nabla\ph)_{\dws} + (\rho^{-1}\CP^2\nabla\te,\nabla\ph)_{\dws}, \nonumber \\
	&\text{(RHS)} = 
		(\rho^{-1}\mu\Di\BD(\Bv)-\rho^{-1}\CF^1(\Bv)
		-\nabla\di\Bv+\CF^3(\Bv),|\det\nabla\Phi|(\BI+\BA_{-1}\BB_{-1}(x)^T)\nabla\ph)_{\dws} \nonumber \\
		&\quad =
		(\rho^{-1}\mu\Di\BD(\Bv)-\rho^{-1}\CF^1(\Bv)-\nabla\di\Bv+\CF^3(\Bv),\nabla\ph)_{\dws}+(\CF^4(\Bv),\nabla\ph)_{\dws} \nonumber
\end{align}
for some $\CP^2$, $\CF^3(\Bv)$, and $\CF^4(\Bv)$ satisfying
\begin{align}\label{151016_7}
	&\CF^3(\Bv)= \CR^3\nabla^2\Bv + \CS^2\nabla\Bv, \quad
	\CF^4(\Bv) = \CR^4(\rho^{-1}\mu\Di\BD(\Bv)-\nabla\di\Bv-\rho^{-1}\CF^1(\Bv)+\CF^3(\Bv)), \\
	&\|(\CP^2,\CR^3,\CR^4)\|_{L_\infty(\ws)}\leq C_N M_1, \quad \|(\nabla\CP^2,\nabla\CR^3,\nabla\CR^4,\CS^2)\|_{L_r(\ws)}\leq C_N M_2. \nonumber
\end{align}
In addition, 
\begin{equation}\label{151016_9}
	\jump{\te} =<\jump{\mu\BD(\Bv)\Bn_0},\Bn_0>-\jump{\di\Bv} +\jump{\CF^5(\Bv)} \quad \text{on $\BR_0^N$,} 
\end{equation}
where $\CF^5(\Bv)$ is given by
\begin{equation}\label{151016_10}
	\CF^5(\Bv) = \mu\CR^5\nabla\Bv+\CR^6\nabla\Bv, \quad \|(\CR^5,\CR^6)\|_{L_\infty(\ws)}\leq C_NM_1, \quad
		\|(\nabla\CR^5,\nabla\CR^6)\|_{L_r(\ws)}\leq C_NM_2.
\end{equation}

Summing up \eqref{150612_1}, \eqref{151016_5}, \eqref{151016_8}, and \eqref{151016_9}, we have obtained the following system:
\begin{equation}\label{150612_2}
	\left\{\begin{aligned}
			\la\Bv-\frac{1}{\rho}\Di\BT(\Bv,\te)-\frac{\mu-\mu_0}{\rho}\Di\BD(\Bv)
			+\frac{1}{\rho}\CF^1(\Bv)+\frac{1}{\rho}\CP^1\nabla\te &= \Bf &&
			\text{in $\dws$,} \\
			\jump{\BT(\Bv,\te)\Bn_0}+\jump{(\mu-\mu_0)\BD(\Bv)\Bn_0}+\jump{\CF^2(\Bv)\Bn_0} &= \jump{\Bh} && \text{on $\BR_0^N$,} \\
			\jump{\Bv} &= 0 && \text{on $\BR_0^N$}
	\end{aligned}\right.
\end{equation}
with $\BT(\Bv,\te)=\mu_0\BD(\Bv)-\te\BI$, and also for any $\ph\in\wh{W}_{q'}^1(\ws)$
\begin{align}\label{150612_3}
	&(\rho^{-1}\nabla\te,\nabla\ph)_\dws + (\rho^{-1}\CP^2\nabla\te,\nabla\ph)_\dws \displaybreak[0] \\
		&= (\rho^{-1}\Di(\mu_0\BD(\Bv))-\nabla\di\Bv+\rho^{-1}(\mu-\mu_0)\Di\BD(\Bv)
			-\rho^{-1}\CF^1(\Bv)+\CF^3(\Bv)+\CF^4(\Bv),\nabla\ph)_{\dws}, \nonumber \displaybreak[0] \\
	&\jump{\te} = <\jump{\mu_0\BD(\Bv)\Bn_0},\Bn_0>-\jump{\di\Bv} + <\jump{(\mu-\mu_0)\BD(\Bv)\Bn_0},\Bn_0> +\jump{\CF^5(\Bv)}
		\quad \text{on $\BR_0^N$.} \label{160119_1}
\end{align}
%
%

From now on, we solve \eqref{150612_2}, \eqref{150612_3}, \eqref{160119_1}.
Let $\te_1= K_I(\Bv)$ given by the solution to \eqref{160125_1}-\eqref{160125_1-2} with $\mu=\mu_0$.
Setting $\te = K_I(\Bv)+\te_2(\Bv)$ in \eqref{150612_3}-\eqref{160119_1},
we have the weak problem for $\te_2=\te_2(\Bv)$ as follows:
\begin{align}
	&(\rho^{-1}\nabla\te_2,\nabla\ph)_\dws + (\rho^{-1}\CP^2\nabla\te_2,\nabla\ph)_\dws \label{150824_3} \\
		&\quad =(\rho^{-1}(\mu-\mu_0)\Di\BD(\Bv)-\rho^{-1}\CF^1(\Bv)+\CF^3(\Bv)+\CF^4(\Bv)-\rho^{-1}\CP^2\nabla K_I(\Bv),\nabla\ph)_\dws \nonumber \\
	&\jump{\te_2} = <\jump{(\mu-\mu_0)\BD(\Bv)\Bn_0},\Bn_0> +\jump{\CF^5(\Bv)} \quad \text{on $\BR_0^N$} \label{150824_3-2}
\end{align}
for any $\ph\in \wh{W}_{q'}^1(\ws)$.
Substituting $\te = K_I(\Bv) + \te_2(\Bv)$ in the problem \eqref{150612_2},
we have
\begin{equation}\label{150824_6}
	\left\{\begin{aligned}
		\la\Bv -\rho^{-1}\Di\BT(\Bv,K_I(\Bv)) +\CU^1(\Bv) 
			&= \Bf && \text{in $\dws$,} \\
		\jump{\BT(\Bv,K_I(\Bv))\Bn_0} +\jump{\CU^2(\Bv)\Bn_0} &= \jump{\Bh} && \text{on $\BR_0^N$,} \\
		\jump{\Bv} &= 0 && \text{on $\BR_0^N$,}
	\end{aligned}\right.
\end{equation}
where
\begin{align*}
	\CU^1(\Bv)
		&=-\rho^{-1}(\mu-\mu_0)\Di\BD(\Bv) + \rho^{-1}\CF^1(\Bv)
		+\rho^{-1}\CP^1\nabla K_I(\Bv)
		+\rho^{-1}(\BI+\CP^1)\nabla\te_2(\Bv), \\
	\CU^2(\Bv) &= (\mu-\mu_0)\BD(\Bv) + \CF^2(\Bv) -\te_2(\Bv)\BI = \CF^2(\Bv)-\{<(\mu-\mu_0)\BD(\Bv)\Bn_0,\Bn_0> +\CF^5(\Bv)\}\BI.
\end{align*}

At this point, we introduce a result about the unique solvablity of the weak problem:
\begin{align}\label{weak:pert}
	(\rho^{-1}\nabla\te,\nabla\ph)_{\dws} + (\rho^{-1}\CP^2\nabla\te,\nabla\ph)_{\dws} &= (\Bf,\nabla\ph)_{\dws} \quad
	\text{for all $\ph\in\wh{W}_{q'}^1(\ws)$,}  \\
	\jump{\te} &= \jump{g} \quad \text{on $\BR_0^N$.} \label{weak:pert-2}
\end{align}

\begin{lemm}\label{lemm:pert}
	Let $1<q<\infty$.
	Then there exists a constant $M_1\in(0,1)$ and an operator
	\begin{equation*}
		\Psi\in\CL(L_q(\dws)^N\times W_q^1(\dws), W_q^1(\dws)+\wh{W}_q^1(\ws))
	\end{equation*}
	such that, for any $\Bf\in L_q(\dws)^N$ and $g\in W_q^1(\dws)$,
	$\te=\Psi(\Bf,g)$ is a unique solution to \eqref{weak:pert}-\eqref{weak:pert-2},
	which possesses the estimate:
	$\|\nabla\te\|_{L_q(\dws)}\leq C_{N,q}(\|\Bf\|_{L_q(\dws)}+\|g\|_{W_q^1(\dws)})$
	with a positive constant $C_{N,q}$ independent of $M_2$.
\end{lemm}

\begin{proof}
Since the weak problem \eqref{160125_1}-\eqref{160125_1-2} is uniquely solvable,
we can prove Lemma \ref{lemm:pert} by the small perturbation method,
so that we may omit the detailed proof.
\end{proof}

By Lemma \ref{lemm:pert}, we have $\te_2(\Bv)=\Psi(\Bf,g)$ with
\begin{align*}
	\Bf &= \rho^{-1}(\mu-\mu_0)\Di\BD(\Bv)-\rho^{-1}\CF^1(\Bv)+\CF^3(\Bv)+\CF^4(\Bv)-\rho^{-1}\CP^2\nabla K_I(\Bv), \\
	g &= <(\mu-\mu_0)\BD(\Bv)\Bn_0,\Bn_0> + \CF^5(\Bv).
\end{align*}
We solve the problem \eqref{150824_6} by using Theorem \ref{theo:ws}.
Substituting $\Bv = \BS_I(\la)G_{\CR,\la}(\Bf,\Bh)$ in \eqref{150824_6} yields that
\begin{equation}\label{bent:pert}
	\left\{\begin{aligned}
		\la\Bv-\rho^{-1}\Di\BT(\Bv,K_I(\Bv)) &= \Bf-\CU^1(\BS_I(\la)G_{\CR,\la}(\Bf,\Bh)) && \text{in $\dws$,} \\
		\jump{\BT(\Bv,K_I(\Bv))\Bn_0} &= \jump{\Bh-\CU^2(\BS_I(\la)G_{\CR,\la}(\Bf,\Bh))\Bn_0} && \text{on $\BR_0^N$,} \\
		\jump{\Bv} &= 0 && \text{on $\BR_0^N$.}
	\end{aligned}\right.
\end{equation}
Set $\CV(\la)(\Bf,\Bh) = (\CV^1(\la)(\Bf,\Bh),\CV^2(\la)(\Bf,\Bh))$ with 
$\CV^i(\la)(\Bf,\Bh)=\CU^i(\BS_I(\la)G_{\CR,\la}(\Bf,\Bh))$ $(i=1,2)$ and 
\begin{equation*}
	Y_{\CR,q}^{\la}(\dws) = \{(\Bf,\nabla\Bh,\la^{1/2}\Bh) \mid (\Bf,\Bh)\in Y_{\CR,q}(\dws)\} \quad (\la\neq0).
\end{equation*}
Then, for each $\la\neq 0$, $\ph_\la(\Bf,\Bh):=G_{\CR,\la}(\Bf,\Bh)$
is a bijection from $ Y_{\CR,q}(\dws)$ onto $Y_{\CR,q}^\la(\dws)$.
Formally, if there is the inverse operator of $(I-\ph_\la\CV(\la)\ph_\la^{-1})$, then 
	$\Bv = \BS_I(\la)G_{\CR,\la}\ph_\la^{-1}(I-\ph_\la\CV(\la)\ph_\la^{-1})^{-1}\ph_\la(\Bf,\Bh)$
solves \eqref{150824_6} since $\ph_\la^{-1}(I-\ph_\la\CV(\la)\ph_\la^{-1})^{-1}\ph_\la=(I-\CV(\la))^{-1}$.

In what follows, we show the invertibility above and the $\CR$-boundedness of the inverse operator. 
To this end, we estimate the remainder terms on the right-hand sides of \eqref{bent:pert}.
We combine Proposition \ref{lemm:Shi13} for $\dot\Om=\dws$
with \eqref{151016_1}, \eqref{151016_6}, \eqref{151016_7}, and \eqref{151016_10} in order to obtain
\begin{align}\label{150825_1}
	&\|\CF^i(\Bv)\|_{L_q(\dws)}
		\leq
			\ga_3(M_1+\si)\|\nabla^2\Bv\|_{L_q(\dws)} + \ga_{\si,M_2}\|\nabla\Bv\|_{L_q(\dws)} \quad (i=1,3,4),  \\
	&\|\CF^i(\Bv)\|_{L_q(\dws)}
		\leq \ga_3M_1\|\nabla\Bv\|_{L_q(\dws)} \quad (i=2,5), \nonumber \\
	&\|\nabla\CF^i(\Bv)\|_{L_q(\dws)}
		\leq \ga_3(M_1+\si)\|\nabla^2\Bv\|_{L_q(\dws)} + \ga_{\si,M_2}\|\nabla\Bv\|_{L_q(\dws)} \quad (i=2,5). \nonumber \\
	&\|\CP^i\nabla K_I(\Bv)\|_{L_q(\dws)} 
		\leq\ga_3M_1 \|\nabla\Bv\|_{W_q^1(\dws)} \quad (i=1,2). \nonumber
\end{align}
Here and subsequently, $\ga_3$ is a generic constant depending, at most,
on $N$, $q$, $r$, $\rho_+$, $\rho_-$, $\mu_{+1}$, $\mu_{+2}$, $\mu_{+ 2}$, and $\mu_{- 2}$; 
$\ga_{\si,M_2}$ is a generic constant depending, at most, on $M_2$, $\si$, $N$, $q$, $r$, $\rho_+$, $\rho_-$, $\mu_{+1}$, $\mu_{+2}$, $\mu_{+ 2}$, and $\mu_{- 2}$.
In addition, by Lemma \ref{lemm:pert}, \eqref{150825_1}, and \eqref{visco:3} together with Proposition \ref{lemm:Shi13}, we have
\begin{equation}\label{150825_2}
	\|(I+\CP^1)\nabla\te_2(\Bv)\|_{L_q(\dws)}
		\leq \ga_3(M_1+\si)\|\nabla^2\Bv\|_{L_q(\dws)}+\ga_{\si,M_2}\|\nabla\Bv\|_{L_q(\dws)}.   
\end{equation}
Define operators $\BV^i(\la)$, $i=1,2$, as 
	$\BV^i(\la)\BH = \CU^i(\BS_I(\la)\BH)$ with $\BH=(H_1,H_2,H_3)\in \CY_{\CR,q}(\dws)$. 
Then we have $\CV^i(\Bf,\Bh)=\BV^i(\la)G_{\CR,\la}(\Bf,\Bh)$ and  
have, by 
Proposition \ref{lemm:multi}, 
\eqref{150825_1}, \eqref{150825_2}, and Theorem \ref{theo:ws},
\begin{align}\label{150825_3}
	&\CR_{\CL(\CY_{\CR,q}(\dws),L_q(\dws)^N)}
		\Big(\Big\{\Big(\la\frac{d}{d\la}\Big)^l\BV^1(\la)\mid \la\in\Si_{\ep,\la_0}\Big\}\Big)^{1/q}  
		\leq \ga_1\left(\ga_3(M_1+\si)+\ga_{\si,M_2}\la_0^{-1/2}\right), \\
	&\CR_{\CL(\CY_{\CR,q}(\dws),L_q(\dws)^{N^2+N})}
		\Big(\Big\{\Big(\la\frac{d}{d\la}\Big)^l\left(\nabla\BV^2(\la),\la^{1/2}\BV^2(\la)\right)\mid \la\in\Si_{\ep,\la_0}\Big\}\Big) \nonumber \\
		&\quad\leq \ga_1\left(\ga_3(M_1+\si)+\ga_{\si,M_2}\la_0^{-1/2}\right) \nonumber
\end{align}
for $l=0,1$ and for any $\la_0>0$. In fact, since $\CF^1$ is linear, we have,
for any $\la_0>0$ and for any $n\in\BN$, $\{\la_j\}_{j=1}^n\subset\Si_{\ep,\la_0}$, and  $\{\BH_j\}_{j=1}^n\subset\CY_{\CR,q}(\dws)$,
\begin{align*}
	&\Big(\int_0^1\Big\|\sum_{j=1}^n r_j(u)\CF^1(\BS_I(\la_j)\BH_j)\Big\|_{L_q(\dws)}^q\intd u\Big)^{1/q} 
	\leq\ga_3(M_1+\si)\Big(\int_0^1\Big\|\sum_{j=1}^n r_j(u)\nabla^2\BS_I(\la_j)\BH_j\Big\|_{L_q(\dws)}^q \intd u\Big)^{1/q} \displaybreak[0] \\
		&\quad +\ga_{\si,M_2}\Big(\int_0^1\Big\|\sum_{j=1}^n r_j(u)\nabla\BS_I(\la_j)\BH_j\Big\|_{L_q(\dws)}^q\intd u\Big)^{1/q} \displaybreak[0] \\
	&\leq \ga_1\big(\ga_3(M_1+\si)+\ga_{\si,M_2}\la_0^{-1/2}\big)
	\Big(\int_0^1\Big\|\sum_{j=1}^n r_j(u)\BH_j\Big\|_{L_q(\dws)}^q\intd u\Big)^{1/q}.
\end{align*}
It holds, by the linearity of $\CF^1$, that
\begin{equation*}
	\la\frac{d}{d\la}\CF^1\left(\BS_I(\la)\BH\right) = 
		\CF^1\left(\la\frac{d}{d\la}\BS_I(\la)\BH\right),
\end{equation*}
so that we have 
in the same manner as above
\begin{equation*}
	\CR_{\CL(\CY_{R,q}(\dws),L_q(\dws))}
	\Big(\Big\{\Big(\la\frac{d}{d\la}\Big)\CF^1(\BS_I(\la)(\cdot))\mid\la\in\Si_{\ep,\la_0}\Big\}\Big)
	\leq \ga_1\big(\ga_3(M_1+\si)+\ga_{\si,M_2}\la_0^{-1/2}\big).
\end{equation*}
Analogously, we can obtain estimates for $\CR$-bound of the other terms,
and thus we have \eqref{150825_3}. 
Setting $\BV(\la)\BH=(\BV^1(\la)\BH,\BV^2(\la)\BH)$ for $\BH\in\CY_{\CR,q}(\dws)$
furnishes that
\begin{align}\label{150825_5}
	&\CV(\la)(\Bf,\Bh) = \BV(\la)G_{\CR,\la}(\Bf,\Bh)\in Y_{\CR,q}(\dws) \quad \text{for $(\Bf,\Bh)\in Y_{\CR,q}(\dws)$}, \\
	&\CR_{\CL(\CY_{\CR,q}(\dws))}\Big(\Big\{\Big(\la\frac{d}{d\la}\Big)^l(G_{\CR,\la}\BV(\la)) \mid \la\in\Si_{\ep,\la_0}\Big\}\Big)
		\leq \ga_1\big(\ga_3(M_1+\si)+\ga_{\si,M_2}\la_0^{-1/2}\big) \quad (l=0,1). \nonumber
\end{align}
If we choose $\si$ and $M_1$ so small that $\ga_1\ga_3\si\leq 1/8$ and $\ga_1\ga_3M_1\leq 1/8$
and if we choose $\la_0\geq 1$ so large that $\ga_{\si,M_2}\la_0^{-1/2}\leq 1/4$,
then  we have by \eqref{150825_5}
\begin{equation}\label{150825_6}
	\CR_{\CL(\CY_{\CR,q}(\dws))}
		\Big(\Big\{\Big(\la\frac{d}{d\la}\Big)^l(G_{\CR,\la}\BV(\la)) \mid \la\in\Si_{\ep,\la_0}\Big\}\Big)\leq \frac{1}{2}
		\quad (l=0,1).
\end{equation}
Since it holds by \eqref{150825_5}, \eqref{150825_6} that 
\begin{align*}
	&\|\ph_\la\CV(\la)\ph_\la^{-1}(\Bf,\nabla\Bh,\la^{1/2}\Bh)\|_{\CY_{\CR,q}(\dws)}
	= \|G_{\CR,\la}\CV(\la)(\Bf,\Bh)\|_{\CY_{\CR,q}(\dws)}  \\
 	&= \|G_{\CR,\la}\BV(\la)G_{\CR,\la}(\Bf,\Bh)\|_{\CY_{\CR,q}(\dws)} 
	\leq \frac{1}{2}\|(\Bf,\nabla\Bh,\la^{1/2}\Bh)\|_{\CY_{\CR,q}(\dws)},
\end{align*}
there exists the inverse mapping $(I-\ph_\la\CV(\la)\ph_\la^{-1})^{-1}\in\CL(Y_{\CR,q}^\la(\dws))$ for any $\la\in\Si_{\ep,\la_0}$.
In addition, $(I-G_{\CR,\la}(\la)\BV(\la))^{-1}=\sum_{j=0}^\infty(G_{\CR,\la}\BV(\la))^j$ exists by \eqref{150825_6} and
satisfies the estimate:
\begin{equation}\label{150825_7}
	\CR_{\CL(\CY_{\CR,q}(\dws))}\Big(\Big\{\Big(\la\frac{d}{d\la}\Big)^l(I-G_{\CR,\la}\BV(\la))^{-1} \mid \la\in\Si_{\ep,\la_0}\Big\}\Big)
		\leq 2 \quad (l=0,1).
\end{equation}
If we set $\Bv = \BS_I(\la)G_{\CR,\la}\ph_\la^{-1}(I-\ph_\la\CV(\la)\ph_\la^{-1})^{-1}\ph_{\la}(\Bf,\Bh)$,
then $\Bv$ is a solution to \eqref{150824_6} as mentioned above. 
Noting that $\ph_\la\CV(\la)\ph_\la^{-1}G_{\CR,\la}(\Bf,\Bh)=G_{\CR,\la}\BV(\la)G_{\CR,\la}(\Bf,\Bh)$ by \eqref{150825_5},
we see that 
\begin{equation*}
G_{\CR,\la}\ph_\la^{-1}(I-\ph_\la\CV(\la)\ph_\la^{-1})^{-1}\ph_{\la}(\Bf,\Bh)
= \sum_{j=0}^\infty\left(\ph_\la\CV(\la)\ph_\la^{-1}\right)^jG_{\CR,\la}(\Bf,\Bh)
=(I-G_{\CR,\la}\BV(\la))^{-1}G_{\CR,\la}(\Bf,\Bh).
\end{equation*} 
Set $\CS_I(\la)=\BS_I(\la)(I-G_{\CR,\la}\BV(\la))^{-1}$, and then $\Bv=\CS_I(\la)G_{\CR,\la}(\Bf,\Bh)$ is a solution to \eqref{150824_6}
for any $\la\in\Si_{\ep,\la_0}$ and $(\Bf,\Bh)\in Y_{\CR,q}(\dws)$.
Furthermore, by \eqref{150825_7} and Theorem \ref{theo:ws}, we have
\begin{align}\label{150825_9}
	&\CS_I(\la) \in \Hol(\Si_{\ep,\la_0},\CL(\CY_{\CR,q}(\dws),W_q^2(\dws)^N)), \\
	&\CR_{\CL(\CY_{\CR,q}(\dws),L_q(\dws)^{\wt N})}
		\Big(\Big\{\Big(\la\frac{d}{d\la}\Big)^l \left(R_\la\CS_I(\la)\right) \mid \la\in\Si_{\ep,\la_0}\Big\}\Big)\leq \ga_2
		\quad (l=0,1). \nonumber
\end{align}

The uniqueness of solutions to \eqref{150824_6} can be proved in the same manner as in \cite[Section 4]{Shibata14}.

Setting $\wt\Bu=\BA_{-1}^T\Bv\circ\Phi^{-1}=[\BA_{-1}^T\CS_I(\la)G_{\CR,\la}(\Bf,\Bh)]\circ\Phi^{-1}$ and noting $\BA_{-1}^T=(\BA_{-1})^{-1}$, 
we see that $\wt\Bu$ is a unique solution to \eqref{reduce:2}.
Recall that $\Bf=\BA_{-1}\wt\Bf\circ\Phi$ and
$\Bh = |(\BA_{-1}+\BB_{-1}(x))^T\Bn_0|(\BA_{-1}+\BB_{-1}(x))^{-T}\wt\Bh\circ\Phi$,
and set $\BE(x) = |(\BA_{-1}+\BB_{-1}(x))^T\Bn_0|(\BA+\BB(x))^T$ in view of Remark \ref{rema:AB}.
Observing that
\begin{equation*}
	G_{\CR,\la}(\Bf,\Bh) = (\Bf,\nabla\Bh,\la^{1/2}\Bh) 
	=\Big(\BA_{-1}\wt\Bf\circ\Phi,(\nabla\BE(x))\wt\Bh\circ\Phi+\BE(x)[(\nabla\wt\Bh)\circ\Phi]\nabla\Phi,\BE(x)(\la^{1/2}\wt\Bh)\circ\Phi\Big),
\end{equation*}
we define, for $\BH=(H_1,H_2,H_3,H_4)\in\CZ_{\CR,q}(\dot\Om)$, an operator $\wt\BS_I(\la)$ by 
\begin{equation*}
	\wt\BS_I(\la)\BH 
		= \big[\BA_{-1}^T\CS_I(\la)(\BA_{-1}H_1\circ \Phi,
		\big(\nabla\BE(x))H_4\circ\Phi+\BE(x)(H_2\circ\Phi)\nabla\Phi,\BE(x)H_3\circ\Phi\big)\big]\circ\Phi^{-1}.
\end{equation*}
Then we can show that $\wt\BS_I(\la)$ satisfies \eqref{150825_12}
by \eqref{150825_9} and Proposition \ref{lemm:Shi13} with $\si=1$,
and also $\wt\Bu = \wt\BS_I(\la)H_{\CR,\la}(\wt\Bf,\wt\Bh)$ solves \eqref{reduce:2} uniquely.
This completes the proof of Theorem \ref{theo:bent}.


\section{A proof of Theorem \ref{theo:redu1}}\label{sec:proof}
As was discussed in Subsection \ref{sub:main}, our main result Theorem \ref{theo:main}
follows from Theorem \ref{theo:redu1},
so that we prove Theorem \ref{theo:redu1} in this section.

\subsection{Some preparations for the proof of Theorem \ref{theo:redu1}}
First we state several properties of uniform $W_r^{2-1/r}$ domain (cf. \cite[Proposition 6.1]{ES13}, \cite{KS1}).

\begin{prop}\label{prop:uniform}
	Let $N<r<\infty$ and let $\Om_\pm$ be uniform $W_r^{2-1/r}$ domains in $\ws$.
	Let $M_1$ the number given in Section $\ref{sec:bent}$.
	Then there exist constants $M_2>0$, $0<d^i<1$ $(i=1,\dots,5)$,
	at most countably many $N$-vectors of functions $\Phi_j^i\in W_r^2(\ws)^N$ $(j\in\BN$, $i=1,2,3)$,
	$x_j^1\in \Ga$, $x_j^2\in\Ga_+$, $x_j^3\in\Ga_-$, $x_j^4\in\Om_+$, and $x_j^5\in\Om_-$ such that
	the following assertions hold:
	\begin{enumerate}[$(1)$]
		\item\label{uniform1}
			The maps: $\ws\ni x \mapsto \Phi_j^i(x)\in\ws$ $(j\in\BN$, $i=1,2,3)$ are bijective such that
			$\nabla \Phi_j^i=\BA_j^i +\BB_j^i(x)$ and $\nabla(\Phi_j^i)^{-1}=\BA_{j,-1}^i+\BB_{j,-1}^i(x)$, where
			$\BA_j^i$, $\BA_{j,-1}^i$ are $N\times N$ constant orthonormal matrices and 
			$\BB_j^i(x)$, $\BB_{j,-1}^i(x)$ are $N\times N$ matrices of $W_r^1(\ws)$ functions which
			satisfy the conditions: 
			$\|(\BB_j^i,\BB_{j,-1}^i)\|_{L_\infty(\ws)}\leq M_1$ and
			$\|\nabla(\BB_j^i,\BB_{j,-1}^i)\|_{L_r(\ws)}\leq M_2$.
		\item\label{uniform2}
			$\Om = \big\{\bigcup_{i=1,2,3}\bigcup_{j=1}^\infty(\Phi_j^i(H^i)\cap B_{d^i}(x_j^i))\big\}\cup
			\big\{\bigcup_{i=4,5}\bigcup_{j=1}^\infty B_{d^i}(x_j^i)\big\}$ with
			$H^1=\ws$, $H^2=\uhs$, and $H^3=\lhs$, where
			$\Phi_j^i(\uhs)\cap B_{d^i}(x_j^i) = \Om_+\cap B_{d^i}(x_j^i)$ $(i=1,2)$,
			$\Phi_j^i(\lhs)\cap B_{d^i}(x_j^i) = \Om_- \cap B_{d^i}(x_j^i)$ $(i=1,3)$,
			$B_{d^4}(x_j^4)\subset \Om_+$, $B_{d^5}(x_j^5)\subset\Om_-$,
			and $\Phi_j^i(\BR_0^N)\cap B_{d^i}(x_j^i)=\Ga^i\cap B_{d^i}(x_j^i)$ $(i=1,2,3)$.
			Here and subsequently, we set
			$\Ga^1=\Ga$, $\Ga^2=\Ga_+$, and $\Ga^3=\Ga_-$ for the notational convenience.
		\item\label{uniform3}
			There exist $C^\infty$ functions $\zeta_j^i$ and $\wt\zeta_j^i$ $(i=1,\dots,5$, $j\in\BN)$
			such that $\|(\zeta_j^i,\wt\zeta_j^i)\|_{W_\infty^2(\ws)}\leq c_0$,
			$0\leq\zeta_j^i,\wt\zeta_j^i\leq1$,
			$\supp\zeta_j^i$, $\supp\wt\zeta_j^i\subset B_{d^i}(x_j^i)$,
			$\wt\zeta_j^i=1$ on $\supp\zeta_j^i$,
			$\sum_{i=1,\dots,5}\sum_{j=1}^\infty\zeta_j^i=1$ on $\overline\Om$, and
			$\sum_{j=1}^\infty\zeta_j^i=1$ on $\Ga^i$ $(i=1,2,3)$.
			Here $c_0$ is a positive constant depending on $M_2$, $N$, and $r$, but independent of $j\in\BN$.
		\item\label{uniform4}
			There exists a natural number $L \geq 2$ such that any $L+1$ distinct sets of
			$\{B_{d^i}(x_j^i)\mid i=1,\dots,5,\enskip j\in\BN\}$ have an empty intersection.
	\end{enumerate}
\end{prop}

Since $\mu_\pm(x)$ is uniformly continuous in $\ws$ as was assumed in the assumption (c),
choosing $d^i>0$ smaller, if necessary, allows us to assume that $|\mu_\pm(x)-\mu_\pm(x_j^i)|\leq M_1$ for any $x\in B_{d^i}(x_j^i)$
with $i=1,\dots,5$ and $j\in\BN$.
Moreover, after choosing $M_2$ and $d^i$ according to $M_1$ in Proposition \ref{prop:uniform},
we choose $M_2$ again so large that $\|\nabla\mu_\pm\|_{L_r(B_{d^i}(x_j^i))}\leq M_2$. 
%
%
%
%
Here and in the following, constants denoted by $C$ are independent of $j\in\BN$.
In view of \eqref{normal}, we may assume that unit normal vectors $\Bn_j^i$ to $\Ga_j^i=\Phi_j^i(\BR_0^N)$
($i=1,2,3$, $j\in\BN$) are defined on $\ws$ together with $\|\Bn_j^i\|_{L_\infty(\ws)}= 1$,
and also they satisfy, by Proposition \ref{prop:uniform} (1), the conditions: $\|\nabla\Bn_j^i\|_{L_r(\ws)}\leq CM_2$.
Note that $\Bn=\Bn_j^1$ on $B_{d^1}(x_j^1)\cap \Ga$ and points from $\Om_+$ to $\Om_-$, 
and besides, the unit outward normal $\Bn_\pm$ to $\Ga_\pm$ satisfy $\Bn_+=\Bn_j^2$ on $B_{d^2}(x_j^2)\cap \Ga_+$
and $\Bn_-=\Bn_j^3$ on $B_{d^3}(x_j^3)\cap \Ga_-$, respectively.

Summing up the above properties, we suppose in this section that
\begin{equation}\label{assump:5.1}
	\mu_{\pm 1}\leq \mu_\pm(x_j^i)\leq\mu_{\pm 2}, \quad |\mu_\pm(x)-\mu_{\pm}(x_j^i)|\leq M_1 \quad (x\in B_{d^i}(x_j^i)), \quad
	\|\nabla\mu_\pm\|_{L_r(B_{d^i}(x_j^i))}\leq M_2.
\end{equation}

Let $B_j^i=B_{d^i}(x_j^i)$ with $i=1,\dots,5$ and $j\in\BN$ for short. 
Then, by the finite intersection property stated in Proposition \ref{prop:uniform} \eqref{uniform4},
we see that, for any $s\in[1,\infty)$, there is a positive constant $C_{s,L}$ such that,
for any $f\in L_s(G)$ with an open set $G$ of $\ws$ and for $i=1,\dots,5$,
\begin{equation}\label{finite}
	\Big(\sum_{j=1}^\infty\|f\|_{L_s(G\cap B_j^i)}^s\Big)^{1/s} \leq C_{s,L}\|f\|_{L_s(G)}.
\end{equation}
In fact, when $1\leq s<\infty$,
\begin{align*}
	\sum_{j=1}^\infty \|f\|_{L_s(G\cap B_j^i)}^s = \int_{G} \Big(\sum_{j=1}^\infty \chi_{B_j^i}(x)\Big)|f(x)|^s \intd x
		\leq \Big\|\sum_{j=1}^\infty \chi_{B_j^i}\Big\|_{L_\infty(\ws)}\|f\|_{L_s(G)}^s \leq L\|f\|_{L_s(G)}^s.
\end{align*}


Next we prepare two lemmas used to construct parametrixes.

\begin{lemm}\label{prepa:1}
	Let $X$ be a Banach space and $X^*$ its dual space,
	while $\|\cdot\|_X$, $\|\cdot\|_{X^*}$, and $<\cdot,\cdot>$ be the norm of $X$, the norm of $X^*$,
	and the duality pairing between of $X$ and $X^*$, respectively.
	Let $n\in\BN$, $l=1,\dots,n$, and $\{a_l\}_{l=1}^n\subset\BC$,
	and let $\{f_j^l\}_{j=1}^\infty$ be sequences in $X^*$ and $\{g_j^l\}_{j=1}^\infty$, $\{h_j\}_{j=1}^\infty$ be sequences of positive numbers.
	Assume that there exist maps $\CN_j:X\to[0,\infty)$ such that
	\begin{equation*}
		|<f_j^l,\ph>|\leq M_3 g_j^l\CN_j(\ph) \quad (l=1,\dots,n), \quad
		\Big|\Big<\sum_{l=1}^n a_l f_j^l,\ph\Big>\Big| \leq M_3h_j\CN_j(\ph)
	\end{equation*}
	for any $\ph\in X$ with some positive constant $M_3$ independent of $j\in\BN$ and $l=1,\dots,n$.
	If
	\begin{equation*}
		\sum_{j=1}^\infty \left(g_j^l\right)^q < \infty, \quad 
		\sum_{j=1}^\infty \left(h_j\right)^q < \infty, \quad 
		\sum_{j=1}^\infty\left(\CN_j(\ph)\right)^{q'} \leq \left(M_4\|\ph\|_{X}\right)^{q'}
	\end{equation*}
	with $1<q<\infty$ and $q'=q/(q-1)$ for some positive constant $M_4$,
	then the infinite sum $f^l = \sum_{j=1}^\infty f_j^l$ exists in the strong topology of $X^*$ and
	\begin{equation}\label{160201_1}
		\|f^l\|_{X^*} \leq M_3M_4\Big(\sum_{j=1}^\infty\big(g_j^l\big)^q \Big)^{1/q}, \quad
		\Big\|\sum_{l=1}^n a_l f^l \Big\|_{X^*} \leq M_3 M_4 \Big(\sum_{j=1}^\infty\big(h_j\big)^q\Big)^{1/q}.
	\end{equation}
\end{lemm}

\begin{proof}
Let $F_m^l = \sum_{j=1}^m f_j^l$.
We can show that $\{F_m\}_{m=1}^\infty$ is a Cauchy sequence in $X^*$, which implies the existence of $f^l$.
Then the estimates \eqref{160201_1} follow immediately.
\end{proof}

The following lemma follows from Lemma \ref{prepa:1} and \eqref{finite}.

\begin{lemm}\label{prepa:2}
	Let $1<q<\infty$, $q'=q/(q-1)$, $i=1,\dots,5$, and $m\in\BN_0$.
	Let $\{f_j\}_{j=1}^\infty$ be a sequence of $W_q^m(\dot\Om)$ and
	let $\{g_j^l\}_{j=1}^\infty$ be sequences of positive numbers for $l=0,1\dots,m$.
	Assume that
	\begin{equation*}
		\sum_{j=1}^\infty \left(g_j^l\right)^q < \infty, \quad
		|(\nabla^l f_j,\ph)_{\dot\Om}| \leq M_5 g_j^l \|\ph\|_{L_{q'}(\dot\Om\cap B_j^i)} \quad 
		\text{for any $\ph\in L_{q'}(\dot\Om)$}
	\end{equation*}
	with some positive constant $M_5$ independent of $j\in\BN$ and $l=0,1\dots,m$.
	Then $f=\sum_{j=1}^\infty f_j$ exists in the strong topology of $W_q^m(\dot\Om)$ and
		$\|\nabla^l f\|_{L_q(\dot\Om)} \leq C_{q,L} M_5 (\sum_{j=1}^\infty (g_j^l)^q )^{1/q}$
	with some positive contatnt $C_{q,L}$.
\end{lemm}

\begin{rema}\label{rema:normal}
At this point, we have a remark on unit normals $\Bn$, $\Bn_+$.
We can see $\Bn$, $\Bn_+$ as vector functions defined on $\ws$ through the relations:
	$\Bn=\sum_{j=1}^\infty\zeta_j^1\Bn_j^1$, $\Bn_+ = \sum_{j=1}^\infty\zeta_j^2\Bn_j^2$.
Then it is clear that $\Bn = \Bn_j^1 $ in $B_j^1\cap \Ga$ and $\Bn_+ = \Bn_j^2$ in $B_j^2\cap \Ga_+$.
Moreover, we have $\|f\Bn\|_{L_q(\dot\Om)}\leq C\|f\|_{L_q(\dot\Om)}$ for any function $f\in L_q(\dot\Om)$.
In fact, for $f\in L_q(\dot\Om)$ and $\ph\in L_{q'}(\dot\Om)^N$, 
\begin{align*}
	|(f\Bn,\ph)_{\dot\Om}| \leq \sum_{j=1}^\infty\|f\|_{L_q(\dot\Om\cap B_j^1)}\|\ph\|_{L_{q'}(\dot\Om\cap B_j^1)}
	\leq\Big(\sum_{j=1}^\infty\|f\|_{L_q(\dot\Om\cap B_j^1)}^q\Big)^{1/q}\Big(\sum_{j=1}^\infty\|\ph\|_{L_{q'}(\dot\Om\cap B_j^1)}^{q'}\Big)^{1/q'},
\end{align*}
which, combined with \eqref{finite}, furnishes that
$|(f\Bn,\ph)_{\dot\Om}|\leq C\|f\|_{L_q(\dot\Om)}\|\ph\|_{L_{q'}(\dot\Om)}$.
This inequality implies that the required estimate holds.
Similarly, for $g\in W_q^1(\dot\Om)$, we can prove
$\|g\nabla\Bn\|_{L_q(\dot\Om)}\leq C\|g\|_{W_q^1(\dot\Om)}$, $\|g\Bn\|_{W_q^1(\dot\Om)} \leq C\|g\|_{W_q^1(\dot\Om)}$
with help of Lemma \ref{lemm:Shi13}.
It is clear that we can replace $\Bn$ by $\Bn_+$ in the above inequalities.
\end{rema}

\subsection{Local solutions}
In view of \eqref{assump:5.1}, we define local viscosity coefficients $\nu_{\pm j}^i(x)$ by
\begin{equation*}
	\nu_{\pm j}^i(x) = \left(\mu_\pm(x)-\mu_\pm(x_j^i)\right)\wt\zeta_j^i(x)+\mu_\pm(x_j^i).
\end{equation*}
Note that $M_1\leq (1/2)\min(\mu_{+1},\mu_{-1},\mu_{+2},\mu_{-2})$ as was stated in Theorem \ref{theo:bent}.
Then, 
using \eqref{assump:5.1} and setting $\mu_{\pm j}^i=\mu_\pm(x_j^i)$, we have
\begin{equation}\label{151022_2}
	\frac{1}{2}\mu_{\pm1} \leq \nu_{\pm j}^i(x)\leq \frac{3}{2}\mu_{\pm 2}, \quad
	|\nu_{\pm j}^i(x)-\mu_{\pm j}^i|\leq M_1 \quad (x\in\ws), \quad
	\|\nabla\nu_{\pm j}^i\|_{L_r(\ws)}\leq C_{M_2,r}
\end{equation}
with $\mu_{\pm 1}\leq \mu_{\pm j}^i\leq \mu_{\pm 2}$.
In fact, 
	$\|\nabla\nu_{\pm j}^i\|_{L_r(\ws)} \leq C(\|\mu_{\pm}-\mu_{\pm_j}^i\|_{L_r(B_j^i)} + \|\nabla\mu_{\pm}\|_{L_r(B_j^i)})$,
which, combined with the estimate:
\begin{align*}
	&\|\mu_{\pm}-\mu_{\pm_j}^i\|_{L_r(B_j^i)}^r 
		=\int_{B_{d^i}(0)}\Big|\mu_{\pm}(x+x_j^i)-\mu_\pm(x_j^i)\Big|^r\intd x  \\
	&= \int_{B_{d^i}(0)}\Big|\int_0^1(\nabla\mu_\pm)(\te x+ x_j^i)\cdot x\intd \te\Big|^r\intd x 
		\leq (d^i)^r \Big\|\int_0^1(\nabla\mu_\pm)(\te \cdot+ x_j^i)\intd \te\Big\|_{L_r(B_{d^i}(0))}^r \\
		&\leq (d^i)^r \Big(\int_0^1\left\|(\nabla\mu_\pm)(\te \cdot+ x_j^i)\right\|_{L_r(B_{d^i}(0))}\intd \te\Big)^r
		\leq (d^i)^r\|\nabla\mu_\pm\|_{L_r(B_j^i)}^r\Big(\int_0^1 \frac{\intd\te}{\te^{N/r}}\Big)^r,
\end{align*}
furnishes that $\|\nabla\nu_{\pm j}^i\|_{L_r(\ws)} \leq C_{M_2,r}$.
The condition \eqref{151022_2} implies that $\nu_{\pm j}^i(x)$ satisfy \eqref{visco:1}.

Set $\CH_j^0=\Phi_j^1(\ws)$,
$\CH_j^1=\CH_{+j}^1\cup\CH_{-j}^1$ ($\CH_{\pm j}^1 = \Phi_j^1(\BR_\pm^N)$),
$\CH_j^2=\Phi_j^2(\uhs)$,
$\CH_j^3=\Phi_j^3(\lhs)$,
$\CH_j^4=\CH_j^5=\ws$,
$\Ga_j^1=\Phi_j^1(\BR_0^N)$, $\Ga_j^2=\Phi_j^2(\BR_0^N)$, and $\Ga_j^3=\Phi_j^3(\BR_0^N)$ in what follows.
Let us define $\nu_j^i(x)$ and $\rho_j^i(x)$ by
\begin{equation*}
	\nu_j^i(x) = 
		\left\{\begin{aligned}
			&\nu_{+j}^1(x)\chi_{\CH_{+ j}^1}(x)+\nu_{-j}^1(x)\chi_{\CH_{-j}^1}(x), && i=1, \\ 
			&\nu_{+j}^i(x), && i=2,4, \\
			&\nu_{-j}^i(x), && i=3,5, 
		\end{aligned}\right. \quad
	\rho_j^i(x) =
		\left\{\begin{aligned}
			&\rho_+\chi_{\CH_{+ j}^1}(x)+\rho_-\chi_{\CH_{-j}^1}(x), && i=1, \\
			&\rho_+, && i=2,4, \\
			&\rho_-, && i=3,5.
		\end{aligned}\right.
\end{equation*}
We see that, for $i=1,\dots,5$ and $j\in\BN$,
\begin{equation}\label{151022_3}
	\nu_j^i(x) = \mu(x) = \mu_+(x)\chi_{\Om_+}(x) +\mu_-(x)\chi_{\Om_-}(x), 
	\ \rho_j^i(x) = \rho(x) = \rho_+\chi_{\Om_+}(x) + \rho_-\chi_{\Om_-}(x),\ x\in\supp\zeta_j^i, 
\end{equation}
because $\wt\zeta_j^i=1$ on $\supp\zeta_j^i$.
Moreover, we set $\BT_j^i(\Bu,\te)=\nu_j^i(x)\BD(\Bu)-\te\BI$.

Let $(\Bf,\Bh,\Bk) \in X_{\CR,q}(\dot\Om)$.
We consider the following problems:
\begin{equation}\label{eq:loc1}
	\left\{\begin{aligned}
		\la\Bu_j^1 -(\rho_j^1)^{-1}\Di\BT_j^1(\Bu_j^1,K_j^1(\Bu_j^1)) &= \wt\zeta_j^1\Bf && \text{in $\CH_j^1$,} \\
		\jump{\BT_j^1(\Bu_j^1,K_j^1(\Bu_j^1))\Bn_j^1} &= \wt\zeta_j^1\Bh && \text{on $\Ga_j^1$,} \\
		\jump{\Bu_j^1} &= 0 && \text{on $\Ga_j^1$,}
	\end{aligned}\right. 
\end{equation}
and furthermore,
\begin{align}
	\la\Bu_j^2 -(\rho_j^2)^{-1}\Di\BT_j^2(\Bu_j^2,K_j^2(\Bu_j^2)) &= \wt\zeta_j^2\Bf \quad \text{in $\CH_j^2$,} \quad
	\BT_j^2(\Bu_j^2,K_j^2(\Bu_j^2))\Bn_{j}^2 = \wt\zeta_j^2\Bk \quad \text{on $\Ga_j^2$,} \label{eq:loc2} \\
	\la\Bu_j^3 -(\rho_j^3)^{-1}\Di\BT_j^3(\Bu_j^3,K_j^3(\Bu_j^3)) &= \wt\zeta_j^3\Bf \quad \text{in $\CH_j^3$,} \quad
	\Bu_j^3 =0 \quad \text{on $\Ga_j^3$,} \label{eq:loc3} \\
	\la\Bu_j^4 - (\rho_j^4)^{-1}\Di\BT_j^4(\Bu_j^4,K_j^4(\Bu_j^4)) &= \wt\zeta_j^4\Bf \quad \text{in $\CH_j^4$,} \label{eq:loc4} \\
	\la\Bu_j^5 - (\rho_j^5)^{-1}\Di\BT_j^5(\Bu_j^5,K_j^5(\Bu_j^5)) &= \wt\zeta_j^5\Bf \quad \text{in $\CH_j^5$.} \label{eq:loc5}
\end{align}
Here $K_j^i(\Bu_j^i)$ ($i=1,\dots,5$, $j\in\BN$) are given as follows:
For $\Bu_j^1\in W_q^2(\CH_j^1)^N$, 
$K_j^1(\Bu_j^1)\in W_q^1(\CH_j^1) + \wh{W}_q^1(\CH_j^0)$
denotes the unique solution to the weak problem:
\begin{align}\label{151106_2}
	((\rho_j^1)^{-1}\nabla K_j^1(\Bu_j^1),\nabla\ph)_{\CH_j^1} 
		&= ((\rho_j^1)^{-1}\Di(\nu_j^1\BD(\Bu_j^1))-\nabla\di\Bu_j^1,\nabla\ph)_{\CH_j^1} \quad
		\text{for all $\ph\in \wh{W}_{q'}^1(\CH_j^0)$,} \\
	\jump{K_j^1(\Bu_j^1)} 
		&= <\jump{\nu_j^1\BD(\Bu_j^1)\Bn_j^1},\Bn_j^1>-\jump{\di\Bu_j^1} \quad \text{on $\Ga_j^1$} \label{151106_2-2}
\end{align}
with $\|\nabla K_j^1(\Bu_j^1)\|_{L_q(\CH_j^1)}\leq C\|\nabla\Bu_j^1\|_{W_q^1(\CH_j^1)}$;
For $\Bu_j^2\in W_q^2(\CH_j^2)^N$, $K_j^2(\Bu_j^2)\in W_q^1(\CH_j^2)+\wh{W}_{q,0}^1(\CH_j^2)$ denotes the unique solution
to the weak problem:
\begin{align*}
	((\rho_j^2)^{-1}\nabla K_j^2(\Bu_j^2),\nabla\ph)_{\CH_j^2}
		&= ((\rho_j^2)^{-1}\Di(\nu_j^2\BD(\Bu_j^2))-\nabla\di\Bu_j^2,\nabla\ph)_{\CH_j^2}\quad
		\text{for all $\ph\in \wh W_{q',0}^1(\CH_j^2)$,} \\
	K_j^2(\Bu_j^2) &= <\nu_j^2\BD(\Bu_j^2)\Bn_j^2,\Bn_j^2>-\di\Bu_j^2 \quad \text{on $\Ga_j^2$} 
\end{align*}
with $\|\nabla K_j^2(\Bu_j^2)\|_{L_q(\CH_j^2)}\leq C\|\nabla\Bu_j^2\|_{L_q(\CH_j^2)}$;
For $\Bu_j^i\in W_q^2(\CH_j^i)^N$ ($i=3,4,5$),
$K_j^i(\Bu_j^i)\in \wh{W}_q^1(\CH_j^i)$ denotes the unique solution to the weak problem:
\begin{equation*}
	((\rho_j^i)^{-1}\nabla K_j^i(\Bu_j^i),\nabla\ph)_{\CH_j^i}
		= ((\rho_j^i)^{-1}\Di(\nu_j^i\BD(\Bu_j^i))-\nabla\di\Bu_j^i,\nabla\ph)_{\CH_j^i}\quad
		\text{for all $\ph\in \wh{W}_{q'}^1(\CH_j^i)$}
\end{equation*}
with $\|\nabla K_j^i(\Bu_j^i)\|_{L_q(\CH_j^i)}\leq C\|\nabla \Bu_j^i\|_{L_q(\CH_j^i)}$.
%
%
%
%
%
%
%

We know  that the following properties hold for the problems \eqref{eq:loc1}-\eqref{eq:loc5}\footnote[2]
{The existence of $\CR$-bounded solution operator families $\BS_j^1(\la)$, $\BS_j^2(\la)$, $\BS_j^3(\la)$ below follow from
Theorem \ref{theo:bent} and \cite[Theorem 4.1, Theorem 4.4]{Shibata14}, respectively.
In addition, concerning $\BS_j^4(\la)$, $\BS_j^5(\la)$, we can construct such $\CR$-bounded solution operator families
with variable viscosities in $\ws$ under the same condition as \eqref{visco:1} by using Theorem \ref{theo:ws} similarly to Section \ref{sec:bent}.}:
There exist a positive constant $\la_0\geq 1$ and
operator families $\BS_j^i(\la)\in \Hol(\Si_{\ep,\la_0},\CL(\CZ_q^i(\CH_j^i),W_q^2(\CH_j^i)^N))$ with
\begin{equation*}
	\CZ_q^i(\CH_j^i) = \CZ_{\CR,q}(\CH_j^i) \quad (i=1,2), \quad \CZ_q^i(\CH_j^i) = L_q(\CH_j^i)^N \quad (i=3,4,5) 
\end{equation*}
such that, for any $\la\in\Si_{\ep,\la_0}$, 
\begin{equation}\label{defi:S_j^i}
	\Bu_j^1 = \BS_j^1(\la)H_{\CR,\la}(\wt\zeta_j^1\Bf,\wt\zeta_j^1\Bh), \quad
	\Bu_j^2 = \BS_j^2(\la)H_{\CR,\la}(\wt\zeta_j^1\Bf,\wt\zeta_j^1\Bk), \quad
	\Bu_j^i = \BS_j^i(\la)(\wt\zeta_j^i\Bf) \quad (i=3,4,5)
\end{equation}
are unique solutions to \eqref{eq:loc1}-\eqref{eq:loc5}, respectively,
where $\CZ_{\CR,q}$ and $H_{R,\la}$ are given in Theorem \ref{theo:bent}.
In addition, 
\begin{equation}\label{150701_1}
	\CR_{\CL(\CZ_q^i(\CH_j^i),L_q(\CH_j^i)^{\wt N})}\Big(\Big\{\Big(\la\frac{d}{d\la}\Big)^l
	\Big(R_\la \BS_j^i(\la)\Big)\mid \la \in \Si_{\ep,\la_0}\Big\}\Big)
	\leq \ga_{4} \quad (l=0,1)
\end{equation}
with some positive constant $\ga_4$ depending on $\la_0$, but independent of $i=1,\dots,5$ and $j\in\BN$.
Since the $\CR$-boundedness implies the usual boundedness, 
we have, by \eqref{defi:S_j^i} and \eqref{150701_1} with $l=0$,  
\begin{align}\label{151107_1}
	\|R_\la\Bu_j^1\|_{L_q(\CH_j^1)} 
		&\leq \ga_4\big(\|(\Bf,\nabla\Bh,\la^{1/2}\Bh)\|_{L_q(\dot\Om\cap B_j^1)}+\|\Bh\|_{W_q^1(\dot\Om\cap B_j^1)}\big), \\
	\|R_\la\Bu_j^2\|_{L_q(\CH_j^2)}
		&\leq \ga_4 \big(\|\Bf\|_{L_q(\dot\Om\cap B_j^2)}+\|(\nabla\Bk,\la^{1/2}\Bk)\|_{L_q(\Om_+\cap B_j^2)}+\|\Bk\|_{W_q^1(\Om_+\cap B_j^2)}\big), \nonumber \\
	\|R_\la\Bu_j^i\|_{L_q(\CH_j^i)}
		&\leq \ga_4\|\Bf\|_{L_q(\dot\Om\cap B_j^i)} \quad (i=3,4,5) \nonumber
\end{align}
for any $\la\in \Si_{\ep,\la_0}$, noting $|\la|^{-1/2}\leq\la_0^{-1/2}$.

\subsection{Construction of parametrices}
For $(\Bf,\Bh,\Bk)\in X_{\CR,q}(\dot\Om)$, we consider the two-phase reduced Stokes equations \eqref{redu-eq:1}.
By Lemma \ref{prepa:2}, together with \eqref{finite}, \eqref{151107_1}, 
we see that the infinite sum $\sum_{i=1}^5\sum_{j=1}^\infty \zeta_j^i\Bu_j^i$ exists in the strong topology of $W_q^2(\dot\Om)^N$,
so that let us define $\Bu$ by 
\begin{equation}\label{conv:u}
	\Bu = \sum_{i=1}^5\sum_{j=1}^\infty \zeta_j^i\Bu_j^i \quad \text{in $W_q^2(\dot\Om)^N$}.
\end{equation}
Then, noting \eqref{151022_3}, $\Bn=\Bn_j^1$ on $\supp\zeta_j^1\cap \Ga$, and $\Bn_+=\Bn_j^2$ on $\supp\zeta_j^2\cap \Ga_+$, we have
\begin{equation*}
	\left\{\begin{aligned}
		\la\Bu- \rho^{-1}\Di\BT(\Bu,K(\Bu)) &= \Bf- \CU^0(\la)(\Bf,\Bh,\Bk) && \text{in $\dot\Om$,} \\
		\jump{\BT(\Bu,K(\Bu))\Bn} &= \jump{\Bh} - \jump{\CU^1(\la)(\Bf,\Bh,\Bk)} && \text{on $\Ga$,} \\
		\jump{\Bu} &= 0 && \text{on $\Ga$,} \\
		\BT(\Bu,K(\Bu))\Bn_+ &= \Bk - \CU^2(\la)(\Bf,\Bh,\Bk) && \text{on $\Ga_+$,} \\
		\Bu &= 0 && \text{on $\Ga_-$,}
	\end{aligned}\right.
\end{equation*}
where we have set
\begin{align}\label{def:CV}
	\CU^i(\la)(\Bf,\Bh,\Bk) = &\CV^i(\la)(\Bf,\Bh,\Bk)+\CP^i(\la)(\Bf,\Bh,\Bk) \quad (i=0,1,2), \quad \\
	\CV^0(\la)(\Bf,\Bh,\Bk) =&
			\sum_{i=1}^5\sum_{j=1}^\infty
			(\rho_j^i)^{-1}\Big\{\zeta_j^i\Di(\nu_j^i\BD(\Bu_j^i))-\Di(\nu_j^i\BD(\zeta_j^i\Bu_j^i))\Big\}, \nonumber \\
	\CV^i(\la)(\Bf,\Bh,\Bk) =&
			\sum_{j=1}^\infty\Big\{\nu_j^i\BD(\zeta_j^i\Bu_j^i)\Bn_j^i-\zeta_j^i\nu_j^i\BD(\Bu_j^i)\Bn_j^i\Big\} \quad (i=1,2), \nonumber \\
	\CP^0(\la)(\Bf,\Bh,\Bk) =&
			\sum_{i=1}^5\sum_{j=1}^\infty(\rho_j^i)^{-1}\Big\{\nabla K (\zeta_j^i\Bu_j^i)-\zeta_j^i\nabla K_j^i(\Bu_j^i)\Big\}, \nonumber \\
	\CP^i(\la)(\Bf,\Bh,\Bk) 
			=&\sum_{j=1}^\infty \Big\{\zeta_j^i K_j^i(\Bu_j^i)\Bn_j^i - K(\zeta_j^i\Bu_j^i)\Bn_j^i\Big\} \quad (i=1,2). \nonumber
\end{align}
Here we have used the fact that 
\begin{equation}\label{160207_1}
	\nabla K(\Bu) = \sum_{i=1}^5\sum_{j=1}^\infty \nabla K(\zeta_j^i\Bu_j^i) \quad \text{in $L_q(\dot\Om)^N$}, \quad
	K(\Bu) = \sum_{j=1}^\infty K(\zeta_j^i\Bu_j^i) \quad \text{in $W_q^{1-1/q}(\Ga^i)$} \quad (i=1,2).
\end{equation}
In fact, we have the following observation: In view of Subsection \ref{sub:reduced}, we see by \eqref{151022_3} that
\begin{equation*}
	K\left(\zeta_j^i\Bu_j^i\right)=\CK\left((\rho_j^i)^{-1}\Di(\nu_j^i\BD(\zeta_j^i\Bu_j^i))-\nabla\di(\zeta_j^i\Bu_j^i),\jump{g_j^i},h_j^i|_{\Ga_+}\right),
\end{equation*}
where $\cdot\,|_{\Ga_+}$ denotes the trace to $\Ga_+$ and 
\begin{align*}
	&\left(g_j^1, h_j^1\right) = \left(<\nu_j^1\BD(\zeta_j^1\Bu_j^1)\Bn_j^1,\Bn_j^1> -\di(\zeta_j^1\Bu_j^1), 0 \right), \displaybreak[0] \\ 
	&\left(g_j^2, h_j^2\right) = \left(0,<\nu_j^2\BD(\zeta_j^2\Bu_j^2)\Bn_j^2,\Bn_j^2> -\di(\zeta_j^2\Bu_j^2) \right), \quad
	g_j^i=h_j^i =0 \quad  (i=3,4,5).
\end{align*}
On the other hand, by \eqref{151022_3} and \eqref{conv:u},
\begin{align*}
	&\rho^{-1}\Di(\mu\BD(\Bu))-\nabla\di\Bu
		= \sum_{i=1}^5\sum_{j=1}^\infty\left((\rho_j^i)^{-1}\Di(\nu_j^i\BD(\zeta_j^i\Bu_j^i))-\nabla\di(\zeta_j^i\Bu_j^i)\right) \quad \text{in $L_q(\dot\Om)^N$}, \\
	&<\jump{\mu\BD(\Bu)\Bn},\Bn> -\jump{\di\Bu} 
		= \sum_{j=1}^\infty\left(<\jump{\nu_j^1\BD(\zeta_j^1\Bu_j^1)\Bn_j^1},\Bn_j^1>-\jump{\di(\zeta_j^1\Bu_j^1)}\right) \quad \text{in $W_q^{1-1/q}(\Ga)$}, \\
	&<\mu\BD(\Bu)\Bn_+,\Bn_+> -\di\Bu 
		= \sum_{j=1}^\infty\left(<\nu_j^2\BD(\zeta_j^2\Bu_j^2)\Bn_j^2,\Bn_j^2>-\di(\zeta_j^2\Bu_j^2)\right) \quad \text{in $W_q^{1-1/q}(\Ga_+)$}. 
\end{align*}
Thus the continuity of $\CK$ implies \eqref{160207_1} and $K(\Bu)=\CK(\al,\beta,\ga)$ with $(\al,\beta,\ga)$ given by \eqref{abc}.

Now it holds that, by \eqref{defi:S_j^i},
\begin{align*}
	&\Bu 
	=\sum_{j=1}^\infty \zeta_j^1\BS_j^1(\la)
		\left(\wt\zeta_j^1\Bf,\wt\zeta_j^1(\nabla\Bh)+\la^{-1/2}(\nabla\wt\zeta_j^1)(\la^{1/2}\Bh),\wt\zeta_j^1(\la^{1/2}\Bh),\wt\zeta_j^1\Bh\right) \\
	&+\sum_{j=1}^\infty \zeta_j^2\BS_j^2(\la)\left(\wt\zeta_j^2\Bf,\wt\zeta_j^2(\nabla\Bk)+\la^{-1/2}(\nabla\wt\zeta_j^2)(\la^{1/2}\Bk),				
		\wt\zeta_j^2(\la^{1/2}\Bk),\wt\zeta_j^2\Bk\right)
	+\sum_{i=3}^5\sum_{j=1}^\infty\zeta_j^i\BS_j^i(\la)\left(\wt\zeta_j^i\Bf\right),
\end{align*}
so that we set, by $\BH=(H_1,\dots,H_7)\footnote[2]{As was mentioned in Remark \ref{rema:div} \eqref{rema:div3},
the symbols $H_1$, $H_2$, $H_3$, $H_4$, $H_5$, $H_6$, and $H_7$
are variables corresponding to $\Bf$, $\nabla\Bh$, $\la^{1/2}\Bh$, $\Bh$, $\nabla\Bk$, $\la^{1/2}\Bk$, and $\Bk$, respectively.}\in\CX_{\CR,q}(\dot\Om)$,
\begin{align}\label{op:S}
	\CS_j^1(\la) \BH
		&=\BS_j^1(\la)\big(\wt\zeta_j^1H_1,\wt\zeta_j^1H_2+\la^{-1/2}(\nabla\wt\zeta_j^1)H_3,\wt\zeta_j^1H_3,\wt\zeta_j^1H_4\big), \displaybreak[0] \\
	\CS_j^2(\la) \BH
		&=\BS_j^2(\la)\big(\wt\zeta_j^2H_1,\wt\zeta_j^2 H_5+\la^{-1/2}(\nabla\wt\zeta_j^2)H_6,\wt\zeta_j^2H_6,\wt\zeta_j^2H_7\big), \nonumber \displaybreak[0] \\
	\CS_j^i(\la) \BH
		&= \BS_j^i(\la)\big(\wt\zeta_j^i H_1\big) \quad (i=3,4,5). \nonumber
\end{align}
It then clear that $\Bu=\sum_{i=1}^5\sum_{j=1}^\infty\zeta_j^i\CS_j^i(\la)F_{\CR,\la}(\Bf,\Bh,\Bk)$.
By \eqref{150701_1} with Definition \ref{defi:R},
it holds that
\begin{equation}\label{160201_10}
	\int_0^1 \Big\|\sum_{l=1}^n r_l(u)R_{\la_l}(\zeta_j^i\CS_j^i(\la_l)\BH_l)\Big\|_{L_q(\dot\Om)}^q\intd u
		\leq\ga_4\int_0^1\Big\|\sum_{l=1}^n r_l(u)\BH_l\Big\|_{\CX_{\CR,q}(\dot\Om\cap B_j^i)}^q\intd u
\end{equation}
for any $n\in\BN$, $\{\la_l\}_{l=1}^n\subset\Si_{\ep,\la_0}$, and $\{\BH_l\}_{l=1}^n\subset \CX_{\CR,q}(\dot\Om)$.
The inequality \eqref{160201_10} with $n=1$, together with Lemma \ref{prepa:2}, yields that
the infinite sum $\sum_{j=1}^\infty\zeta_j^i\CS_j^i(\la)\BH$ exists in the strong topology of $W_q^2(\dot\Om)^N$,
so that let us define $\CT^i(\la)\BH = \sum_{j=1}^\infty\zeta_j^i\CS_j^i(\la)\BH$ for $i=1,\dots,5$.
In addition, by Lemma \ref{prepa:1},  
\begin{equation*}
	\Big\|\sum_{l=1}^n a_l R_{\la_l}\CT^i(\la_l)\BH_l\Big\|_{L_q(\dot\Om)}^q \leq
		\ga_4\sum_{j=1}^\infty\Big\|\sum_{l=1}^n a_l R_{\la_l}\Big(\zeta_j^i\CS_j^i(\la_l)\BH_l\Big)\Big\|_{L_q(\dot\Om)}^q \quad (i=1,\dots,5)
\end{equation*}
for any $n\in\BN$, $\{a_l\}_{l=1}^n\subset\BC$, $\{\la_l\}_{l=1}^n\subset\Si_{\ep,\la_0}$, and $\{\BH_l\}_{l=1}^n\subset\CX_{\CR,q}(\dot\Om)$.
The last inequality combined with \eqref{finite}, \eqref{160201_10} furnishes that, by the monotone convergence theorem,
\begin{align*}
	&\int_0^1\Big\|\sum_{l=1}^n r_l(u)R_{\la_l}\CT^i(\la_l)\BH_l\Big\|_{L_q(\dot\Om)}^q\intd u \\
	&\leq \ga_4\sum_{j=1}^\infty \int_0^1\Big\|\sum_{l=1}^n r_l(u)\BH_l \Big\|_{\CX_{\CR,q}(\dot\Om\cap B_j^i)}^q\intd u 
	\leq \ga_4  \int_0^1 \Big\|\sum_{l=1}^n r_l(u)\BH_l \Big\|_{\CX_{\CR,q}(\dot\Om)}^q \intd u,
\end{align*}
which implies that, for $i=1,\dots,5$,
\begin{equation*}
	\CT^i(\la) \in \Hol(\Si_{\ep,\la_0},\CL(\CX_{\CR,q}(\dot\Om),W_q^2(\dot\Om)^N))\footnote[3]{Holomorphic property can be proved in the same manner
	as in \cite[Proposition 5.3 (ii)]{Shibata13}.},
	\quad
	\CR_{\CL(\CX_{\CR,q}(\dot\Om),L_q(\dot\Om)^{\wt N})}
	\left(\left\{R_\la\CT^i(\la)\mid\la\in\Si_{\ep,\la_0}\right\}\right) \leq \ga_4.
\end{equation*}
Analogously, we have the $\CR$-boundedness of $\{(\la \frac{d}{d\la}) (R_\la \CT^i(\la)) \mid \la\in\Si_{\ep,\la_0}\}$
on $\CL(\CX_{\CR,q}(\dot\Om),L_q(\dot\Om)^{\wt{N}})$.
Thus, setting $\BS(\la)\BH = \sum_{i=1}^5\CT^i(\la)\BH$ yields that, by Proposition \ref{prop:R}, 
\begin{align}\label{160207_12}
	&\BS(\la)\in \Hol(\Si_{\ep,\la_0},\CL(\CX_{\CR,q}(\dot\Om),W_q^2(\dot\Om)^N)), \quad \Bu = \BS(\la)F_{\CR,\la}(\Bf,\Bh,\Bk), \\
	&\CR_{\CL(\CX_{\CR,q}(\dot\Om),L_q(\dot\Om)^{\wt N})}
	\Big(\Big\{\Big(\la\frac{d}{d\la}\Big)^l\Big(R_\la\BS(\la)\Big)\Big\}\Big) \leq \ga_4 \quad (l=0,1). \nonumber
\end{align}

\subsection{Estimates of the remainder terms $\CU^i(\la)(\Bf,\Bh,\Bk)$}
In this subsection, we prove the following lemma.

\begin{lemm}\label{lemm:CV}
	Let $\la_0$ and $\ga_4$ be the same numbers as in \eqref{150701_1}.
	Let $\CU^i(\la)$, $\CV^i(\la)$, and $\CP^i(\la)$ $(i=0,1,2)$ be the operators defined in \eqref{def:CV} and set
	\begin{align*}	
		\CU(\la)(\Bf,\Bh,\Bk) &= \CV(\la)(\Bf,\Bh,\Bk) + \CP(\la)(\Bf,\Bh,\Bk), \\
		\CV(\la)(\Bf,\Bh,\Bk) &= (\CV^0(\la)(\Bf,\Bh,\Bk),\CV^1(\la)(\Bf,\Bh,\Bk), \CV^2(\la)(\Bf,\Bh,\Bk)), \\
		\CP(\la)(\Bf,\Bh,\Bk) &= (\CP^0(\la)(\Bf,\Bh,\Bk),\CP^1(\la)(\Bf,\Bh,\Bk), \CP^2(\la)(\Bf,\Bh,\Bk)).
	\end{align*}
	Then there exists an operator family $\BU(\la)\in \Hol(\Si_{\ep,\la_1},\CL(\CX_{\CR,q}(\dot\Om)))$
	such that
	\begin{align*}
		&\CU(\la)(\Bf,\Bh,\Bk) = \BU(\la)F_{\CR,\la}(\Bf,\Bh,\Bk) \quad \text{for $(\Bf,\Bh,\Bk)\in X_{\CR,q}(\dot\Om)$}, \\ 
		&\CR_{\CL(\CX_{\CR,q}(\dot\Om))}\Big(\Big\{\Big(\la\frac{d}{d\la}\Big)^l F_{\CR,\la}\BU(\la) \mid \la \in\Si_{\ep,\la_1}\Big\}\Big)
			\leq \ga_4(\si_2+ \ga_{\si_2}\si_1 +  \ga_{\si_1}\ga_{\si_2}\la_1^{-1/2}) \quad (l=0,1) 
	\end{align*}
	for any $\si_1,\si_2>0$ and for any $\la_1\geq\la_0$.
	Here and subsequently, $\ga_{\si_1}$, $\ga_{\si_2}$ are positive constants depending on $\si_1$, $\si_2$, respectively.
\end{lemm}

\begin{proof}
{\bf Step 1: Case $\CV^i(\la)$.}
First we consider $\CV^0(\la)(\Bf,\Bh,\Bk)$. 
We write $\Di(\mu\BD(\ph\Bu))-\ph\Di(\mu\BD(\Bu))=\CC_1(\mu,\ph)\nabla\Bu+\CC_0(\mu,\ph)\Bu$ for any
scalar functions $\mu$, $\ph$ and for any $N$-vector function $\Bu$, where we have set
\begin{align*}
	&\CC_0(\mu,\ph)\Bu
		= <\nabla\mu,\Bu>\nabla\ph+<\nabla\mu,\nabla\ph>\Bu
		+ \mu\left\{(\nabla^2\ph)\Bu+(\De\ph)\Bu\right\}, \\
	&\CC_1(\mu,\ph)\nabla\Bu
		= \mu\left\{\BD(\Bu)\nabla\ph+(\nabla\ph)\di\Bu+(\nabla\Bu)\nabla\ph\right\}.
\end{align*}
Using the above symbols $\CC_0$, $\CC_1$ and \eqref{op:S}, we write
\begin{equation*}
	\Di(\nu_j^i\BD(\zeta_j^i\Bu_j^i))-\zeta_j^i\Di(\nu_j^i\BD(\Bu_j^i))
		=\CC_1(\nu_j^i,\zeta_j^i)\nabla\CS_j^i(\la)F_{\CR,\la}(\Bf,\Bh,\Bk)
		+\CC_0(\nu_j^i,\zeta_j^i)\CS_j^i(\la)F_{\CR,\la}(\Bf,\Bh,\Bk)
\end{equation*}
for $i=1,\dots,5$ and $j\in\BN$.
By \eqref{151022_2} and Proposition \ref{lemm:Shi13} with $\si=1$, we have, for $\BH\in\CX_{\CR,q}(\dot\Om)$,
\begin{equation*}
	\|\CC_0(\nu_j^i,\zeta_j^i)\CS_j^i(\la)\BH\|_{L_q(\dot\Om)} \leq C\|\CS_j^i(\la)\BH\|_{W_q^1(\CH_j^i)}, \quad
	\|\CC_1(\nu_j^i,\zeta_j^i)\nabla\CS_j^i(\la)\BH\|_{L_q(\dot\Om)} \leq C\|\nabla\CS_j^i(\la)\BH\|_{L_q(\CH_j^i)},
\end{equation*}
which, combined with \eqref{150701_1} and Proposition \ref{lemm:multi}, funishes that
\begin{align}\label{160207_13}
	&\int_0^1\Big\|\sum_{l=1}^n r_l(u)\CC_1(\nu_j^i,\zeta_j^i) \nabla\CS_j^i(\la_l)\BH_l\Big\|_{L_q(\dot\Om)}^q\intd u
		\leq (\ga_4\la_1^{-1/2})^q \int_0^1\Big\|\sum_{l=1}^nr_l(u) \BH_l\Big\|_{\CX_{\CR,q}(\dot\Om\cap B_j^i)}^q\intd u, \\
	&\int_0^1\Big\|\sum_{l=1}^n r_l(u)\CC_0(\nu_j^i,\zeta_j^i) \CS_j^i(\la_l)\BH_l\Big\|_{L_q(\dot\Om)}^q\intd u
		\leq (\ga_4\la_1^{-1})^q \int_0^1\Big\|\sum_{l=1}^nr_l(u) \BH_l\Big\|_{\CX_{\CR,q}(\dot\Om\cap B_j^i)}^q\intd u \nonumber
\end{align}
for any $\la_1\geq\la_0$ and for any $n\in\BN$, $\{\la_l\}_{l=1}^n\subset\Si_{\ep,\la_1}$, and $\{\BH_l\}_{l=1}^n\subset\CX_{\CR,q}(\dot\Om)$.
Define $\BV^0(\la)\BH$ as
\begin{equation*}
	\BV^0(\la)\BH
		=\sum_{i=1}^5\sum_{j=1}^\infty
		\Big\{	\CC_1(\nu_j^i,\zeta_j^i)\nabla\CS_j^i(\la)\BH
		+\CC_0(\nu_j^i,\zeta_j^i)\CS_j^i(\la)\BH \Big\} \quad \text{with $\BH\in\CX_{\CR,q}(\dot\Om)$}.
\end{equation*}
%
%
In the same manner as we have obtained \eqref{160207_12} from \eqref{160201_10},
we can prove, by \eqref{160207_13}, the following properties: 
\begin{align}\label{fin:1}
	&\BV^0(\la)\in\Hol(\Si_{\ep,\la_1},\CL(\CX_{\CR,q}(\dot\Om),L_q(\dot\Om)^N)), \quad
	\CV^0(\la)(\Bf,\Bh,\Bk) = \BV^0(\la)F_{\CR,\la}(\Bf,\Bh,\Bk) \\
	&\CR_{\CL(\CX_{\CR,q}(\dot\Om),L_q(\dot\Om)^N)}\Big(\Big\{\Big(\la\frac{d}{d\la}\Big)^l\BV^0(\la)\mid \la\in\Si_{\ep,\la_1}\Big\}\Big)
		\leq \ga_4\la_1^{-1/2} \quad (l=0,1) \nonumber
\end{align}
for any $\la_1\geq\la_0$.
Here and hereafter, $\la_1$ denotes any number satisfying $\la_1\geq\la_0$.
Analogously, we can construct operator families $\BV^i(\la)$ $(i=1,2)$ such that
\begin{align}\label{fin:2}
	&\BV^i(\la)\in\Hol(\Si_{\ep,\la_1},\CL(\CX_{\CR,q}(\dot\Om),W_q^1(\dot\Om)^N)), \quad
	\CV^i(\la)(\Bf,\Bh,\Bk) = \BV^i(\la)F_{\CR,\la}(\Bf,\Bh,\Bk), \\
	&\CR_{\CL(\CX_{\CR,q}(\dot\Om),\wt\CX_{\CR,q}(\dot\Om))}
	\Big(\Big\{\Big(\la\frac{d}{d\la}\Big)^l\big(\wt F_{\CR,\la}\BV^i(\la)\big)\mid\la\in\Si_{\ep,\la_1}\Big\}\Big) \leq \ga_4\la_1^{-1/2}
	\quad (l=0,1,i=1,2) \nonumber
\end{align}
for any $\la_1\geq\la_0$.
Here and hereafter, we set
\begin{equation*}
	\wt \CX_{\CR,q}(\dot\Om) = L_q(\dot\Om)^{N^2}\times L_q(\dot\Om)^N\times W_q^1(\dot\Om)^N, \quad
	\wt F_{\CR,\la}\Bu = (\nabla\Bu,\la^{1/2}\Bu,\Bu).
\end{equation*}

\noindent{\bf Step 2: Case $\CP^i(\la)$.}
We consider the term:
\begin{equation*}
	(\rho_j^i)^{-1}\{\nabla K(\zeta_j^i\Bu_j^i) -\zeta_j^i \nabla K(\Bu_j^i)\} 
		= (\rho_j^i)^{-1}\nabla\left(K(\zeta_j^i\Bu_j^i)-\zeta_j^i K_j^i(\Bu_j^i)\right) + (\rho_j^i)^{-1}(\nabla\zeta_j^i) K_j^i(\Bu_j^i).
\end{equation*}
We start with the following inequalities of Poincar\'e type with uniform constant,
which are proved in the same manner as in the proof of \cite[Lemma 3.4, Lemma 3.5]{Shibata13}.

\begin{lemm}\label{lemm:Po}
Let $1<q<\infty$. 
Then there exists a constant $c_1>0$, independent of $j\in\BN$, such that
\begin{align*}
	&\|\ph-c_j^1(\ph)\|_{W_q^1(\CH_j^0 \cap B_j^1)} \leq c_1 \|\nabla\ph\|_{L_{q(\CH_j^0 \cap B_j^1)}} 
		&& \text{for any $\ph\in\wh{W}_{q}^1(\CH_j^0)$,} \\
	&\|\psi-c_j^1(\psi)\|_{W_q^1(\Om\cap B_j^1)} \leq c_1\|\nabla\psi\|_{L_q(\Om\cap B_j^1)} && \text{for any $\psi\in \CW_q^1(\Om)$,} \\
	&\|\psi\|_{W_q^1(\Om\cap B_j^2)}\leq c_1\|\nabla\psi\|_{L_q(\Om\cap B_j^2)} && \text{for any $\psi\in \CW_q^1(\Om)$,} \\
	&\|\psi-c_j^i(\psi)\|_{W_q^1(\Om\cap B_j^i)} \leq c_1\|\nabla\psi\|_{L_q(\Om\cap B_j^i)} && \text{for any $\psi\in \CW_q^1(\Om)$, $i=3,4,5$.}
\end{align*}
Here $c_j^1(\ph)$ and $c_j^i(\psi)$ $(i=1,3,4,5)$ are suitable constants depending on $\ph$ and $\psi$, respectively.
\end{lemm}

To handle $(\rho_j^i)^{-1}(\nabla\zeta_j^i)K_j^i(\Bu_j^i)$, we use the following lemma.

\begin{lemm}\label{lemm:K}
Let $1<q<\infty$. Then there exists a constant $c_2$, independent of $j\in\BN$, such that
\begin{align}\label{160114_10}
	&\|K_j^1(\Bu)\|_{L_q(\CH_j^1\cap B_j^1)}
		\leq c_2\big(\|\nabla\Bu\|_{L_q(\CH_j^1)}+\|\nabla\Bu\|_{L_q(\CH_{+j}^1)}^{1-1/q}\|\nabla^2\Bu\|_{L_q(\CH_{+j}^1)}^{1/q}
		+\|\nabla\Bu\|_{L_q(\CH_{-j}^1)}^{1-1/q}\|\nabla^2\Bu\|_{L_q(\CH_{-j}^1)}^{1/q}\big), \\
	&\|K_j^i(\Bv)\|_{L_q(\CH_j^i\cap B_j^i)}
		\leq c_2\big(\|\nabla\Bv\|_{L_q(\CH_j^i)}+\de^i\|\nabla\Bv\|_{L_q(\CH_j^i)}^{1-1/q}\|\nabla^2\Bv\|_{L_q(\CH_j^i)}^{1/q}\big) \nonumber
\end{align}
for any $\Bu\in W_q^2(\CH_j^1)^N$ and for any $\Bv\in W_q^2(\CH_j^i)^N$ $(i=2,\dots,5)$, respectively,
where $\de^i$ are symbols defined by $\de^i=1$ $(i=2,3)$ and $\de^i=0$ $(i=4,5)$.
\end{lemm}

\begin{rema}
Applying Young's inequality to \eqref{160114_10}, we have
\begin{equation}\label{160114_11}
	\|K_j^i(\Bu)\|_{L_q(\CH_j^i\cap B_j^i)} \leq \si_1\|\nabla^2\Bu\|_{L_q(\CH_j^i)} + \ga_{\si_1}\|\nabla\Bu\|_{L_q(\CH_j^i)} \quad (i=1,\dots,5)
\end{equation}
for any $\si_1>0$ and for any $\Bu\in W_q^2(\CH_j^i)^N$. 
\end{rema}

\begin{proof}[Proof of Lemma \ref{lemm:K}]
We here show the case $K_j^1(\Bu)$\footnote[2]{The other cases were already proved in \cite[Lemma 5.6]{Shibata13}.}.
In the following, $C$ stands for generic constants independent of $j\in\BN$,
and recall that $\CH_j^0 = \Phi_j^1(\ws) = \ws$, $\CH_j^1 = \CH_{+ j}^1\cup\CH_{-j}^1$ ($\CH_{\pm j}^1= \Phi_j^1(\BR_\pm^N)$, respectively),
and $\Ga_j^1 = \Phi_j^1(\BR_0^N)$. 

Let $\eta_{ j}\in C_0^\infty(\CH_{ j}^0\cap B_j^1)$ in such a way that $\int_{B_j^1}\eta_{ j}\intd x =1$ and $\eta_{ j}\geq0$.
Fix $\Bu\in W_q^2(\CH_j^1)$ in what follows.
Since $K_j^1(\Bu) + c$ satisfies the weak problem \eqref{151106_2}-\eqref{151106_2-2} for any constant $c$,
we may assume that $\int_{\CH_j^1}\eta_j K_j^1(\Bu)\intd x=0$.
Given $\psi \in C_0^\infty(\CH_{j}^0\cap B_j^1)$, we define a function
by $\wt\psi=\psi -\eta_{ j}\int_{B_j^1}\psi\intd x$. Then, 
\begin{equation*}
	\wt\psi\in C_0^\infty(\CH_{ j}^0\cap B_j^1), \quad 
	\|\wt\psi\|_{L_{q'}(\CH_{ j}^0)} \leq C\|\psi\|_{L_{q'}(\CH_{j}^0)}, \quad
	\int_{\CH_{ j}^0 \cap B_j^1}\wt\psi \intd x = 0
\end{equation*}
for $q'=q/(q-1)$. 
These properties combined with Lemma \ref{lemm:Po} yields that
\begin{equation*}
	|(\wt\psi,\ph)_{\CH_{ j}^0}| = |(\wt\psi, \ph - c_{j}^1(\ph)) _{\CH_{ j}^0\cap B_j^1}| \\
		\leq\|\wt\psi\|_{L_{q'}(\CH_{j}^0\cap B_j^1)}\|\ph-c_{ j}^1(\ph)\|_{L_q(\CH_{ j}^0\cap B_j^1)}
		\leq C\|\psi\|_{L_{q'}(\CH_{ j}^0)}\|\nabla \ph\|_{L_q(\CH_{ j}^0)}
\end{equation*}
for any $\ph\in \wh{W}_{q}^1(\CH_{ j}^0)$.
Thus $\|\wt\psi\|_{\wh W_{q'}^{-1}(\CH_j^0)} \leq C\|\psi\|_{L_{q'}(\CH_j^0)}$,
where $\wh W_{q'}^{-1}(\CH_j^0)$ is the dual spaces of $\wh W_{q}^1(\CH_j^0)$. 


Let $\wh{W}_{q'}^2(\CH_{\pm j}^1)$ be function spaces defined as
$\wh{W}_{q'}^2(\CH_{\pm j}^1) = \{\te\in \wh{W}_{q'}^1(\CH_{\pm j}^1) \mid \nabla\te \in W_{q'}^1(\CH_{\pm j}^1)^N\}$, respectively.
We choose a $\Psi \in \wh{W}_{q'}^2(\CH_{+ j}^1)\cap \wh{W}_{q'}^2(\CH_{-j}^1)$ satisfying the following equations:
\begin{equation}\label{strong:1}
	-\De\Psi = \wt\psi \quad \text{in $\CH_{ j}^1$,} \quad \jump{\frac{\pa \Psi}{\pa\Bn_j^1}}=0 \quad \text{on $\Ga_j^1$},
	\quad  \jump{\rho_j^1\Psi} =0 \quad \text{on $\Ga_j^1$}
\end{equation}
and the estimate: $\|\nabla\Psi\|_{W_q^1(\CH_j^1)}
\leq C(\|\wt\psi\|_{L_{q'}(\CH_j^0)}+ \|\wt\psi\|_{\wh W_{q'}^{-1}(\CH_j^0)})\footnote[3]{
Since $\wt\psi\in\wh{W}_{q'}^{-1}(\CH_j^0)$, we can construct, by the Hahn-Banach theorem, $\om\in L_{q'}(\CH_j^0)^N$ such that
$(\wt\psi,\ph)_{\CH_j^0}=(\om,\nabla\ph)_{\CH_j^0}$ for any $\ph\in\wh{W}_{q}^{1}(\CH_j^0)$
and $\|\wt\psi\|_{\wh{W}_{q'}^{-1}(\CH_j^0)}=\|\om\|_{L_{q'}(\CH_j^0)}$.
Let $u\in\wh{W}_{q'}^1(\CH_j^0)$ be the solution of the weak elliptic transmission problem:
$((\rho_j^1)^{-1}\nabla u,\nabla\ph)_{\CH_j^1}=(\om,\nabla\ph)_{\CH_j^0}$ for any $\ph\in\wh{W}_{q}^{1}(\CH_j^0)$,
which possesses the estimate: $\|\nabla u\|_{L_{q'}(\CH_j^0)} \leq C \|\om\|_{L_{q'}(\CH_j^0)}$.
Then choosing suitable $\ph$ shows that $u$ satisfies the following strong problem:
$- (\rho_j^1)^{-1}\De u = \wt\psi$ in $\CH_j^1$, $\jump{(\rho_j^1)^{-1}\pa u/\pa\Bn_j^1}=0$ on $\Ga_j^1$, $\jump{u}=0$ on $\Ga_j^1$,
and also $u$ is a unique solution to the strong problem by the unique solvability of the weak elliptic transmission problem. 
Thus, by the standard Fourier analysis similarly to Section \ref{sec:ws} and Section \ref{sec:bent},
we have $\|\nabla^2 u\|_{L_{q'}(\CH_j^1)}\leq C\|\wt\psi\|_{L_{q'}(\CH_j^0)}$.
If we set $\Psi = (\rho_j^1)^{-1} u$, then $\Psi$ satisfies \eqref{strong:1} and the required estimates.
}$.
Then the estimates of $\wt\psi$ above yields that $\|\nabla\Psi\|_{W_{q'}^1(\CH_{ j}^1)}\leq C\|\psi\|_{L_{q'}(\CH_{j}^0)}$,
and furthermore, by Gauss's divergence theorem,
\begin{equation*}
	(\nabla\Psi,\nabla\te)_{\CH_{ j}^1} - \Big(\frac{\pa \Psi}{\pa \Bn_j^1},\jump{\te}\Big)_{\Ga_j^1} = (\wt\psi,\te)_{\CH_{j}^1}
	\quad \text{for any $\te \in W_{q}^1(\CH_j^1)+\wh{W}_{q}^1(\CH_j^0)$.}
\end{equation*}
This identity allows us to see that
\begin{align*}
	&(K_j^1(\Bu),\psi)_{\CH_j^0} = (K_j^1(\Bu),\wt\psi)_{\CH_j^0}  = (K_j^1(\Bu),\wt\psi)_{\CH_j^1}
	=(\nabla K_j^1(\Bu),\nabla\Psi)_{\CH_j^1} -\Big(\jump{K_j^1(\Bu)},\frac{\pa \Psi}{\pa\Bn_j^1}\Big)_{\Ga_j^1} \\
	&=\Big((\rho_j^1)^{-1}\nabla K_j^1(\Bu),\nabla(\rho_j^1\Psi)\Big)_{\CH_j^1}-\Big(\jump{K_j^i(\Bu)},\frac{\pa \Psi}{\pa\Bn_j^1}\Big)_{\Ga_j^1},
\end{align*}
which, combined with $\rho_j^1\Psi \in \wh{W}_{q'}^1(\CH_j^0)$ as follows from $\jump{\rho_j^1\Psi}=0$ on $\Ga_j^1$, implies that,
by \eqref{151106_2}-\eqref{151106_2-2},
\begin{equation*}
	(K_j^1(\Bu),\psi)_{\CH_j^0}
		= \Big((\rho_j^1)^{-1}\Di(\nu_j^1\BD(\Bu))-\nabla\di\Bu,\nabla(\rho_j^1\Psi)\Big)_{\CH_j^1} 
 		-\Big(<\jump{\nu_j^1\BD(\Bu)\Bn_j^1},\Bn_j^1>-\jump{\di\Bu},\frac{\pa \Psi}{\pa\Bn_j^1}\Big)_{\Ga_j^1}.
\end{equation*}
Thus, by Gauss's divergence theorem, we have
\begin{align}\label{160210_11}
	&(K_j^1(\Bu),\psi)_{\CH_j^0} = -\Big(\nu_j^1\BD(\Bu),\nabla^2\Psi\Big)_{\CH_j^1} 
		+\int_{\Ga_j^1}\jump{<\nu_j^1\BD(\Bu)\Bn_j^1,\nabla\Psi>}\intd\si \\
		&+\Big(\di\Bu,\rho_j^1\De\Psi\Big)_{\CH_j^1}
		-\int_{\Ga_j^1}\jump{<(\di\Bu)\Bn_j^1,\rho_j^1\nabla\Psi>} \intd\si
		-\Big(<\jump{\nu_j^1\BD(\Bu)\Bn_j^1},\Bn_j^1>-\jump{\di\Bu},\frac{\pa \Psi}{\pa\Bn_j^1}\Big)_{\Ga_j^1}, \nonumber
\end{align}
where $d\si$ denotes the surface element of $\Ga_j^1$.

At this point, we introduce trace inequalities as follows:
there exists a positive constant $c_3$, independent of $j\in\BN$, such that
\begin{align}\label{trace}
	&\|f_\pm\|_{L_q(\Ga_j^1)} \leq c_3\|f_\pm\|_{L_q(\CH_{\pm j}^1)}^{1-1/q}\|\nabla f_\pm\|_{L_q(\CH_{\pm j}^1)}^{1/q} 
		\quad \text{for any $f_\pm\in W_q^1(\CH_{\pm j}^1)$, respectively,} \\
	&\|f\|_{L_q(\Ga_j^i)} \leq c_3\|f\|_{L_q(\CH_j^i)}^{1-1/q}\|\nabla f\|_{L_q(\CH_j^i)}^{1/q}
		\quad \text{for any $f\in W_q^1(\CH_j^i)$, $i=2,3$,} \nonumber
\end{align}
which can be proved by 
Proposition \ref{prop:uniform} and \cite[Section 4, Proposition 16.2]{DiBenedetto02}.
These inequalities combined with \eqref{160210_11} furnish that
\begin{equation*}
	|(K_j^1(\Bu_j^1),\psi)_{\CH_j^0}| \leq  C\left(\|\nabla\Bu\|_{L_q(\CH_j^1)} 
	+ \|\nabla\Bu\|_{L_q(\CH_{+ j}^1)}^{1-1/q}\|\nabla^2\Bu\|_{L_q(\CH_{+ j}^1)}^{1/q}
	+ \|\nabla\Bu\|_{L_q(\CH_{- j}^1)}^{1-1/q}\|\nabla^2\Bu\|_{L_q(\CH_{- j}^1)}^{1/q}\right)\|\psi\|_{L_{q'}(\CH_j^0)},
\end{equation*}
which implies that the required estimate \eqref{160114_10} holds.
This completes the proof of the lemma.
\end{proof}

We consider $(\rho_j^i)^{-1}(\nabla\zeta_j^i)K_j^i(\Bu_j^i)$.
By Definition \ref{defi:R}, \eqref{160114_11}, and \eqref{150701_1},  together with the formulas \eqref{op:S},
we have, for any $n\in\BN$, $\{\la_l\}_{l=1}^n\subset\Si_{\ep,\la_1}$, and $\{\BH_l\}_{l=1}^n\subset\CX_{\CR,q}(\dot\Om)$,
\begin{equation}\label{160210_13}
	\int_0^1 \Big\|\sum_{l=1}^n r_l(u)(\nabla\zeta_j^i) K_j^i(\CS_j^i(\la_l)\BH_l)\Big\|_{L_q(\dot\Om)}^q\intd u
		\leq \{\ga_4(\si_1+\ga_{\si_1}\la_1^{-1/2})\}^q
		\int_0^1 \Big\|\sum_{l=1}^n r_l(u)\BH_l\Big\|_{\CX_{\CR,q}(\dot\Om\cap B_j^i)}^q\intd u
\end{equation}
for any $\si_1>0$ and for any $\la_1\geq\la_0$.  
Define $\BK^0(\la)\BH$ as
\begin{equation*}
	\BK^0(\la)\BH = \sum_{i=1}^5\sum_{j=1}^\infty(\rho_j^i)^{-1}(\nabla\zeta_j^i) K_j^i(\CS_j^i(\la)\BH) \quad \text{for $\BH\in\CX_{\CR,q}(\dot\Om)$.}
\end{equation*}
In the same manner as we have obtained \eqref{160207_12} from \eqref{160201_10},
we can prove, by \eqref{160210_13}, the following properties: 
%
%
%
%
\begin{align}\label{fin:3}
	&\BK^0(\la)\in\Hol(\Si_{\ep,\la_1},\CL(\CX_{\CR,q}(\dot\Om),L_q(\dot\Om)^N)), \quad 
	\BK^0(\la)F_{\CR,\la}(\Bf,\Bh,\Bk) = \sum_{i=1}^5\sum_{j=1}^\infty(\rho_j^i)^{-1}(\nabla\zeta_j^i)K_j^i(\Bu_j^i), \\
	&\CR_{\CL(\CX_{\CR,q}(\dot\Om),L_q(\dot\Om)^N)}
		\Big(\Big\{\Big(\la\frac{d}{d\la}\Big)^l \BK^0(\la) \mid \la\in\Si_{\ep,\la_1}\Big\}\Big)
		\leq \ga_4(\si_1+\ga_{\si_1}\la_1^{-1/2}) \quad (l=0,1) \nonumber
\end{align}
for any $\si_1>0$ and for any $\la_1\geq\la_0$.

Finally we consider the term: $(\rho_j^i)^{-1}\nabla(K(\zeta_j^i\Bu_j^i)-\zeta_j^i K_j^i(\Bu_j^i))$.
We define a function $\mathfrak{g}_j^i, \mathfrak{h}_j^i\in W_q^1(\dot\Om)$ by
\begin{align*}
	(\mathfrak{g}_j^1,\mathfrak{h}_j^1) &= (<\nu_j^1(\CD(\nabla\zeta_j^1)\Bu_j^1)\Bn_j^1,\Bn_j^1>-\CE(\nabla\zeta_j^1)\Bu_j^1,0), \\
	(\mathfrak{g}_j^2,\mathfrak{h}_j^2) &= (0,<\nu_j^2(\CD(\nabla\zeta_j^2)\Bu_j^2)\Bn_j^2,\Bn_j^2>-\CE(\nabla\zeta_j^2)\Bu_j^2), \quad
	(\mathfrak{g}_j^i,\mathfrak{h}_j^i) = (0,0) \quad (i=3,4,5).
\end{align*}
Here and hereafter, for the sake of simplicity, we write $\BD(\ph\Bu)-\ph\BD(\Bu)=\CD(\nabla\ph)\Bu$
and $\di(\ph\Bu)-\ph\di\Bu = \CE(\nabla\ph)\Bu$, which satisfy
	$\|\CD(\nabla\ph)\|_{W_\infty^1(\ws)}\leq C\|\nabla \ph\|_{W_\infty^1(\ws)}$, 
	$\|\CE(\nabla\ph)\|_{W_\infty^1(\ws)}\leq C\|\nabla\ph\|_{W_\infty^1(\ws)}$.
Then,
\begin{equation*}
	\mathfrak{g}_j^1 = K(\zeta_j^1\Bu_j^1)-\zeta_j^1 K_j^1(\Bu_j^1) \quad \text{on $\Ga_j^1$}, \quad
	\mathfrak{h}_j^2 = K(\zeta_j^2\Bu_j^2)-\zeta_j^2 K_j^2(\Bu_j^2) \quad \text{on $\Ga_j^2$}.
\end{equation*}
In addition, for any $\ph\in\CW_{q'}^1(\Om)$, we have
\begin{align*}
	&((\rho_j^i)^{-1}\nabla(K(\zeta_j^i\Bu_j^i)-\zeta_j^iK_j^i(\Bu_j^i)),\nabla\ph)_{\dot\Om} 
		= ((\rho_j^i)^{-1}\Di(\nu_j^i \BD(\zeta_j^i\Bu_j^i))-\nabla\di(\zeta_j^i\Bu_j^i),\nabla\ph)_{\dot\Om} \\
		&\quad-((\rho_j^i)^{-1}(\nabla\zeta_j^i)K_j^i(\Bu_j^i),\nabla\ph)_{\dot\Om} - ((\rho_j^i)^{-1}\zeta_j^i\nabla K_j^i(\Bu_j^i),\nabla(\ph-c_j^i(\ph)))_{\dot\Om} \\
		&= ((\rho_j^i)^{-1}\Di(\nu_j^i \BD(\zeta_j^i\Bu_j^i))-\nabla\di(\zeta_j^i\Bu_j^i),\nabla\ph)_{\CH_j^i}
			-((\rho_j^i)^{-1}(\nabla\zeta_j^i)K_j^i(\Bu_j^i),\nabla\ph)_{\CH_j^i} \\
			&\quad+((\rho_j^i)^{-1}\nabla K_j^i(\Bu_j^i),(\nabla\zeta_j^i)(\ph-c_j^i(\ph)))_{\CH_j^i}
			-((\rho_j^i)^{-1}\nabla K_j^i(\Bu_j^i),\nabla\{\zeta_j^i(\ph-c_j^i(\ph))\})_{\CH_j^i} \\
		&= ((\rho_j^i)^{-1}\Di(\nu_j^i \BD(\zeta_j^i\Bu_j^i))-\nabla\di(\zeta_j^i\Bu_j^i),\nabla\ph)_{\CH_j^i}
			-((\rho_j^i)^{-1}(\nabla\zeta_j^i)K_j^i(\Bu_j^i),\nabla\ph)_{\CH_j^i} \\
			&\quad +((\rho_j^i)^{-1}\nabla K_j^i(\Bu_j^i),(\nabla\zeta_j^i)(\ph-c_j^i(\ph)))_{\CH_j^i}
			-((\rho_j^i)^{-1}\Di(\nu_j^i(\BD(\Bu_j^i)))-\nabla\di\Bu_j^i,\nabla\{\zeta_j^i(\ph-c_j^i(\ph))\})_{\CH_j^i},
\end{align*}
where $c_j^i(\ph)$ are constants given in Lemma \ref{lemm:Po} for $i=1,3,4,5$ and $c_j^2(\ph)=0$.
Let $\CW_q^{-1}(\Om)$ be the dual space of $\CW_{q'}^1(\Om)$ and
$<\cdot,\cdot>_\Om$ denote the duality pairing between $\CW_q^{-1}(\Om)$ and $\CW_{q'}^1(\Om)$.
Thus, if we define $I_j^i\in \CW_q^{-1}(\Om)$ by
\begin{align*}
	&<I_j^i,\ph>_\Om
		=((\rho_j^i)^{-1}\CC_1(\nu_j^i,\zeta_j^i)\nabla\Bu_j^i+\CC_0(\nu_j^i,\zeta_j^i)\Bu_j^i,\nabla\ph)_{\CH_j^i}
			-(\nabla\{(\nabla\zeta_j^i)\cdot\Bu_j^i\}+(\nabla\zeta_j^i)\di\Bu_j^i,\nabla\ph)_{\CH_j^i} \\
			&-2((\rho_j^i)^{-1}(\nabla\zeta_j^i)K_j^i(\Bu_j^i),\nabla\ph)_{\CH_j^i} 
			 -((\rho_j^i)^{-1}(\De\zeta_j^i)K_j^i(\Bu_j^i),\ph-c_j^i(\ph))_{\CH_j^i} \\
			&+((\rho_j^i)^{-1}\nu_j^i\BD(\Bu_j^i),\nabla\{(\nabla\zeta_j^i)(\ph-c_j^i(\ph))\})_{\CH_j^i}
			-(\di\Bu_j^i,\di\{(\nabla\zeta_j^i)(\ph-c_j^i(\ph))\})_{\CH_j^i} + [\CB_j^i(\Bu_j^i),\ph],
\end{align*}
where we have set $[\CB_j^i(\Bu_j^i),\ph] = 0$ for $i=2,4,5$ and
\begin{align*}
	&[\CB_j^1(\Bu_j^1),\ph]   = (\jump{(\rho_j^1)^{-1}K_j^1(\Bu_j^1)}\Bn_j^1,(\nabla\zeta_j^1)(\ph-c_j^1(\ph)))_{\Ga_j^1} \\
		&\quad - (\jump{(\rho_j^1)^{-1}\nu_j^1\BD(\Bu_j^1)}\Bn_j^1-\Bn_j^1\jump{\di\Bu_j^1},(\nabla\zeta_j^1)(\ph-c_j^1(\ph)))_{\Ga_j^1}, \\
	&[\CB_j^3(\Bu_j^3),\ph] = ((\rho_j^3)^{-1}K_j^3(\Bu_j^3)\Bn_j^3,(\nabla\zeta_j^3)(\ph-c_j^3(\ph)))_{\Ga_j^3} \\
		&\quad - ((\rho_j^3)^{-1}\nu_j^3\BD(\Bu_j^3)\Bn_j^3-\Bn_j^3\di\Bu_j^3,(\nabla\zeta_j^3)(\ph-c_j^3(\ph)))_{\Ga_j^3}, 
\end{align*}
then we have, for $i=1,\dots,5$,
\begin{equation*}
	((\rho_j^i)^{-1}\nabla(K(\zeta_j^i\Bu_j^i)-\zeta_j^i K(\Bu_j^i)),\nabla\ph)_{\dot\Om} = <I_j^i,\ph>_{\Om} \quad \text{for any $\ph\in\CW_{q'}^1(\Om)$}.
\end{equation*}
Let $\BF$ be an element of $\CL(\CW_q^{-1}(\Om),L_q(\Om)^N)$ such that, for $\te\in \CW_q^{-1}(\Om)$,
\begin{equation*}
	<\te,\ph>_\Om = (\BF(\te),\nabla\ph)_{\Om} \quad \text{for all $\ph\in \CW_{q'}^1(\Om)$,} \quad
	\|\BF(\te)\|_{L_q(\Om)} = \|\te\|_{\CW_{q}^{-1}(\Om)}.
\end{equation*}
Such a $\BF$ can be constructed by the Hahn-Banach theorem.
Since it holds that $((\rho_j^i)^{-1}\nabla(K(\zeta_j^i\Bu_j^i)-\zeta_j^i K(\Bu_j^i)),\nabla\ph)_{\dot\Om} = (\BF(I_j^i),\nabla\ph)_{\dot\Om}$,
we see that
$\nabla(K(\zeta_j^i\Bu_j^i)-\zeta_j^i K_j^i(\Bu_j^i))$ is given by the following formula:
\begin{equation}\label{bulk:1}
	\nabla(K(\zeta_j^i\Bu_j^i)-\zeta_j^i K_j^i(\Bu_j^i)) = \nabla\CK(\BF(I_j^i),\jump{\mathfrak{g}_j^i},\mathfrak{h}_j^i|_{\Ga_+}).
\end{equation}

To see the $\CR$-boundedness, for $i=1,\dots,5$, we define operators $\CI_j^i(\la)$ by
\begin{align*}
	&<\CI_j^i(\la)\BH,\ph>_\Om = 
		((\rho_j^i)^{-1}\CC_1(\nu_j^i,\zeta_j^i)\nabla(\CS_j^i(\la)\BH)+\CC_0(\nu_j^i,\zeta_j^i)\CS_j^i(\la)\BH,\nabla\ph)_{\CH_j^i}  \\
		&-(\nabla\{(\nabla\zeta_j^i)\cdot\CS_j^i(\la)\BH\}+(\nabla\zeta_j^i)\di(\CS_j^i(\la)\BH),\nabla\ph)_{\CH_j^i}
		-2((\rho_j^i)^{-1}(\nabla\zeta_j^i)K_j^i(\CS_j^i(\la)\BH),\nabla\ph)_{\CH_j^i} \nonumber \\
		&-((\rho_j^i)^{-1}(\De\zeta_j^i)K_j^i(\CS_j^i(\la)\BH),\ph-c_j^i(\ph))_{\CH_j^i}
		+((\rho_j^i)^{-1}\nu_j^i\BD(\CS_j^i(\la)\BH),\nabla\{(\nabla\zeta_j^i)(\ph-c_j^i(\ph))\})_{\CH_j^i} \nonumber \\
		&-(\di(\CS_j^i(\la)\BH),\di\{(\nabla\zeta_j^i)(\ph-c_j^i(\ph))\})_{\CH_j^i} + [\CB_j^i(\CS_j^i(\la)\BH),\ph] \nonumber
\end{align*}
for any $\BH\in\CX_{\CR,q}(\dot\Om)$ and for any $\ph\in\CW_{q'}^1(\Om)$.
In addition, we  define operators $\CJ_j^i(\la)$ by
\begin{equation*}
	\CJ_j^i(\la)\BH = <\nu_j^i(\CD(\nabla\zeta_j^i)(\CS_j^i(\la)\BH))\Bn_j^i,\Bn_j^i> - \CE(\nabla\zeta_j^i)(\CS_j^i(\la)\BH) \quad (i=1,2).
\end{equation*}
By Lemma \ref{lemm:Po} and \eqref{trace}, we have 
\begin{align*}
	&|[\CB_j^1(\CS_j^1(\la)\BH),\ph]|
		\leq \ga_4\Big(\|K_j^1(\CS_j^1(\la)\BH)\|_{L_q(\CH_{+j}^1)}^{1-1/q}\|\nabla K_j^1(\CS_j^1(\la)\BH)\|_{L_q(\CH_{+j}^1)}^{1/q} \\
		&+\|K_j^1(\CS_j^1(\la)\BH)\|_{L_q(\CH_{-j}^1)}^{1-1/q}\|\nabla K_j^1(\CS_j^1(\la)\BH)\|_{L_q(\CH_{-j}^1)}^{1/q}
		+ \|\nabla\CS_j^1(\la)\BH\|_{L_q(\CH_{+j}^1)}^{1-1/q}\|\nabla^2\CS_j^1(\la)\BH\|_{L_q(\CH_{+j}^1)}^{1/q} \\
		&+ \|\nabla\CS_j^1(\la)\BH\|_{L_q(\CH_{-j}^1)}^{1-1/q}\|\nabla^2\CS_j^1(\la)\BH\|_{L_q(\CH_{-j}^1)}^{1/q}\Big)\|\nabla\ph\|_{L_{q'}(\Om\cap B_j^1)},
\end{align*}
which, combined with Young's inequality and \eqref{160114_11}, furnishes that
\begin{equation*}
	|[\CB_j^1(\CS_j^1(\la)\BH),\ph]|
		\leq \ga_4\Big((\si_2+\si_1\ga_{\si_2})\|\nabla^2\CS_j^1(\la)\BH\|_{L_q(\CH_j^1)} + \ga_{\si_1}\ga_{\si_2}\|\nabla\CS_j^1(\la)\BH\|_{L_q(\CH_j^1)}\Big)
		\|\nabla\ph\|_{L_{q'}(\Om\cap B_j^1)}
\end{equation*}
for any $\si_1,\si_2>0$. 
Similarly to the last inequality, we can estimate $[\CB_j^3(\CS_j^3(\la)\BH),\ph]$.
Since $[\CB_j^i(\Bu_j^i),\ph]$ are linear with respect to $\Bu_j^i$, the inequalities of $[\CB_j^i(\CS_j^i(\la)\BH),\ph]$ $(i=1,3)$ above yields that 
\begin{align}\label{160213_4}
	&\Big|<\sum_{l=1}^n a_l \CI_j^i(\la_l)\BH_l,\ph>_\Om\Big| \\
		&\leq
		\ga_4\Big\{(\si_2+\si_1\ga_{\si_2})\Big\|\sum_{l=1}^n a_l\nabla^2\CS_j^i(\la_l)\BH_l\Big\|_{L_q(\CH_j^i)}
		+ \ga_{\si_1}\ga_{\si_2}\Big\|\sum_{l=1}^n a_l \CS_j^i(\la_l)\BH_l\Big\|_{W_q^1(\CH_j^i)}\Big\}
		\|\nabla\ph\|_{L_{q'}(\Om\cap B_j^i)} \nonumber
\end{align}
with $i=1,\dots,5$ for any $\ph\in\CW_{q'}^1(\Om)$
and for any $n\in\BN$, $\{a_l\}_{l=1}^n\subset\BC$, $\{\la_l\}_{l=1}^n \subset\Si_{\ep,\la_1}$, and $\{\BH_l\}_{l=1}^n\subset\CX_{\CR,q}(\dot\Om)$.
The estimate \eqref{160213_4} with $n=1$, together with \eqref{defi:S_j^i} and \eqref{op:S}, shows that
\begin{equation*}
	|< \CI_j^i(\la)\BH,\ph>_\Om| \leq M\|\BH\|_{\CX_{\CR,q}(\dot\Om\cap B_j^i)}\|\nabla\ph\|_{L_{q'}(\Om\cap B_j^i)}
	\quad \text{for all $\ph\in\CW_{q'}^1(\Om)$}
\end{equation*}
for any $\la\in\Si_{\ep,\la_1}$ and $\BH\in\CX_{\CR,q}(\dot\Om)$ with some positive constant $M$ independent of $j\in\BN$,
which, combined with Lemma \ref{prepa:1}, furnishes that 
%
%
the infinite sum $\BI^i(\la)\BH=\sum_{j=1}^\infty \CI_j^i(\la)\BH$ exists in the strong topology of $\CW_q^{-1}(\Om)$.
In addition, by \eqref{160213_4} with H\"older's inequality and by Lemma \ref{prepa:1} again, we have
\begin{align*}
	&\Big\|\sum_{l=1}^n a_l \BI^i(\la_l)\BH_l\Big\|_{\CW_q^{-1}(\Om)}^q \\
		&\leq 2^q(\ga_4)^q\Big\{(\si_2+\si_1\ga_{\si_2})^q\sum_{j=1}^\infty\Big\|\sum_{l=1}^n a_l \nabla^2\CS_j^i(\la_l)\BH_l\Big\|_{L_q(\CH_j^i)}^q
		+ (\ga_{\si_1}\ga_{\si_2})^q\sum_{j=1}^\infty\Big\|\sum_{l=1}^n a_l \CS_j^i(\la_l)\BH_l\Big\|_{W_q^1(\CH_j^i)}^q\Big\}.
\end{align*}
This inequality combined with monotone convergence theorem, Proposition \ref{lemm:multi},
and \eqref{150701_1}, together with the formulas \eqref{op:S}, yields that,
by Definition \ref{defi:R} and \eqref{finite},
\begin{align*}
	\int_0^1 \Big\|\sum_{l=1}^n r_l(u) \BI^i(\la_l)\BH_l\Big\|_{\CW_q^{-1}(\Om)}^q\intd u
	\leq \ga_4\Big((\si_2+\si_1\ga_{\si_2})^q	+ (\ga_{\si_1}\ga_{\si_2}\la_1^{-1/2})^q\Big)
	\int_0^1\Big\|\sum_{l=1}^n r_l(u) \BH_l\Big\|_{\CX_{\CR,q}(\dot\Om)}^q\intd u
\end{align*}
for any $\si_1,\si_2>0$ and $\la_1\geq\la_0$.
Thus, setting $\BI(\la)\BH = \sum_{i=1}^5\BI^i(\la)\BH$ and using Proposition \ref{prop:R}, we have
\begin{align}\label{fin:4}
	&\BI(\la)\in\Hol(\Si_{\ep,\la_1},\CL(\CX_{\CR,q}(\dot\Om),\CW_q^{-1}(\Om))), \quad \BI(\la)F_{\CR,\la}(\Bf,\Bh,\Bk)=\sum_{i=1}^5\sum_{j=1}^\infty I_j^i , \\
	&\CR_{\CL(\CX_{\CR,q}(\dot\Om),\CW_q^{-1}(\Om))}
		\Big(\Big\{\Big(\la\frac{d}{d\la}\Big)^l \BI(\la):\la\in\Si_{\ep,\la_1}\Big\}\Big)
		\leq \ga_4 (\si_2+\si_1\ga_{\si_2}+ \ga_{\si_1}\ga_{\si_2}\la_1^{-1/2}) \quad (l=0,1). \nonumber
\end{align} 
Analogously, we can prove the existence of operator families $\BJ^1(\la)$, $\BJ^2(\la)\in\Hol(\Si_{\ep,\la_1}, \CL(\CX_{\CR,q}(\dot\Om)),W_q^1(\dot\Om))$
such that
\begin{align}\label{fin:5}
	&\BJ^1(\la)F_{\CR,\la}(\Bf,\Bh,\Bk) = \sum_{j=1}^\infty\mathfrak{g}_j^1, \quad
	\BJ^2(\la)F_{\CR,\la}(\Bf,\Bh,\Bk) = \sum_{j=1}^\infty \mathfrak{h}_j^2, \\
	&\CR_{\CL(\CX_{\CR,q}(\dot\Om),\wt\CX_{\CR,q}(\dot\Om))}
		\Big(\Big\{\Big(\la\frac{d}{d\la}\Big)^l\big(\wt F_{\CR,\la}\BJ^i(\la)\big)\mid \la\in\Si_{\ep,\la_1}\Big\}\Big) 
		\leq \ga_4 (\si_2+\si_1\ga_{\si_2}+ \ga_{\si_1}\ga_{\si_2}\la_1^{-1/2})  \nonumber 
\end{align}
with $i=1,2$ and $l=0,1$ for any $\si_1,\si_1>0$ and for any $\la_1\geq\la_0$.

In view of \eqref{bulk:1},
we define $\BL^0(\la)\BH$ as $\BL^0(\la)\BH = \nabla\CK(\BF(\BI(\la)\BH),\jump{\BJ^1(\la)\BH},\BJ^2(\la)\BH|_{\Ga_+})$
for $\BH\in\CX_{\CR,q}(\dot\Om)$.
Then, by the continuity of $\CK$, \eqref{fin:4}, \eqref{fin:5}, and Proposition \ref{prop:R}, we see that
\begin{align}\label{fin:6}
	&\BL^0(\la)\in\Hol(\Si_{\ep,\la_1},\CL(\CX_{\CR,q}(\dot\Om),L_q(\dot\Om)^N)), \quad
	\BL^0(\la)F_{\CR,\la}(\Bf,\Bh,\Bk) = \sum_{i=1}^5\sum_{j=1}^\infty \nabla\left(K(\zeta_j^i\Bu_j^i)-\zeta_j^i K_j^i(\Bu_j^i)\right), \\
	&\CR_{\CL(\CX_{\CR,q}(\dot\Om),L_q(\dot\Om)^N)}
		\Big(\Big\{\Big(\la\frac{d}{d\la}\Big)^l\BL^0(\la) \mid \la\in\Si_{\ep,\la_1}\Big\}\Big)
		\leq \ga_4 (\si_2+\si_1\ga_{\si_2}+ \ga_{\si_1}\ga_{\si_2}\la_1^{-1/2}) \quad (l=0,1) \nonumber
\end{align}
for any $\si_1,\si_2>0$ and $\la_1\geq\la_0$.
Summing up \eqref{fin:1}, \eqref{fin:2}. \eqref{fin:3}, \eqref{fin:5}, and \eqref{fin:6},
we define $\BU(\la)\BH$ as
\begin{equation*}
\BU(\la)\BH = (\BV^0(\la)\BH + \BK^0(\la)\BH+\BL^0(\la)\BH,\BV^1(\la)\BH + \BJ^1(\la)\BH,\BV^2(\la)\BH + \BJ^2(\la)\BH)
\quad \text{for $\BH\in\CX_{\CR,q}(\dot\Om)$,}
\end{equation*}
and then $\BU(\la)$ is the required operator in Lemma \ref{lemm:CV}.
This completes the proof of the lemma.
%
%
%
%
%
\end{proof}

\subsection{Proof of Theorem \ref{theo:redu1}}
In Lemma \ref{lemm:CV}, we choose $\si_1$, $\si_2$, and $\la_1$ in such a way that
$\ga_4\si_2<1/8$, $\ga_4\ga_{\si_2}\si_1<1/8$, and $\ga_4\ga_{\si_1}\ga_{\si_2}\la_1^{-1/2}<1/4$, successively,
and thus
\begin{equation*}
\CR_{\CL(\CX_{\CR,q}(\dot\Om))}\Big(\Big\{\Big(\la\frac{d}{d\la}\Big)^l F_{\CR,\la}\BU(\la) \mid \la \in\Si_{\ep,\la_1}\Big\}\Big)\leq \frac{1}{2}
\quad (l=0,1).
\end{equation*}
These inequalities imply that
\begin{equation*}
\CR_{\CL(\CX_{\CR,q}(\dot\Om))}\Big(\Big\{\Big(\la\frac{d}{d\la}\Big)^l (I-F_{\CR,\la}\BU(\la))^{-1} \mid \la \in\Si_{\ep,\la_1}\Big\}\Big)\leq 2
\quad (l=0,1).
\end{equation*}
Similarly to Section \ref{sec:bent}, setting $\BB(\la)=\BS(\la)(I-F_{\CR,\la}\BU(\la))^{-1}$ with \eqref{160207_12} yields that
$\Bu = \BB(\la)F_{\CR,\la}(\Bf,\Bh,\Bk)$ solves the problem \eqref{redu-eq:1} and $\BB(\la)$ satisfies \eqref{160211_1}.
%
%
%
%
%
The uniqueness of \eqref{redu-eq:1} follows from the solvablility of the weak elliptic transmission problem on $\CW_{q'}^1(\Om)$ for $\rho_\pm$
and the solvability of \eqref{redu-eq:1} for $q'$ in the same manner as in the proof of Theorem \ref{theo:main}.
This completes the proof of Theorem \ref{theo:redu1}.

\quad \\
\noindent{\it Acknowledgments.}
The authors gratefully acknowledge the many helpful suggestions of Professor Yoshihiro Shibata during the preparation of the paper.
This research was partly supported by JSPS Japanese-German Graduate Externship and
by unit ``Multiscale Analysis, Modeling and Simulation", Top Global University Project of Waseda University.



\begin{thebibliography}{10}

\bibitem{AF03}
R.A. Adams and J.J.F. Fournier.
\newblock {\em Sobolev spaces}, volume 140 of {\em Pure and {A}pplied
  {M}athematics}.
\newblock Elsevier/Academic Press, Amsterdam, 2nd edition, 2003.

\bibitem{Denisova90}
I.V. Denisova.
\newblock A priori estimates for the solution of the linear nonstationary
  problem connected with the motion of a drop in a liquid medium.
\newblock {\em Trudy Mat. Inst. Steklov.}, 188:3--21; English transl.: Proc.
  Steklov Inst. Math. (1991), no. 3, 1--24, 1990.

\bibitem{Denisova93}
I.V. Denisova.
\newblock Solvability in {H}\"older spaces of a linear problem on the motion of
  two fluids separated by a closed surface.
\newblock {\em Algebra i Analiz}, 5(4):122--148; English transl.: St.
  Petersburg Math. J. {\bf 5} (1994), no. 4, 765--787, 1993.

\bibitem{Denisova94}
I.V. Denisova.
\newblock Problem of the motion of two viscous incompressible fluids separated
  by a closed free interface.
\newblock {\em Acta Appl. Math.}, 37(1-2):31--40, 1994.

\bibitem{Denisova05}
I.V. Denisova.
\newblock On the problem of thermocapillary convection for two incompressible
  fluids separated by a closed interface.
\newblock In {\em Trends in partial differential equations of mathematical
  physics}, volume~61 of {\em Progr. Nonlinear Differential Equations Appl.},
  pages 45--64. Birkh\"auser, Basel, 2005.

\bibitem{Denisova07}
I.V. Denisova.
\newblock Global solvability of a problem on two fluid motion without surface
  tension.
\newblock {\em Zap. Nauchn. Sem. S.-Peterburg. Otdel. Mat. Inst. Steklov.
  (POMI)}, 348:19--39; English transl.: J. Math. Sci. (N. Y.) {\bf 152} (2008),
  no. 5, 625--637, 2007.

\bibitem{Denisova14}
I.V. Denisova.
\newblock Global $l_2$-solvability of a problem governing two-phase fluid
  motion without surface tension.
\newblock {\em Port. Math.}, 71(1):1--24, 2014.

\bibitem{DN08}
I.V. Denisova and {\v S}.~Ne{\v c}asov\'a.
\newblock Oberbeck-{B}oussinesq approximation for the motion of two
  incompressible fluids.
\newblock {\em Zap. Nauchn. Sem. S.-Peterburg. Otdel. Mat. Inst. Steklov.
  (POMI)}, 362:92--119; English transl.: J. Math. Sci. (N. Y.) {\bf 159}
  (2009), no. 4, 436--451, 2008.

\bibitem{DS91}
I.V. Denisova and V.A. Solonnikov.
\newblock Solvability in {H}\"older spaces of a model initial-boundary value
  problem generated by a problem on the motion of two fluids.
\newblock {\em Zap. Nauchn. Sem. Leningrad. Otdel. Mat. Inst. Steklov. (LOMI)},
  188:5--44; English transl.: J. Math. Sci. {\bf 70} (1994), no. 3, 1717--1746,
  1991.

\bibitem{DS95}
I.V. Denisova and V.A. Solonnikov.
\newblock Classical solvability of the problem of the motion of two viscous
  incompressible fluids.
\newblock {\em Algebra i Analiz}, 7(5):101--142; English transl.: St.
  Petersburg Math. J. {\bf 7} (1996), no. 5, 755--786, 1995.

\bibitem{DS11}
I.V. Denisova and V.A. Solonnikov.
\newblock Global solvability of the problem of the motion of two incompressible
  capillary fluids in a container.
\newblock {\em Zap. Nauchn. Sem. S.-Peterburg. Otdel. Mat. Inst. Steklov.
  (POMI)}, 397:20--52; English transl.: J. Math. Sci. (N. Y.) {\bf 185} (2012),
  no. 5, 668--686, 2011.

\bibitem{DHP03}
R.~Denk, M.~Hieber, and J.~Pr\"uss.
\newblock $\mathcal{R}$-boundedness, {F}ourier multipliers and problems of
  elliptic and parabolic type.
\newblock {\em Mem. Amer. Math. Soc.}, 166(788):viii+114 pp., 2003.

\bibitem{DiBenedetto02}
E.~DiBenedetto.
\newblock {\em Real analysis}.
\newblock Birkh\"auser Advanced Texts. Birkh\"auser Boston, Inc., Boston, 2002.

\bibitem{ES13}
Y.~Enomoto and Y.~Shibata.
\newblock On the $\mathcal{R}$-sectoriality and the initial boundary value
  problem for the viscous compressible fluid flow.
\newblock {\em Funkcial. Ekvac.}, 56(3):441--505, 2013.

\bibitem{HS1}
M.~Hieber and H.~Saito.
\newblock Strong solutions for two-phase free boundary problems for a class of
  non-{N}ewtonian fluids.
\newblock submitted.

\bibitem{KPW13}
M.~K\"ohne, J.~Pr\"uss, and M.~Wilke.
\newblock Qualitative behaviour of solutions for the two-phase navier-stokes
  equations with surface tension.
\newblock {\em Math. Ann.}, 2(356):737--792, 2013.

\bibitem{KS1}
T.~Kubo and Y.~Shibata.
\newblock On the evolution of compressible and incompressible viscous fluids
  with a sharp interface.
\newblock 2015.
\newblock preprint.

\bibitem{KW04}
P.C. Kunstmann and L.~Weis.
\newblock Maximal ${L}_{p}$-regularity for parabolic equations, {F}ourier
  multiplier theorems and ${H}^\infty$-functional calculus.
\newblock In {\em Functional Analytic Methods for Evolution Equations}, Lect.
  Notes in Math. {\bf 1855}, pages 65--311. Springer, Berlin, 2004.

\bibitem{Mikhlin65}
S.G. Mikhlin.
\newblock {\em Multidimensional {S}ingular {I}ntegrals and {I}ntegral
  {E}quations}.
\newblock Pure and Applied Mathematics Monograph. Pergamon Press, Oxford-New
  York-Paris, 1965.

\bibitem{PS10}
J.~Pr\"uss and G.~Simonett.
\newblock On the {R}ayleigh-{T}aylor instability for the two-phase
  {N}avier-{S}tokes equations.
\newblock {\em Indiana Univ. Math. J.}, 59(6):1853--1871, 2010.

\bibitem{PS10b}
J.~Pr\"uss and G.~Simonett.
\newblock On the two-phase {N}avier-{S}tokes equations with surface tension.
\newblock {\em Interfaces Free Bound.}, 12(3):311--345, 2010.

\bibitem{PS11}
J.~Pr\"uss and G.~Simonett.
\newblock Analytic solutions for the two-phase {N}avier-{S}tokes equations with
  surface tension and gravity.
\newblock In {\em Parabolic Problems}, volume~80 of {\em Progr. Nonlinear
  Differential Equations Appl.}, pages 507--540. Birkh\"{a}user/Springer Basel
  AG, Basel, 2011.

\bibitem{Saito15}
H.~Saito.
\newblock {\em Free boundary problems of the incmopressible {N}avier-{S}tokes
  equations in some unbounded domains}.
\newblock PhD thesis, Waseda {U}niversity, 2015.

\bibitem{Saito15b}
H.~Saito.
\newblock On the $\mathcal{R}$-boundedness of solution operator families of the
  generalized stokes resolvent problem in an infinite layer.
\newblock {\em Math. Methods Appl. Sci.}, 38(9):1888--1925, 2015.

\bibitem{Shibata13}
Y.~Shibata.
\newblock Generalized resolvent estimate of the {S}tokes equations with first
  order boundary condition in a general domain.
\newblock {\em J. Math. Fluid Mech.}, 15(1):1--40, 2013.

\bibitem{Shibata14}
Y.~Shibata.
\newblock On the $\mathcal{R}$-boundedness of solution operators for the
  {S}tokes equations with free boundary condition.
\newblock {\em Differential Integral Equations}, 27(3-4):313--368, 2014.

\bibitem{SS08}
Y.~Shibata and S.~Shimizu.
\newblock On the ${L}_p\text{-}{L}_q$ maximal regularity of the {N}eumann
  problem for the {S}tokes equations in a bounded domain.
\newblock {\em J. Reine Angew. Math.}, 615:157--209, 2008.

\bibitem{SS11b}
Y.~Shibata and S.~Shimizu.
\newblock Maximal ${L}_p\text{-}{L}_q$ regularity for the two-phase {S}tokes
  equations; model problems.
\newblock {\em J. Differential Equations}, 251(2):373--419, 2011.

\bibitem{SS12}
Y.~Shibata and S.~Shimizu.
\newblock On the maximal ${L}_p\text{-}{L}_q$ regularity of the {S}tokes
  problem with first order boundary condition; model problems.
\newblock {\em J. Math. Soc. Japan}, 64(2):561--626, 2012.

\bibitem{Tanaka93}
N.~Tanaka.
\newblock Global existence of two phase nonhomogeneous viscous incompressible
  fluid flow.
\newblock {\em Comm. Partial Differential Equations}, 18(1-2):41--81, 1993.

\bibitem{Weis01}
L.~Weis.
\newblock Operator-valued {F}ourier multiplier theorems and maximal
  ${L}_p$-regularity.
\newblock {\em Math. Ann.}, 319(4):735--758, 2001.

\end{thebibliography}


\end{document}